\documentclass[opre,nonblindrev]{informs3NoHeading} % current default for manuscript submission

\OneAndAHalfSpacedXI % current default line spacing
\usepackage{endnotes}
\let\footnote=\endnote
\let\enotesize=\normalsize
\def\notesname{Endnotes}%
\def\makeenmark{$^{\theenmark}$}
\def\enoteformat{\rightskip0pt\leftskip0pt\parindent=1.75em
	\leavevmode\llap{\theenmark.\enskip}}

\newcommand{\Ex}{\mathrm{E}}
\usepackage{mathtools}
\usepackage{diagbox}
\usepackage{mathrsfs}
\usepackage{bbm,tabularx}
\usepackage{amsmath,adjustbox,booktabs}
\usepackage{graphics,dsfont}
\usepackage{mathtools}
\DeclarePairedDelimiter\ceil{\lceil}{\rceil}

\usepackage{natbib}
\bibpunct[, ]{(}{)}{,}{a}{}{,}%
\def\bibfont{\small}%
\TheoremsNumberedThrough     % Preferred (Theorem 1, Lemma 1, Theorem 2)
\ECRepeatTheorems

\EquationsNumberedThrough    % Default: (1), (2), ...
\newcommand{\rr}[1]{\textcolor{black}{#1}}

\begin{document}
	%%%%%%%%%%%%%%%%
	
	%\setcounter{page}{1}
	%\DOI{10.1287/xxxx.0000.0000}%
	
	\RUNAUTHOR{Authors' names blinded}
	
	\RUNTITLE{Posted Price versus Auction Mechanisms}
	
	% \TITLE{Bundling Information Goods of Decreasing Value}
	\TITLE{Posted Price versus \rr{Hybrid} Mechanisms in Freight Transportation Marketplaces}
	
	\ARTICLEAUTHORS{%
            \AUTHOR{Ruoran Chen}
		\AFF{School of Economics and Management, Southwest Jiaotong University, Chengdu, China, \EMAIL{chenrr@swjtu.edu.cn}} %, \URL{}}
		\AUTHOR{Sungwoo Kim}
		\AFF{DoorDash, New York, NY,  \EMAIL{asdf1358@gmail.com}} %, \URL{}}
	\AUTHOR{He Wang}
	\AFF{School of Industrial and Systems Engineering, Georgia Institute of Technology, GA 30332, \EMAIL{he.wang@isye.gatech.edu}}
	\AUTHOR{Xuan Wang}
	\AFF{Department of Information Systems, Business Statistics and Operations Management, Hong Kong University of Science and Technology, Clear Water Bay, Kowloon, Hong Kong, \EMAIL{xuanwang@ust.hk}}
} % end of the block

\ABSTRACT{%
	We consider a freight platform that serves as an intermediary between shippers and carriers in a truckload transportation network. The platform's objective is to design a policy that determines prices for shippers and payments to carriers, as well as how carriers are matched to loads to be transported, to maximize its long-run average profit. \rr{We propose a two-stage decision framework to model carriers' load choice behavior, where carriers choose a lane according to the multinomial logit (MNL) model based on the platform's posted price in the first stage and book a load in the second stage. We analyze two types of carrier-side mechanisms commonly used by freight platforms: a posted price mechanism and a hybrid mechanism where carriers can either book loads at posted price or submit their bids in an auction}. The proposed mechanisms are constructed using a fluid approximation model to incorporate carrier interactions in the freight network. We show that \rr{the hybrid mechanism has higher profits than the posted price mechanism}. We prove tight bounds between these mechanisms for varying market sizes. The findings are validated through a numerical simulation using industry data from the U.S. freight market.

	% Enter your abstract
}%

%\KEYWORDS{deterministic inventory theory; infinite linear programming duality;
	%  existence of optimal policies; semi-Markov decision process; cyclic schedule}


\maketitle
%%%%%%%%%%%%%%%%%%%%%%%%%%%%%%%%%%%%%%%%%%%%%%%%%%%%%%%%%%%%%%%%%%%%%%

% Samples of sectioning (and labeling) in OPRE
% NOTE: (1) \section and \subsection do NOT end with a period
%       (2) \subsubsection and lower need end punctuation
%       (3) capitalization is as shown (title style).
%
%\section{Introduction.}\label{intro} %%1.
%\subsection{Duality and the Classical EOQ Problem.}\label{class-EOQ} %% 1.1.
%\subsection{Outline.}\label{outline1} %% 1.2.
%\subsubsection{Cyclic Schedules for the General Deterministic SMDP.}
%  \label{cyclic-schedules} %% 1.2.1
%\section{Problem Description.}\label{problemdescription} %% 2.

% Text of your paper here

\maketitle
%%%%%%%%%%%%%%%%%%%%%%%%%%%%%%%%%%%%%%%%%%%%%%%%%%%%%%%%%%%%%%%%%%%%%%

% Samples of sectioning (and labeling) in OPRE
% NOTE: (1) \section and \subsection do NOT end with a period
%       (2) \subsubsection and lower need end punctuation
%       (3) capitalization is as shown (title style).
%
%\section{Introduction.}\label{intro} %%1.
%\subsection{Duality and the Classical EOQ Problem.}\label{class-EOQ} %% 1.1.
%\subsection{Outline.}\label{outline1} %% 1.2.
%\subsubsection{Cyclic Schedules for the General Deterministic SMDP.}
%  \label{cyclic-schedules} %% 1.2.1
%\section{Problem Description.}\label{problemdescription} %% 2.

\section{Introduction}

The trucking industry transports 73\% of freight by weight in the U.S.\ and generates a gross revenue of around \$987 billion annually \citep{ATA2023}. However, this large industry is characterized by high market fragmentation. 
According to the U.S. Department of Transportation, there are over 577,000 active motor carriers in the United States, of which 95.5\% operate 10 or fewer trucks \citep{ATA2023}.
Traditionally, the market relies on brokers and freight forwarders to connect shippers and carriers using phones or emails, a process that is time-consuming and labor-intensive.
%Recent technology helps automate this process and improve its efficiency. 
In recent years, digital freight marketplaces have grown rapidly. These digital platforms allow shippers and carriers to list and book loads through smartphone apps or websites in an automated process that is significantly more efficient than traditional brokers.

A central question faced by operators of digital freight platforms is how to design mechanisms for participants in their marketplaces. 
Unlike other types of transportation marketplaces such as ride-sharing, where matching decisions are made by the platform operator in a centralized manner, most freight platforms allow carriers to browse any open loads and choose which loads they want to transport. The decentralized matching scheme is a result of demand and supply heterogeneity in freight transportation, as carriers have their specific preferences for load features including cargo type, trailer type, length of haul, etc.
A variety of market mechanisms are currently being used in practice to match loads to carriers: 
Some freight platforms set posted prices for loads; some use auctions that allow carriers to bid their own prices for loads; others apply hybrid (or dual-channel) mechanisms that allow carriers to choose between posted price and bidding \citep[e.g.][]{uber2020, convoy2023}.
%For example, Convoy first introduced in-app bidding in 2016, and Uber Freight followed to add in-app bidding in 2020. 
%Some individual brokers offer auction mechanisms. 
%As various types of mechanisms are used in practice, 
For example, Figure \ref{Figure: Uber} illustrates the carrier app of a freight marketplace that uses such a hybrid mechanism. The platform displays information of available loads to carriers, such as the origin and destination of a load, distance, weight, pickup and drop-off time, as well as a posted price set by the platform (i.e., the expected payment to carriers). If a carrier wants to book a load, they can either book it instantly at the displayed price or submit a bid in order to receive possibly higher payments. However, booking a load through bidding requires the carrier to wait for the platform's assignment decision, during which the carrier may be outbid by others.
\begin{figure}[!htb]
	\centering
	\caption{An example of freight brokerage apps for carriers (source: Uber freight).    \label{Figure: Uber}}
	\vspace{12pt}
	\begin{minipage}{0.3\textwidth}
		\includegraphics[width=\textwidth]{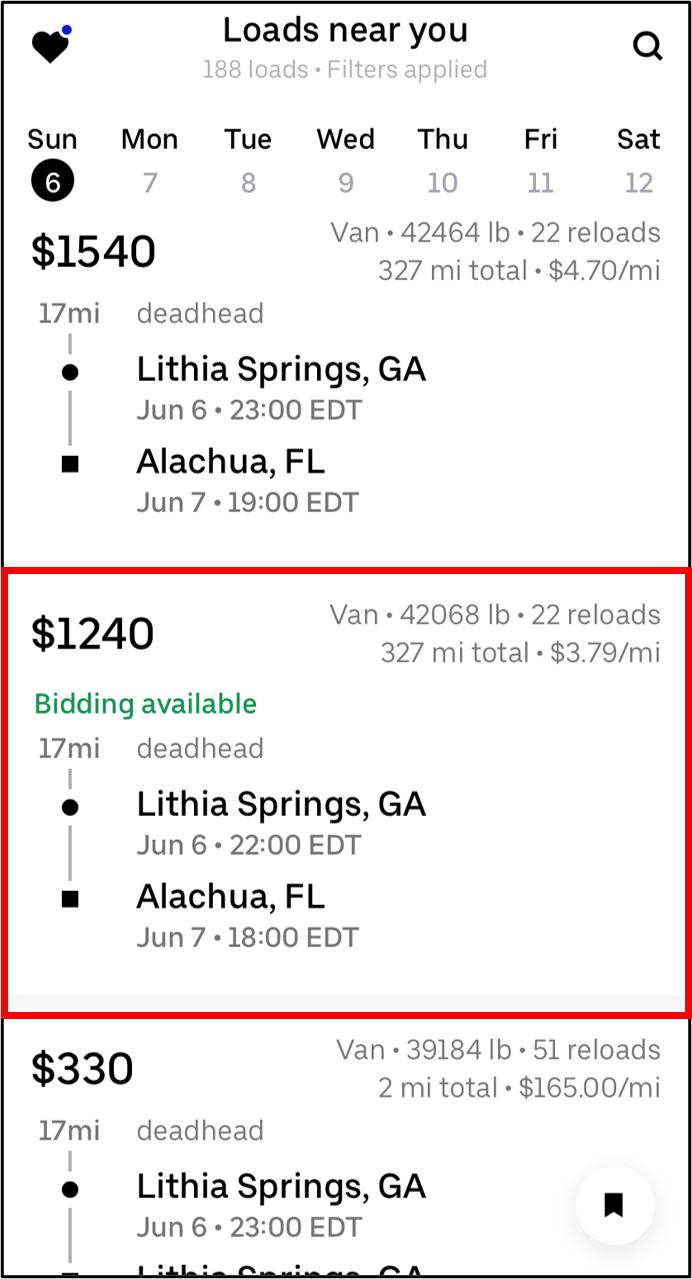}
	\end{minipage}
	\hspace{30pt}
	\begin{minipage}{0.3\textwidth}
		\includegraphics[width=\textwidth]{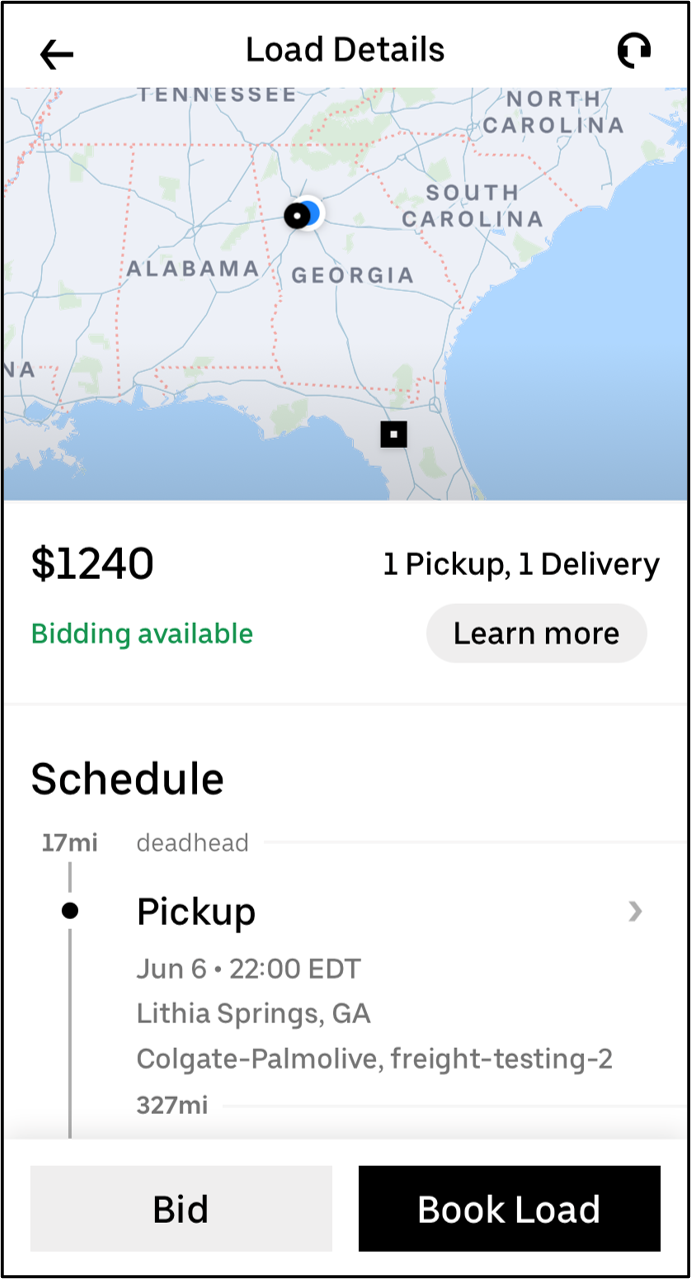}
	\end{minipage}
	\begin{minipage}{1\textwidth}
		\vspace{12pt}
		{\footnotesize \it Note: The left screen shows a list of open loads available to carriers for booking. The right screen shows a load being selected that displays two booking options: instant booking at the posted price (bottom right of the screen) or bidding (bottom left of the screen).}
	\end{minipage}
\end{figure}

The goal of this paper is to
analyze different mechanisms used by freight marketplaces and understand their implications on market efficiency and platform profitability. 
Moreover, we aim to understand the motivation for freight platforms to use hybrid mechanisms.
To formulate the mechanism design problem, we consider a stochastic model for a two-sided freight marketplace. The market operator observes the movement of loads and carriers in an interconnected freight transportation network.
We assume that the behaviors of carriers are characterized by a random utility model, and carriers' private information is their
true opportunity costs to transport a given load.
By analyzing the decisions of carriers, we consider \rr{different types of mechanisms mentioned earlier}: posted price \rr{and hybrid (dual-channel that combines both posted price and auction) mechanisms}.

\subsection{Contributions}
The main contributions of this paper are summarized as follows.

\begin{itemize}
	\item \textbf{Freight marketplace model:} We formulate a two-sided freight marketplace as a discrete-time infinite-horizon Markov decision process (MDP). The model is defined over a transportation network where shipper demand and carrier supply arrive dynamically. We assume that all loads are available for one period and unmatched loads expire at the end of the period.
	Meanwhile, carriers may stay in the marketplace for multiple periods. Carriers choose a load based on a two-stage decision framework. In the first stage, carriers choose a lane according to the multinomial logit (MNL) model based on the prices posted by the platform. Then in the second stage, carriers decide whether to book a load on the chosen lane instantly by accepting the posted price or submit a bid for a load on the chosen lane (if an auction option is provided by the platform).
	A fluid approximation model of the MDP is formulated by replacing random variables with their means. We show that the fluid approximation model can be reformulated as a convex optimization problem. In addition, the optimal objective value of the fluid model serves as an upper bound benchmark of \emph{any} stable, incentive compatible, and individually rational dynamic mechanisms. 
    
	\item \textbf{Mechanism design and theoretical analysis:} Using the fluid model, we construct several mechanisms and compare their relative performance in terms of long-run average profit.
	We consider a static posted price with fixed rates on each edge of the network.
	We also consider a hybrid mechanism that allows the carriers to self-select between a uniform price auction and the static posted price policy.
	We show that the expected profit of the hybrid mechanism is always higher than that of the static posted price mechanism.
	Both mechanisms are asymptotically optimal as both supply and demand scale up in the system. 
	\rr{Furthermore, we provide tight bounds between the posted price mechanism and the hybrid mechanism as a function of the scaling factor.}

	\item \textbf{Numerical experiment:} To evaluate the numerical performance of the proposed mechanisms, we conduct a simulation using real data obtained from the U.S.\ freight market. We consider a freight network consisting of the lower 48 states, where each state is treated as a node in the network. We simulated the performance of the posted price and hybrid mechanisms under different load volume assumptions. In general, the hybrid (dual-channel) mechanism has around 7-17\% lower carrier costs than the posted price mechanism.  Under the hybrid mechanism, over 87\% of carriers choose the posted price option, and this percentage is even higher as the load volume of the platform increases. This implies that most carriers
can confirm a booking instantly and do not have to wait for load confirmation under the hybrid
mechanism.	Finally, the hybrid mechanism has a higher market efficiency, represented by fewer unmatched loads.  
\end{itemize}

\subsection{Literature Review}

%\textcolor{blue}{\cite{chen2019efficacy} consider a customer behavior where customers strategically determine the time to purchase products. They showed that the static posted price mechanism is asymptotically optimal using the fluid approximation.} 

\textbf{Mechanism design.}
The seminal result of \cite{Myerson_1981} characterized the optimal revenue maximizing mechanism for single-parameter mechanism design problems. 
However, the optimal revenue maximizing mechanism for general multi-item or multi-period problems can be very complex and may not be practically appealing. A stream of literature has studied the performance of simple-to-implement mechanisms such as sequential posted-price mechanisms or first/second-price auctions \citep{edelman2007internet, chawla2010multi, dutting2016revenue, jin2020tight, balseiro2019dynamic}; most of these papers are based on applications such as e-commerce and online advertising. %\cite{ronen2001approximating} considered a single item auction problem and proposed a 2-approximation auction, denoted to 1-lookahead auction, compared to the optimal mechanism. The author generalized this auction to $k$-lookahead auction and showed that the approximation ratio for this auction is $(k+1)/k$. 
%\cite{jin2020tight} proved tight approximation ratios among various simple mechanisms for a single-item problem. 
%The tight approximation ratio between sequential posted pricing over anonymous pricing is 2.62, and the tight ratio of the second price auction with anonymous reserve to anonymous pricing is 1.64. The authors also showed the ratio of the Myerson auction to the second price auction with anonymous reserve is 2.15. 
%\cite{jin2019tight} analyzed approximation ratios for both single-item and multi-item cases. The authors showed that the approximation ratio of the Myerson auction to the anonymous pricing and the approximation ratio of the item pricing to the uniform pricing share a tight bound $2.62.$ 
%\cite{dutting2016revenue} considered a pricing problem for selling $k$ homogeneous goods to $n$ unit-demand buyers, and examined approximation ratios between static and dynamic pricing policies under various conditions.
%Their approximation ratio is a function of $k$ and $n$ except for a special case when $n$ goes to infinity and buyers' value distribution is irregular. 
Several papers have compared posted price and auction mechanisms \citep{Wang_1993, vulcano2002optimal, einav2018auctions}. 
%\cite{Wang_1993} considered a single-item selling problem, and showed that auctions are preferred over the posted price mechanism when the buyers' valuation distribution becomes more dispersed. 
%Our research also validates the result when the auctioning costs are not considered. If the posted price and the reserve price are the same, our auction gives a better profit. 
%\cite{einav2018auctions} analyzed the trade-off between posted price and auction mechanisms for an e-commerce marketplace and used historical data from eBay to validate their model. Our research validates their results and theoretically extends for multi-location and multi-period problems. % Although this research gives deep insights on the two different mechanisms, authors did not give theoretically important results. We focus more on theoretical analysis on the expected profits of mechanisms. In addition, thee simulation results of our mechanisms show that the trade-off between profit and waiting time, leading a similar result to the previous work. 
% We take a similar approach and consider simple mechanisms in this paper.
Our contribution is that we consider auction design in a multi-location network setting where the number of bidders in each location is endogenous and depends on previous decisions from other locations.

 %Therefore, the optimal mechanism for a single-period problem does not necessarily become the optimal mechanism for our problem. 
%The interaction among different locations and across different periods makes the analysis of our model more involved. 

There is a stream of literature considering \emph{hybrid} or \emph{dual-channel} mechanisms that offer both posted price and auction at the same time. % Although a hybrid mechanism is a new concept in the freight marketplaces, there are few studies related to this type of mechanism.
\cite{Etzion_etal_2006} studied a problem of optimally selecting and designing auction and posted price mechanisms in parallel. 
%They showed that the optimal design of the dual channel can significantly outperform a single posted-price channel. 
\cite{caldentey2007online} extended the work of \cite{Etzion_etal_2006} and considered dual channel problems with two variants, i.e., the external posted-price channel or the combined posted-price and auction channels. 
\cite{osadchiy2010selling, board2016revenue} considered selling goods to strategic buyers using a hybrid mechanism that consists of a dynamic pricing mechanism during the season and an auction for end-of-season clearing sales. %Our hybrid mechanism can be seen as a special case of the mechanism proposed in this study. However, the network model setting differentiated our paper from this literature.
\cite{cohen2022best} considered designing dual-channel mechanisms for online advertising. Although the idea of a hybrid mechanism is not new, our paper differs from this stream of literature by studying hybrid mechanisms in a network setting. In our model, the number of bidders for a lane in each period is influenced not only by the arrival, movement, and departure of carriers in the network but also by the platform's pricing decisions for different lanes in the transportation network. Additionally, carriers' choices of lanes, their bidding behavior, and potential arrivals and departures will also influence the platform's pricing decisions in the subsequent period. Therefore, the interplay between network dynamics, pricing, and auction design across multiple lanes over multiple periods poses technical challenges that necessitate the development of new methodologies and analysis. 

%\cite{celis2014buy} considered a buy-it-now or take-a-chance mechanism to sell a single item. 
%They showed that their proposed mechanism is both efficient and revenue-optimal in the two-type valuation case. 
%Compared with these two studies, a key challenge in analyzing hybrid mechanisms for freight marketplaces arises from the dependence of the ``book now" decisions on the arrival sequence of carriers (bidders), since loads are booked on a first-come-first-serve basis. 

%Moreover, carriers may take advantage of the real-time information of the remaining loads on the platform when submitting their bids, which poses additional challenges for the design of incentive compatible mechanisms.

\textbf{Pricing in two-sided platforms.} Our work is also related to the stream of work on two-sided platforms, especially those on ride-sharing. Several recent papers have examined the dynamic matching and pricing problem of two-sided platforms (e.g., \citealt{Banerjee2016}, \citealt{Cachon_etal_2017}, \citealt{Ozkan_Ward_2020}, \citealt{Chen_and_Hu_2020}, \citealt{varma2020dynamic}, \citealt{Banerjee_etal_2021}, \citealt{Feng_etal_2021}, \citealt{Hu_Zhou_2021}, \citealt{Balseiro_etal_2021}, \citealt{Aveklouris_etal_2021}). In view of this literature, our paper makes some modeling assumptions that reflect distinct features of freight platforms.
First, the majority of papers consider centralized matching control, which is common practice in ride-sharing; in other words, it is the platform operator who makes matching decisions for participants in the platform. In contrast, most freight platforms use decentralized matching schemes where carriers are free to choose any available loads, so we assume a setting with decentralized control in this paper. 
Second, the travel time of a long-haul trip in freight transportation is much longer (up to several days) than a ride-sharing trip,
so we assume an open network where some carriers leave the platform at the end of each period and new carriers may arrive, rather than using a closed network assumed in some of the ride-sharing papers.
Third, our paper studies auction design in hybrid mechanisms, which are commonly used in freight transportation but are less used in ride-sharing.

%\cite{Ozkan_Ward_2020} consider the dynamic matching problem in a ride-sharing network, where a passenger requesting a ride from a region with no available cars is willing to wait for a driver to relocate from a nearby region. They show that a fluid-based static (matching) mechanism achieves asymptotic optimality in the heavy traffic regime. 

% \citealt{Ozkan_Ward_2020} study the spatial matching of an arriving customer request with available drivers, and propose a fluid-based matching mechanism that is asymptotically optimal and significantly outperforms the commonly-used “closest driver” matching policy.

%\cite{Banerjee_etal_2021} develop approximation methods to determine prices in a ride-sharing network to maximize various long-run average performance metrics such as throughput and revenue. 

% \cite{Hu_Zhou_2021} consider a dynamic matching problem with heterogeneous demand and supply types in a periodic-review manner, and characterize conditions under which the optimal matching policy follows a priority hierarchy.

% \cite{Chen_and_Hu_2020} consider the dynamic pricing decisions of a ride-sharing platform in a strategic environment where customers and suppliers may wait strategically for better prices. They show that a waiting-adjusted fixed pricing heuristic is close to optimal when the market is thick with large transaction volume. 

Closer to our work are the ride-sharing papers that study pricing with spatial considerations. \cite{Bimpikis2019} consider the spatial transition of a ride-sharing network and characterize the value of spatial price discrimination. The authors show that a fixed-commission rate pricing policy could result in significant profit loss if demand is not “balanced” across locations. \citet{Guda_Subramanian_2019}
consider a two-location two-period setting and show the benefits of using surge pricing when the supply is strategic. \citet{Afeche2018} examine the impact of platforms' demand-side admission control and supply-side repositioning control on system performance in a two-location network. In related work, \citet{Besbes_etal_MS_2021} study the optimal spatial pricing strategy in a static equilibrium model and examine how the platform should respond to short-term supply-demand imbalances given that the supply units are strategic. 
%Compared with ride-sharing marketplaces where the platform operator can provide price incentives for driver
%empty reposition, empty truck reposition is rarely used by freight platform operators.

% \cite{Balseiro_etal_2021} study dynamic pricing of resources that relocate in large networks, where the number of resources scales at the same rate as the number of locations. They develop a Lagrangian-based policy that is asymptotically optimal for large hub-and-spoke networks. 

% (No spatial considerations:) \cite{Aveklouris_etal_2021} consider a matching problem in which the demand and supply are heterogeneous and impatient. The authors propose centralized matching policies that are asymptotically optimal, which balance the trade-off between making a less good match quickly and waiting for a potential better match but at the risk of losing impatient demand or supply. 

% \cite{Chawla_etal_2010} showed that sequential posted-price mechanisms can achieve a constant factor of the optimal revenue for both single- and multiple-parameter settings under various constraints on the allocations. 

% A new contribution of our results is to show that, if one aims at cost minimization rather than revenue maximization, a posted pricing policy can have arbitrarily bad performance in finite sized systems.

% \cite{Adamczyk_etal_2017}

\textbf{Freight marketplaces.}
\cite{andres2003framework} proposed a model for freight marketplaces using auction mechanisms and applied Vickery auctions in a simulation study.
%For simulations, a Vickery auction was used. 
In subsequent work, \cite{figliozzi2005impacts} compared three different sequential auctions numerically.
However, the two papers above assume that auctions are run by shippers rather than some freight brokerage platform.
\cite{caplice2007electronic} gave a survey of auctions in truckload transportation marketplaces.
\cite{topaloglu2007incorporating} proposed an optimization model for a carrier's fleet by integrating shipper pricing and load assignment decisions.
\cite{cao2020dynamic} considered a dynamic pricing problem for freight marketplaces that is closely related to this paper.
The authors proposed a static pricing policy by setting a fixed price for each lead time. There are several key differences between \cite{cao2020dynamic} and this paper.
Their work focused on posted price mechanisms, whereas our paper focuses on comparing posted price mechanisms with hybrid mechanisms. Moreover, \cite{cao2020dynamic} considered individual lanes and thus did not consider a network setting, whereas our paper explicitly models the interactions of carriers in the freight network. 

% To this end, establishing the fluid model as an upper bound benchmark is not as straightforward as in Cao et al. (2020), because we need to prove that the fluid model is an upper bound for all mechanisms.

%\citet{Besbes_etal_OR_2021} consider the problem of determining the optimal number of drivers in ride-hailing systems to balance service utilization and customer waiting times while accounting for pickup and travel times, under a heavy traffic regime.

%We refer the readers to two excellent reviews by \citet{Benjaafar_Hu_2020} and \citet{HuMing_POM_Review} that connect classical operations management models and theory with the new applications of sharing economy and online marketplaces.

\section{Model}  \label{section: model}
We formulate the marketplace design problem for a freight transportation platform that serves both shippers (the demand side) and carriers (the supply side). On this platform, shippers post information about goods that need to be transported, commonly referred to as \emph{loads} in the freight industry. Carriers book loads on the platform and transport them.
Our objective is to 
design a market mechanism that specifies how the platform should set prices for shippers and determine allocation and payment rules for carriers in order to maximize the expected profits of the platform.

The model is defined in a transportation network represented by a graph ${\mathcal{G(N,E)}}$. The set of nodes $\mathcal{N} := \{1,...,n\}$ represents geographical locations, and the set of arcs $\mathcal{E} = \{(i,j): i, j \in \mathcal{N}\}$ contains all the origin-destination (O-D) pairs, also known as \emph{lanes} in truckload transportation. 
We allow self-loops in the network, i.e., $(i,i) \in \mathcal{E}$ for all $i\in \mathcal{N}$, which represent local lanes.
Let $\delta^+(j) := \{k \in \mathcal{N}:(j,k)\in \mathcal{E}\}$ be the outbound nodes from node $j$ and $\delta^-(j) := \{i \in \mathcal{N}: (i,j)\in \mathcal{E}\}$ be the inbound nodes to node $j$.
For simplicity, we assume all loads require exactly one period to transport on any O-D pair. \rr{It might be possible to generalize the single period travel time assumption using the approach in \cite{godfrey2002adaptive}, but this would complicate our model significantly and may necessitate the development of new methodologies and the use of different tools, so we shall leave it for future research. Our numerical experiments conducted in Section~\ref{section: case study} relax the single period travel time assumption, and we demonstrate the effectiveness of our proposed mechanisms under a more general setting in which the travel times of certain O-D pairs require multiple time periods.}

We model the dynamics in the marketplace as a discrete-time, infinite horizon, average cost Markov decision process (MDP).
The objective is to maximize its long-run average profit.
The system state of the MDP consists of both the state of loads and the state of carriers.
\rr{On the carrier side, let $\mathcal{S}_{it}$ denote the set of carriers available to deliver loads at node $i$ at the beginning of period $t$ (including the new carriers who have just arrived to the marketplace), and let $S_{it} := |\mathcal{S}_{it}|$ be its cardinality.} 
On the demand side, let $D_{ijt}$ denote the number of loads in the marketplace that need to be transported from node $i$ to node $j$ at the beginning of period $t$, where $D_{ijt}$ follows a Poisson distribution with the rate $d_{ijt}$.\footnotemark\footnotetext{The Poisson distribution assumption can be replaced with other distributions without loss of generality as long as the variance of $D_{ijt}$ increases at most linearly with the mean. Specifically, when the mean demand grows from $d_{ijt}$ (in the original $\textsf{MDP}$) to $\theta d_{ijt}$ (in instance $\textsf{MDP}^{\theta}$ with scaling parameter $\theta$, which is formally defined in Section \ref{subsection: Fluid Model}), the variance of the demand should grow from $\sigma_{ijt}^2$ to no more than $C \theta \sigma_{ijt}^2$ for some constant $C$. When $D_{ijt}$ follows a Poisson distribution with rate $d_{ijt}$,  the mean demand rate and the variance are both $\theta d_{ijt}$ in $\textsf{MDP}^{\theta}$, which clearly satisfies the above distributional property.} The vector $(\mathbf{S}_t,\mathbf{D}_t)$ is the system state of the MDP in period $t$, where \rr{$\mathbf{S}_t=(S_{it},\forall i\in \mathcal{N})$} and $\mathbf{D}_t=(D_{ijt},\forall (i,j)\in \mathcal{E})$. 

In each period of the MDP, the following events occur in sequence on the \emph{shipper side} of the marketplace:
(1) upon observing the current system state $(\mathbf{S}_{t},\mathbf{D}_{t})$, the platform determines shipper spot prices for the next period $(t+1)$; 
(2) with the announced shipper prices, shippers submit load demand to the platform, which determines $\mathbf{D}_{t+1}$; 
(3) all the loads submitted by shippers during period $t$ are processed and will be released to the carriers for booking at the beginning of the next period $t+1$.

Meanwhile, the following sequence of events occur on the \emph{carrier side} of the marketplace: (1) at the beginning of period $t$, upon observing the current system state $(\mathbf{S}_{t},\mathbf{D}_{t})$, the platform sets prices and/or auction parameters; 
(2) during the period, carriers browse the freight brokerage app to decide which load to book or bid, or leave the app without any booking;
(3) bookings through posted pricing are confirmed immediately, and bookings through auctions are resolved at the end of the period;
(4) after the end of each period, carriers deliver loads, and the platform pays a penalty cost $b_{ij}$ for each unmatched load on lane $(i,j)$; 
(5) carriers who delivered loads on $(i,j)$ remain in the marketplace with probability $q_{ij}$, and become available \rr{at the beginning of the next period at their new location $j$ together with $\Lambda_{j t+1}$ number of new carriers who exogenously arrive to node $j$ at the beginning of period $t+1$}.

\begin{remark}
	To further explain the penalty term $b_{ij}$, note that when a freight platform is not able to find a carrier for a load in practice, a few possible outcomes may occur. First, the shipper may cancel the load and use another brokerage platform. Because our model accounts for revenues from shippers when they tender loads to the platform, if the load is canceled, then the shipper will be refunded and the penalty term can be set equal to the shipper price. The penalty term can also be adjusted to include any additional compensation that the platform pays to the shipper. Second, the platform may choose to hire a third-party logistics company to transport the load. In this case, the penalty term represents the cost of hiring the third-party company. Finally, it is also possible that the shipper may agree with the platform to reschedule the load for a future date, in which case the penalty term can be set as the current shipper rate and the shipper revenue will be accounted for in the future period to reflect the changes in future shipper rates. 
\end{remark}

\subsection{The Platform's Mechanism} \label{subsection: mechanism}

In this section, we introduce the platform's mechanism $(\mathcal{M}_r,\mathcal{M}_p)$, where the shipper-side mechanism $\mathcal{M}_r$ sets the prices charged to the shippers, and the carrier-side mechanism $\mathcal{M}_p$ specifies payment and load allocation rules for the carriers. %Let $\Pi$ denote the set of all feasible platform mechanisms. 

\vspace{0.1in}
\noindent \textbf{Mechanisms for the Shipper Side}

Given the system state $(\mathbf{S}_t,\mathbf{D}_t)$, the shipper-side mechanism $\mathcal{M}_r$ specifies a pricing vector $\mathbf{r}_t = (r_{ijt}: (i,j)\in \mathcal{E})$, where $r_{ijt}$ is the spot price for each load that needs to be shipped from node $i$ to node $j$ in period $t$. In response to the pricing vector $\mathbf{r}_t$, shippers choose to buy transportation services from the marketplace and submit load information to the platform.
The loads that need to be transported from node $i$ to node $j$ are submitted to the platform in period $t$ according to a Poisson distribution with demand rate $d_{ij}(r_{ijt})$. In other words, the demand $D_{ijt+1}$ for the \emph{next} period is drawn from a Poisson distribution with mean $d_{ij}(r_{ijt})$. (Recall that all loads submitted by shippers during period $t$ will be released to carriers at the beginning of period $t+1$.)
Assuming the demand function $d_{ij}(r_{ijt})$ is strictly decreasing,
we denote by $r_{ij}(d_{ijt})$ the inverse demand function given demand rate $d_{ijt}$ in period $t$.
%that explicitly captures the dependency of $r_{ijt}$ on $d_{ijt}$. 
Finally, we make the standard assumption that the revenue function $r_{ij}(d_{ijt})d_{ijt}$ is concave in $d_{ijt}$. The concavity assumption on the revenue function is commonly assumed in the literature, stemming from the economic principle of diminishing marginal returns.

\begin{remark}
	In addition to spot rates, it is common for freight platforms to offer contract rates to shippers. 
	A truckload contract covers all loads offered by a shipper during
	a specific period (e.g., a year) at fixed rate formulas. We remark that our shipper model defined above holds for this general setting, where the demand $d_{ijt}$ is interpreted as the sum of contract and spot demands. Specifically, we define the overall demand as $d_{ijt}(r_{ijt}) := d^{\text{contract}}_{ijt} + d^{\text{spot}}_{ijt}(r_{ijt})$, where $d^{\text{contract}}_{ijt}$ is assumed to be exogenous and independent of the spot price $r_{ijt}$.
\end{remark}

\vspace{0.1in}
\noindent \textbf{Mechanism for the Carrier Side}

Carrier-side mechanisms are more complex than shipper-side mechanisms. To formally describe the carrier-side mechanism, we need to introduce some notations to define carriers' utility functions. Consider a carrier \rr{$s \in \mathcal{S}_{it}$}. 
If the carrier $s$ transports a load on lane $(i,j)$ in period $t$
and receives a payment amount $p^s_{ijt}$, his net utility $U^s_{ijt}$ is determined by a random utility model.
More specifically, we assume that
\begin{equation} \label{eq: carrier utility}
	U^s_{ijt} := \beta p^s_{ijt} - \alpha_{ij} + \epsilon^s_{ijt}, \quad \forall (i,j)\in \mathcal{E},
\end{equation}
where $\beta > 0$ denotes the price sensitivity parameter, $\alpha_{ij} > 0$ represents the average cost for carriers to transport a load from node $i$ to node $j$, and $\epsilon^s_{ijt}$ represents the idiosyncratic error terms of the carrier $s$, which are
assumed to be \rr{independent and identically distributed (i.i.d.) random variables following the Gumbel distribution with location parameter zero and scale parameter one.} 
If the carrier $s$ chooses to not book any load from the platform, he may choose any outside option (i.e., the null alternative) resulting in a utility of \rr{$U^s_{i0t} :=\epsilon^s_{i0t}$, where $\epsilon^s_{i0t}$ follows the same Gumbel distribution with location parameter zero and scale parameter one.} The definition of this random utility model implies that the \emph{opportunity cost} of the carrier $s$ for transporting a load on lane $(i,j)$ in period $t$ with the platform is
\rr{\begin{equation} \label{eq: true opportunity cost}
		C^s_{ijt}:=(\alpha_{ij} -\epsilon^s_{ijt} + \epsilon^s_{i0t})/\beta, \quad \forall (i,j)\in \mathcal{E},
	\end{equation}
	where Eq \eqref{eq: true opportunity cost} is obtained by letting $U^s_{ijt} = U^s_{i0t}$.}

\begin{remark}
	When modeling carrier-side mechanisms, we assume that the platform's policy only affects carriers' utilities in the current period, but not in future periods.
	This assumption is reflected by the highly competitive nature of freight marketplaces. 
	%Unlike other transportation marketplaces (e.g. ride-sharing, delivery) with high market concentration, the truckload freight markets are very fragmented. For example, in the U.S. truckload freight market, there are thousands of brokerage companies; even the largest freight platforms among these have only a few percentage points of market share. If a carrier hauls a load with one platform in one period, it is quite possible that they book loads through another digital platform or through a human broker in the next period. Therefore, a carrier's utility in future periods is mainly decided by the overall market condition, rather than the decision made by a single platform.
    Unlike other transportation sectors, such as ride-sharing or delivery, which are dominated by a few major platforms, the truckload freight market is highly fragmented, with tens of thousands of freight  brokerage companies
operating in the United States. Carriers frequently book loads with different brokerage platforms in subsequent periods, which means that the pricing policy of one particular platform in one period has small impact on carriers’ earnings in the future. Therefore, a carrier’s utility in future periods is primarily determined by the overall market condition, rather than the decision made by a single platform.
\end{remark}

We now define a carrier-side mechanism $\mathcal{M}_p=(\mathbf{A},\mathbf{P})$. By the revelation principle, we focus on direct mechanisms exclusively. 
In period $t$, when carriers arrive to the platform, they submit their bids (representing their opportunity costs) to the platform.
Given a vector of bids $\mathbf{c}_{t}$ submitted by carriers, an allocation rule $\mathbf{A}$ determines the probability $A_{ijt}^s(\mathbf{c}_{t}) \in [0,1]$ that the carrier $s$ is allocated a load from origin node $i$ to destination node $j$ in period $t$, and a payment rule $\mathbf{P}$ determines the expected payment $P^s_{ijt}(\mathbf{c}_{t})$ made to the carrier $s$ for the service provided. 

In this paper, we consider \rr{two types of carrier-side mechanism $\mathcal{M}_p$: posted price mechanisms, and hybrid mechanisms that combine pricing and auction.} In a \emph{posted price mechanism}, for each load in the marketplace, the platform sets a price that is offered to carriers for transporting the load. Since  carriers have i.i.d.\ opportunity costs, we shall drop the carrier index $s$ if we are referring to an arbitrary carrier in \rr{$\mathcal{S}_{it}$. Given a posted price vector $\mathbf{p_t}=(p_{ijt}: (i,j)\in \mathcal{E})$, the carriers' utilities defined in Eq~\eqref{eq: carrier utility} imply that
	each carrier's choice among the available loads follows the multinomial logit (MNL) model. It is well known by the multinomial logit model (cf.~\citealt{McFadden73}) that a carrier at node $i$ will choose to transport a load from node $i$ to node $j$ with probability $x_{ijt}$, where
	\begin{equation*}  \label{eq: def of x_ijt}
		x_{ijt} = \frac{e^{\beta p_{ijt} -\alpha_{ij}}}{\sum_{k\in \delta^+(i)} e^{\beta p_{ikt} -\alpha_{ik}}+1}, \quad \forall j\in \delta^+(i),
	\end{equation*}
	and the probability that a carrier chooses to not book any load and leave the marketplace is given by
	\begin{equation*} \label{eq: def of w_ijt}
		w_{it} = \frac{1}{\sum_{k\in \delta^+(i)} e^{\beta p_{ikt} -\alpha_{ik}}+1}.
	\end{equation*}
	It then follows that $p_{ijt}$ can be expressed as a function of $x_{ijt}$ and $w_{it}$:
	\begin{equation}   \label{eq: def of p_ijt}
		p_{ijt}= \frac{1}{\beta}\left[\ln \left (\frac{x_{ijt}}{w_{it}} \right)+\alpha_{ij} \right], \quad \forall j\in \delta^+(i).
\end{equation}}
Alternatively, the platform may \rr{adopt a hybrid mechanism by combining the pricing mechanism with an \emph{auction} mechanism. In a \emph{hybrid mechanism}, carriers follow a two-stage decision process. In the first stage, carriers choose a lane according to the MNL model based on the prices posted. Then in the second stage, carriers decide whether to book a load on the chosen lane instantly by accepting the posted price, submit a bid for a load on the chosen lane and join an auction, or choose the outside option and leave the platform. Since carriers may have the opportunity to receive updated information in the second stage (e.g., the number of carriers who chose a specific lane in the first stage) and possibly obtain a larger utility through the auction, we assume that carriers consider the outside option in the second stage rather than in the first stage. At the beginning of the second stage, carriers' lane choices are realized and the distribution of carriers' true opportunity cost should be updated to a posterior distribution conditioning on the lane choice decisions. Let $F_{ij}(\cdot)$ and $f_{ij}(\cdot)$ be the posterior cumulative distribution function and the posterior probability density function of a carrier's true opportunity cost after choosing lane $(i,j)$, respectively. It is easy to verify that 
	\begin{align*}
		F_{ij}(c) &= \Pr\{C_{ijt}^s\leq c |U_{ijt}^s>U_{ikt}^s,\forall k\in\delta^+(i), k\neq j\} \\
		&=\Pr\{C_{ijt}^s\leq c | C_{ijt}^s<(p_{ijt}-p_{ikt})+C_{ikt}^s,\forall k\in\delta^+(i), k\neq j\}.
	\end{align*}
	Finally, let $\psi_{ij}(\cdot)$ be a \emph{virtual cost} function defined as
	\begin{align}  
		\psi_{ij}({C}_{ijt}^s) := {C}_{ijt}^s + \frac{F_{ij}({C}_{ijt}^s)}{f_{ij}({C}_{ijt}^s)}. \label{eq:psi-def}
	\end{align}
	We show in the following lemma  that $\psi_{ij}$ satisfies the standard regularity condition, and we refer the readers to the appendix for the complete proof. 
	\begin{lemma}\label{Lemma: virtual function properties}
		The virtual cost function $\psi_{ij}$ defined in (\ref{eq:psi-def}) is strictly increasing and its inverse function $\psi^{-1}_{ij}$ exists.
\end{lemma}}
\rr{We next use a simple example to illustrate how an auction may work in the context of truckload platforms when a hybrid mechanism is adopted. Consider an origin-destination pair $(i,j) \in \mathcal{E}$ and let $\tilde{S}_{ijt}$ denote the total number of carriers in $\mathcal{S}_{it}$ who have chosen lane $(i,j)$ based on the posted prices in the first stage}. Given the observed system state $(\mathbf{S}_t,\mathbf{D}_t)$ in period $t$ and the set of bids $\mathbf{c}_{t}$ that the carriers have reported, the platform sorts the carriers' bids from the smallest to the largest such that $c_{ijt}^{[1]} \le c_{ijt}^{[2]} \le \cdots \le c_{ijt}^{[\rr{\tilde{S}_{ijt}}]}$ and uses the following allocation and payment rules. For each carrier \rr{$s$}, $A_{ijt}^s(\mathbf{c}_{t}) = 1$ and $P_{ijt}^s(\mathbf{c}_{t}) = c^s_{ijt}$ if $c^s_{ijt} \le c_{ijt}^{[\min(\rr{\tilde S_{ijt}}, D_{ijt})]}$. Otherwise, $A_{ijt}^s(\mathbf{c}_{t}) = 0$ and $P_{ijt}^s(\mathbf{c}_{t}) = 0$ if $c^s_{ijt} > c_{ijt}^{[\min(\rr{\tilde S_{ijt}}, D_{ijt})]}$. In other words, when the number of available carriers \rr{$\tilde S_{ijt}$} is no less than the number of loads $D_{ijt}$ that need to be transported, the platform allocates the $D_{ijt}$ loads to the first $D_{ijt}$ lowest bids with probability one, and the payment amount is equal to each carrier's reported opportunity cost. When there are not enough carriers, i.e., $\rr{\tilde S_{ijt}} < D_{ijt}$, each carrier is allocated a load with probability one and receives a payment amount equal to their bid $c_{ijt}^s$. 

\subsection{Dynamic Program Formulation} \label{subsection: DP formulation}

In this section, we provide a dynamic program formulation for the platform's optimization problem.
In period $t$, given the state $(\mathbf{S}_t,\mathbf{D}_t)$, a mechanism $\pi$ maps the state to a pair of shipper/carrier mechanisms $(\mathcal{M}_r, \mathcal{M}_p)$.
Given the carrier-side mechanism $\mathcal{M}_p$, each available carrier either chooses to book a load or leaves the marketplace without any booking. 
We consider a full truck load assumption where each carrier can transport only one load at a time.
For each $(i,j) \in \mathcal{E}$, let $Y_{ijt}$ be a random variable that denotes the number of carriers who have transported a load on lane ($i,j)$ in period $t$, and let \rr{$V_{it}$ be the number of carriers in $\mathcal{S}_{it}$ who leave the marketplace from node $i$} in period $t$. It then follows that for each period $t$, we have 
\begin{equation} \label{eq: sum Y and V = S}
		\sum_{j\in\delta^+(i)} Y_{ijt} + V_{it}=S_{it}, \quad \forall i \in \mathcal{N}.
\end{equation}
For a carrier who just completes a load shipment from node $i$ to node $j$ in period $t$, we assume that the carrier will \rr{stay in the marketplace with probability $q_{ij} \in (0,1)$, and will leave the system with probability $1-q_{ij}$. Let $Z_{ijt}$ be a random variable that denotes the number of carriers who decide to stay in the marketplace after completing a load shipment from node $i$ to node $j$ in period $t$. Note that $Z_{ijt}$ follows a Binomial distribution with parameters ($Y_{ijt}$, $q_{ij}$). Then, the dynamics of how the number of carriers evolving over time is characterized as follows:
	\begin{equation} \label{eq: dynamics of St}
		S_{jt+1} =  \sum_{i\in \delta^-(j)} Z_{ijt} + \Lambda_{jt+1}, \quad \forall j\in \mathcal{N},
	\end{equation} 
	where $\Lambda_{jt+1}$ is a Poisson random variable with rate $\lambda_{j}>0$ that denotes the number of new carriers who exogenously arrive to node $j$ in period $t+1$. }

For carriers who have completed a load shipment in period $t$, the platform makes payments to them.
%based on the pre-announced policy $\mathcal{M}_p$. 
Let $P_{ijt}$ denote the total payment to all the carriers who have transported a load from node $i$ to node $j$ in period $t$. It is then clear from the definitions that $\Ex\left[\sum_{s=1}^{\rr{{S}_{it}}}A_{ijt}^s(\mathbf{C}_t)\middle|\mathbf{S}_t,\mathbf{D}_t\right] =\Ex[Y_{ijt}|\mathbf{S}_t,\mathbf{D}_t]$ and $\Ex\left[\sum_{s=1}^{\rr{{S}_{it}}}P_{ijt}^s(\mathbf{C}_t)\middle|\mathbf{S}_t,\mathbf{D}_t\right] =\Ex[P_{ijt}|\mathbf{S}_t,\mathbf{D}_t]$ for each $(i,j)\in \mathcal{E}$. If there is not enough carriers to transport all the loads (i.e., $Y_{ijt} < D_{ijt}$), we assume that the excess demand $(D_{ijt}-Y_{ijt})$ incurs a unit penalty cost  
$b_{ij}$.
Therefore, the total payment that the platform has made during period $t$ is equal to $\sum_{(i,j)\in \mathcal{E} } P_{ijt} + b_{ij}(D_{ijt}-Y_{ijt})$.

Next we consider the shipper side dynamics. 
Recall that a shipper-side mechanism $\mathcal{M}_r$ specifies the prices $\mathbf{r}_t$ charged to the shippers, which determine the number of loads $D_{ijt+1}$ that need to be transported from node $i$ to node $j$ in period 
$t+1$. We assume that $D_{ijt+1}$ follows a Poisson process with rate $d_{ij}(r_{ijt})$. 
The revenue that the platform collects from the shippers is $r_{ijt}D_{ijt+1}$ with mean $r_{ijt} d_{ij}(r_{ijt})$.
It is worth pointing out that, although these loads are transported in period $t+1$, we account them for the revenue in period $t$. Therefore, the platform's expected profit in period $t$ is given by: 
\begin{equation}  \label{Ex profit under policy pi}
	G^{{\pi}}_t(\mathbf{S}_t, \mathbf{D}_t) :=  \Ex \left[ \sum_{(i,j)\in \mathcal{E}} r_{ijt} d_{ij}(r_{ijt}) - b_{ij}(D_{ijt}-Y_{ijt}) - P_{ijt}  \;\middle|\; \mathbf{S}_t,\mathbf{D}_t \right].
\end{equation}
Define $\gamma^{{\pi}}$ as the long-run average profit per period under a stationary policy ${\pi}$:
\[
\gamma^{{\pi}} := \lim_{T \rightarrow \infty} \frac{1}{T}\Ex \left[
\sum_{t=1}^{T} G^{{\pi}}_t(\mathbf{S}_t, \mathbf{D}_t)\right]. 
\] 
The existence of $\gamma^{{\pi}}$ is implied by Proposition~\ref{Prop: stationary} that we detail below in Section \ref{subsubsection: stability IC and IR}.
Let $\gamma^*$ denote the optimal long-run average profit per period, and let $h(\mathbf{S}_t,\mathbf{D}_t)$ denote the differential cost for the state $(\mathbf{S}_t,\mathbf{D}_t)$. Then the optimality equation is given by
\rr{
	\begin{equation}
		\begin{aligned}
			& \gamma^*+h(\mathbf{S}_t,\mathbf{D}_t) = \max_{{\pi}\in \Pi} \, \Ex\left[G^{{\pi}}_t(\mathbf{S}_t,\mathbf{D}_t) + \Ex[h(\mathbf{Z}_t^T\mathds{1}+\boldsymbol{\Lambda}_{t+1},\mathbf{D}_{t+1})]\;\middle|\; \mathbf{S}_t,\mathbf{D}_t \right], \quad \forall (\mathbf{S}_t, \mathbf{D}_t)
			\label{eq: MDP}
		\end{aligned} \tag{\textsf{MDP}} 
	\end{equation}
	where $\boldsymbol{\Lambda_t}=(\Lambda_{it},\forall i\in \mathcal{N})$, $\mathbf{Z}_t=(Z_{ijt},\forall (i,j)\in \mathcal{E})$, and $\mathds{1}$ is a vector with all entries equal to one. }

\subsection{Stability, Incentive Compatibility (IC) and Individual Rationality (IR).} \label{subsubsection: stability IC and IR}

In this subsection, we introduce some properties of the platform's mechanism $(\mathcal{M}_r,\mathcal{M}_p)$ introduced above. 
Let $\pi \in \Pi$ be a stationary policy that maps the system state $(\mathbf{S}_t, \mathbf{D}_t)$ to a mechanism $(\mathcal{M}_r,\mathcal{M}_p)$.
Proposition \ref{Prop: stationary} below shows that the Markov chain induced by any stationary policy ${\pi}$ is stable. The proof uses the Foster-Lyapunov theorem \citep{foster1953stochastic}
with a Lyapunov function on the total number of carriers in the system, and we refer the readers to the appendix for the complete proof.

\begin{proposition} \label{Prop: stationary}
	There exists a stationary distribution of the Markov chain induced by the platform's policy ${\pi}$ and the system is stable (i.e., positive recurrent).
\end{proposition}

By the revelation principle \citep{Myerson_1981}, we will focus on direct mechanisms that satisfy the Bayesian incentive compatibility (IC) and individual rationality (IR) properties. 
We next briefly discuss these properties and the conditions under which IC and IR constraints are satisfied.

\rr{Consider a carrier $s \in \mathcal{S}_{it}$ who chooses to deliver a load from node $i$ to node $j$ according to the MNL model at the first stage of period $t$.} Let $c^s_{ijt}$ denote the bid for a load on lane $(i,j)$ submitted by the carrier $s$ to the platform. In other words, $c^s_{ijt}$ represents the opportunity cost reported by the carrier $s$ for transporting a load from node $i$ to node $j$ in period $t$. \rr{Let $\tilde S_{ijt}$ denote the total number of carriers in $\mathcal{S}_{it}$ who choose lane $(i,j)$, and let $\mathbf{C}^{-s}_{ijt} = \left({C}^{1}_{ijt},...,{C}^{s-1}_{ijt},{C}^{s+1}_{ijt},...,{C}^{\tilde{S}_{ijt}}_{ijt}\right)$ be a vector that represents the opportunity costs of all the  carriers who have chosen lane $(i,j)$ in period $t$ other than $s$, where $C^{s}_{ijt}$ is defined in \eqref{eq: true opportunity cost}. Let $g^s(\cdot)$ denote the joint probability density function of $\mathbf{C}^{-s}_{ijt}$. Finally, let $(c^s_{ijt},\mathbf{C}^{-s}_{ijt}) = \left({C}^{1}_{ijt},...,{C}^{s-1}_{ijt},c^s_{ijt},{C}^{s+1}_{ijt},...,{C}^{\tilde{S}_{ijt}}_{ijt}\right)$.}

Assume that all the carriers other than $s$ report their true opportunity costs $\mathbf{C}^{-s}_{ijt}$ when submitting the bids. Then the carrier $s$ who submitted bid $c^s_{ijt}$ will be allocated a load to be transported from node $i$ to node $j$ in period $t$ with probability 
\[
a^s_{ijt}(c^s_{ijt}) := \int A^s_{ijt}(c^s_{ijt},\mathbf{C}^{-s}_{ijt})g^s(\mathbf{C}^{-s}_{ijt})d\mathbf{C}^{-s}_{ijt},
\] 
and the expected payment to the carrier $s$ for transporting a load from node $i$ to node $j$ is 
\[
p^s_{ijt}(c^s_{ijt}) := \int P^s_{ijt}(c^s_{ijt},\mathbf{C}^{-s}_{ijt})g^s(\mathbf{C}^{-s}_{ijt})d\mathbf{C}^{-s}_{ijt}.
\]

Let $C^s_{ijt}$ denote the true opportunity cost of the carrier $s$. Then the expected net utility of the carrier $s$ when he submits a bid $c^s_{ijt}$ is given by
\[
{u}^s_{ijt}(c^s_{ijt}):={p}^s_{ijt}(c^s_{ijt}) - a^s_{ijt}(c^s_{ijt}) C^s_{ijt}. 
\]
A carrier-side mechanism $\mathcal{M}_p$ is \emph{incentive compatible} if it is a Bayesian Nash equilibrium for each carrier to report his true opportunity cost. That is, for each carrier $s$, we have
$$
{u}^s_{ijt}(C^s_{ijt}) = {p}^s_{ijt}(C^s_{ijt})-a^s_{ijt}(C^s_{ijt}) C^s_{ijt} \ge {p}^s_{ijt}(c^s_{ijt})-a^s_{ijt}(c^s_{ijt}) C^s_{ijt}, \quad \forall c^s_{ijt}.
$$
We say a carrier-side mechanism $\mathcal{M}_p$ is \emph{individual rational} if each carrier's expected net utility is non-negative when they report their true opportunity costs, i.e., $u^s_{ijt}(C^s_{ijt})\ge 0$ for each carrier $s$.

Given a platform mechanism $(\mathcal{M}_r,\mathcal{M}_p)$ that is IC and IR, 
the following revenue equivalence principle characterizes the carrier's expected payment by the allocation rule.
\begin{proposition}\label{prop:p = a and psi}
	Under a given state $(\mathbf{S},\mathbf{D})$ with the platform's mechanism ${\pi}(\mathbf{S},\mathbf{D})=(\mathcal{M}_r,\mathcal{M}_p)$ that is IC and IR, the expected payment to a carrier $s\in \mathcal{S}_{i}$ who chooses lane $(i,j)$ is
	\begin{equation} \label{eq: expected payment to carrier s}
		\Ex[{p}^s_{ij}(C^s_{ij})|\mathbf{S},\mathbf{D}]=\Ex[a^s_{ij}(C^s_{ij})\psi_{ij}(C^s_{ij})|\mathbf{S},\mathbf{D}],
	\end{equation}
	where $C^s_{ij}$ denotes the true opportunity cost of the carrier $s$, and the virtual cost function $\psi_{ij}(\cdot)$ is defined in \rr{Eq~\eqref{eq:psi-def}}. 
\end{proposition}

\section{Fluid Approximation and Profit Benchmark}
\label{section: Fluid Approx}

The dynamic programming formulation \ref{eq: MDP} is intractable since the number of nodes may be large in practice and the state space grows at least exponentially with the number of nodes. This motivates us to consider a fluid approximation of the \ref{eq: MDP}, where the random shipper demands and carrier arrivals are replaced with their respective mean values, and we consider the system under the stationary distribution. In this section, we first provide the formulation of the fluid model. Then, we show that the fluid model can be transformed into a convex optimization problem and hence can be solved efficiently. Finally, we show that the optimal objective value of the fluid optimization problem serves as an upper bound for the long-run average profit for the \ref{eq: MDP} under \emph{any} stationary mechanism. This upper bound is useful in that it can be used as a benchmark to establish performance guarantees of any given mechanism. As we shall show later in Section \ref{section: posted price mechanism}, a simple static posted price mechanism for the \ref{eq: MDP} based on the solution to the fluid optimization problem is asymptotically optimal.

\subsection{The Fluid Model}  \label{subsection: Fluid Model}

In the fluid model, the random shipper demands $D_{ijt}$ and random carrier arrivals $\Lambda_{it}$ in each period $t$ are replaced with their mean values $d_{ij}$ and $\lambda_i$, respectively, for each node $i \in \mathcal{N}$ and every origin-destination pair $(i,j) \in \mathcal{E}$. We consider the fluid system in a steady state.
	Let $x_{ij}$ denote the fraction of carriers at node $i$ who choose to transport a load from node $i$ to node $j$ in the fluid system, and let $w_{i}$ be the fraction of carriers at node $i$ who choose to not book any load and leave the marketplace. Suppose for now that a posted price mechanism is used on the carrier side. It then follows from Eq~\eqref{eq: def of p_ijt} that the payment $p_{ij}$ offered to carriers for transporting a load from node $i$ to node $j$ is given by
	\[
	p_{ij}= \frac{1}{\beta}\left[\ln \left (\frac{x_{ij}}{w_{i}} \right)+\alpha_{ij} \right], \quad \forall (i,j)\in \mathcal{E}.
	\] 
It is worth pointing out that the fraction of carriers who \emph{choose} a load from node $i$ to node $j$ may not be equal to the number of carriers who actually \emph{ship} a load from node $i$ to node $j$, since the actual shipment depends on both the carriers' choices and the shipper demand. In view of this distinction, we let $y_{ij}$ be the fraction of carriers at node $i$ who end up hauling loads from node $i$ to node $j$, and let $v_{i}$ be the fraction of carriers who leave the system from node $i$. Let $\bar{\lambda}_i$ denote the total available  carriers at node $i$,  which includes both the new arrivals and the existing carriers in the marketplace who just finished a shipment and decided to stay at node $i$. 

We define the fluid approximation model as
\begin{subequations}
		\begin{align}
			\max_{\mathbf{d},\mathbf{x},\mathbf{w},\mathbf{y},\mathbf{\bar{\lambda}},\mathbf{v}}  \quad &   \sum_{(i, j)\in \mathcal{E}}r_{ij}(d_{ij}) d_{ij} -  \frac{1}{\beta}\left[\ln \left (\frac{x_{ij}}{w_{i}} \right)+\alpha_{ij} \right] \bar{\lambda}_i y_{ij} - b_{ij}(d_{ij}-\bar{\lambda}_i y_{ij}) \label{objective_naive} \\
			\mbox{s.t.}\quad 
			&\;\sum_{i\in\delta^-(j)}q_{ij}\bar{\lambda}_i y_{ij}+\lambda_j = \bar{\lambda}_j \left( \sum_{k\in\delta^+(j)} y_{jk}+v_j \right),  \quad \forall j\in \mathcal{N}, \label{flow_naive}\\
			&\;\sum_{k\in\delta^+(j)} x_{jk}+w_j =1, \quad \forall j\in \mathcal{N},\label{prob_x_naive}\\
			&\;\sum_{k\in\delta^+(j)} y_{jk}+v_j =1, \quad \forall j\in \mathcal{N},\label{prob_y_naive}\\
			&\; \bar{\lambda}_iy_{ij}\leq d_{ij}, \quad \quad \forall (i,j)\in \mathcal{E}, \label{demand_naive}\\
			&\; y_{ij}\leq x_{ij}, \quad \forall (i,j)\in \mathcal{E},\label{relation_x_y_naive}\\
			&\; x_{ij},y_{ij}\ge 0, \quad \forall (i,j)\in \mathcal{E} \mbox{ and }  w_{i},v_{i},\bar{\lambda}_i\ge 0, \quad \forall i\in \mathcal{N}\nonumber.
		\end{align} \label{original FA}
\end{subequations}
In the above fluid model, the objective function \eqref{objective_naive} maximizes the platform's per period profit, where the first term represents the revenues received from shippers, the second term represents payments made to carriers, and the last term represents the penalty costs incurred from unsatisfied demand. \rr{Constraint \eqref{flow_naive} represents the flow balance equations, where the left-hand side represents the total inflow rate to node $j$, and the right-hand side represents the outflow rate from node $j$.} Constraints \eqref{prob_x_naive} and \eqref{prob_y_naive} follow from the definition of probability vectors $\mathbf{x}, \mathbf{w}, \mathbf{y}$, and $\mathbf{v}$. Constraint \eqref{demand_naive} states that the flow of loaded carriers cannot exceed the number of loads that are available for each O-D pair. Constraint \eqref{relation_x_y_naive} requires that $y_{ij}$ cannot exceed $x_{ij}$ since the actual shipment is constrained by the carriers' supply.

\subsection{Convex Reformulation}  \label{subsection: convex reformulation}

In this section, we show that the fluid model presented in Section \ref{subsection: Fluid Model} can be reformulated into a \rr{(convex) conic} optimization problem and hence can be solved efficiently. To that end, we first introduce some notations.
\rr{For each node $i \in \mathcal{N}$ and each O-D pair $(i,j) \in \mathcal{E}$, define $\bar{x}_{ij} := \bar{\lambda}_ix_{ij}$, $\bar{w}_{i} := \bar{\lambda}_iw_{i}$, $\bar{y}_{ij} := \bar{\lambda}_iy_{ij}$, and $\bar{v}_{i} := \bar{\lambda}_iv_{i}$. By the definition of these new variables, it is clear that 
	\[
	p_{ij}=\frac{1}{\beta}\left[\ln \left (\frac{x_{ij}}{w_{i}} \right)+\alpha_{ij} \right] = \frac{1}{\beta}\left[\ln \left (\frac{\bar{x}_{ij}}{\bar{w}_{i}} \right)+\alpha_{ij} \right], \quad \forall (i,j)\in \mathcal{E}.
	\]} 
By substituting the new variables into formulation (\ref{original FA}), we have
\rr{\begin{subequations}
		\begin{align}
			\max_{\mathbf{d}, \bar{\mathbf{x}}, \bar{\mathbf{w}}, \bar{\mathbf{y}}, \bar{\mathbf{v}}}  & \quad 
			\sum_{(i, j)\in \mathcal{E}}r_{ij}(d_{ij})d_{ij} - \frac{1}{\beta}\left[\ln \left (\frac{\bar{x}_{ij}}{\bar{w}_{i}} \right)+\alpha_{ij} \right] \bar{y}_{ij} - b_{ij}(d_{ij}-\bar{y}_{ij}) \label{obj: FA}
			\\
			\mbox{s.t. } \; & \sum_{i\in\delta^-(j)}q_{ij}\bar{y}_{ij}+\lambda_j = \sum_{k\in\delta^+(j)}\bar{y}_{jk}+\bar{v}_{j},\quad \forall j\in \mathcal{N},\label{flow_FA}\\
			\; & \sum_{k\in\delta^+(j)}\bar{y}_{jk}+\bar{v}_j = \sum_{k\in\delta^+(j)}\bar{x}_{jk}+\bar{w}_j, \quad \forall j\in\mathcal{N},\label{flow_x_y_FA}\\
			&\;\bar{y}_{ij} \leq d_{ij}, \quad  \forall (i,j)\in\mathcal{E},\\
			& \;\bar{y}_{ij} \le \bar{x}_{ij}, \quad  \forall (i,j)\in\mathcal{E}, \\
			& \;\bar{x}_{ij},\bar{y}_{ij} \ge 0, \quad  \forall (i,j)\in\mathcal{E} \mbox{  and  } \bar{w}_{i},\bar{v}_{i} \ge 0, \quad  \forall i \in \mathcal{N}. \label{noneneagetive_FA}
		\end{align} \label{new FA}
\end{subequations}}
The interpretations of the constraints in the new formulation \eqref{new FA} are straightforward and similar to those in the original formulation \eqref{original FA}. Next, we show that the new formulation \eqref{new FA} and the original fluid model \eqref{original FA} are equivalent. First, it is easy to check that for any feasible solution to the original formulation \eqref{original FA}, there exists a corresponding feasible solution to the new formulation \eqref{new FA} with the same objective value by the definition of variables $\bar{\mathbf{x}}, \bar{\mathbf{w}}, \bar{\mathbf{y}}$, \rr{and $\bar{\mathbf{v}}$}.
\rr{On the other hand, for any given feasible solution $(\mathbf{d}, \bar{\mathbf{x}}, \bar{\mathbf{w}}, \bar{\mathbf{y}}, \bar{\mathbf{v}})$ to the new formulation, let $\bar{\lambda}_j = \sum_{i}q_{ij}\bar{y}_{ij}+\lambda_j$ for each $j \in \mathcal{N}$, and define variables $y_{ij}=\bar{y}_{ij}/\bar{\lambda}_i$, $x_{ij}=\bar{x}_{ij}/\bar{\lambda}_i$, $w_{i}=\bar{w}_i/\bar{\lambda}_i$,
	$v_{i}=\bar{v}_i/\bar{\lambda}_i$ for each $(i,j) \in \mathcal{E}$ and $\mathbf{d}$ remains unchanged.} It is easy to check that these newly defined variables  are feasible to problem \eqref{original FA}, and the new formulation can be reduced to the original formulation. 

\rr{Let $(\mathbf{d}^*,\bar{\mathbf{x}}^*, \bar{\mathbf{w}}^*, \bar{\mathbf{y}}^*, \bar{\mathbf{v}}^*)$ be an optimal solution to \eqref{new FA}. We note that the optimal solution satisfies $\bar{\mathbf{x}}^* = \bar{\mathbf{y}}^*$. To see this, suppose that there exists some node $i \in \mathcal{N}$ such that $\bar{y}^*_{ij}<\bar{x}^*_{ij}$ for some $j\in \delta^+(i)$. This implies that $\bar{w}^*_{i} < \bar{v}^*_{i}$. In addition, we have
	\[
	\frac{1}{\beta}\left[\ln \left (\frac{\bar{x}^*_{ij}}{\bar{w}^*_{i}} \right)+\alpha_{ij} \right]  > \frac{1}{\beta}\left[\ln \left (\frac{\bar{y}^*_{ij}}{\bar{v}^*_{i}} \right)+\alpha_{ij} \right]. \]  
	In this case, we can construct a new solution $(\mathbf{d}', \bar{\mathbf{x}}', \bar{\mathbf{w}}', \bar{\mathbf{y}}', \bar{\mathbf{v}}')$, where  $\bar{x}'_{ij} := \bar{y}^*_{ij}$ and $\bar{w}'_{i} := \bar{w}^*_{i} + (\bar{x}^*_{ij} - \bar{y}^*_{ij})$, and the rest of the variables have the same value as that in $(\mathbf{d}^*, \bar{\mathbf{x}}^*, \bar{\mathbf{w}}^*, \bar{\mathbf{y}}^*, \bar{\mathbf{v}}^*)$. It is straightforward to check that $(\mathbf{d}', \bar{\mathbf{x}}', \bar{\mathbf{w}}', \bar{\mathbf{y}}', \bar{\mathbf{v}}')$ is feasible to \eqref{new FA} and achieves a strictly larger objective value. This leads to a contradiction with the optimality of $(\mathbf{d}^*, \bar{\mathbf{x}}^*, \bar{\mathbf{w}}^*, \bar{\mathbf{y}}^*, \bar{\mathbf{v}}^*)$, and hence we must have $\bar{\mathbf{x}}^* = \bar{\mathbf{y}}^*$. 
	In view of this, formulation \eqref{new FA} can be simplified as follows:
	\begin{align}
		(\textsf{FA}): \quad \max_{\mathbf{d}, \bar{\mathbf{y}}, \bar{\mathbf{v}}}  & \quad 
		\sum_{(i, j)\in \mathcal{E}}r_{ij}(d_{ij})d_{ij} - \frac{1}{\beta}\left[\ln \left (\frac{\bar{y}_{ij}}{\bar{v}_{i}} \right)+\alpha_{ij} \right] \bar{y}_{ij} - b_{ij}(d_{ij}-\bar{y}_{ij})\nonumber
		\\
		\mbox{s.t. } \; & \sum_{i\in \delta^-(j)}q_{ij}\bar{y}_{ij}+\lambda_j = \sum_{k\in \delta^+(j)}\bar{y}_{jk}+\bar{v}_{j},\quad \forall j\in \mathcal{N},\label{FA linking constraint}  \\
		&\;\bar{y}_{ij} \leq d_{ij}, \quad  \forall (i,j)\in \mathcal{E},\label{FA demand constraint}\\
		& \;\bar{y}_{ij} \ge 0, \quad  \forall (i,j)\in \mathcal{E} \mbox{ and } \bar{v}_{i} \ge 0, \quad  \forall i\in \mathcal{N}. \nonumber 
\end{align} }
\rr{In what follows, we show that \textsf{FA} is a convex optimization problem under the assumption that $r_{ij}(d_{ij})d_{ij}$ is concave in $d_{ij}$. Let $\tau_{ij}$ be an upper bound of $\frac{1}{\beta}\left[\ln \left (\frac{\bar{y}_{ij}}{\bar{v}_{i}} \right)\right] \bar{y}_{ij}$. Then, we have
	\[
	\tau_{ij}\geq \frac{1}{\beta}\left[\ln \left (\frac{\bar{y}_{ij}}{\bar{v}_{i}} \right)\right] \bar{y}_{ij} \, \Leftrightarrow \, \beta \tau_{ij}\ge \ln \left(\frac{ \bar{y}_{ij}}{ \bar{v}_i} \right) \bar{y}_{ij} \, \Leftrightarrow \, (\bar{v}_i,\bar{y}_{ij}, -\beta \tau_{ij})\in K_{\text{exp}},    
	\]
	where $K_{\text{exp}}$ is the \emph{exponential cone} defined as
	\[
	K_{\text{exp}}=\{\boldsymbol{\xi}\in \mathbb{R}^3:\xi_0\ge \xi_1 \exp(\xi_2/\xi_1), \xi_0,\xi_1\ge 0\}.
	\]
	Then, the \textsf{FA} reduces to the following (convex) conic optimization problem, which can be readily solved by commercial solvers:
	\begin{align*}
		\max_{\mathbf{d}, \bar{\mathbf{y}}, \bar{\mathbf{v}},\mathbf{\tau}}  & \quad 
		\sum_{(i, j)\in \mathcal{E}}r_{ij}(d_{ij})d_{ij} - \frac{\alpha_{ij} }{\beta}\bar{y}_{ij} - \tau_{ij} -
		b_{ij}(d_{ij} - \bar{y}_{ij}) 
		\\
		\mbox{s.t. } \; 
		& \sum_{i\in \delta^-(j)}q_{ij}\bar{y}_{ij}+\lambda_j = \sum_{k\in \delta^+(j)}\bar{y}_{jk}+\bar{v}_{j},\quad \forall j\in \mathcal{N},  \\
		&\;\bar{y}_{ij} \leq d_{ij}, \quad \quad \forall (i,j)\in \mathcal{E}, \\
		&(\bar{v}_i,\bar{y}_{ij}, -\beta \tau_{ij})\in K_{\text{exp}}, \quad \forall (i,j)\in \mathcal{E}.
\end{align*}}

\vspace{-0.2in}
\subsection{\textsf{FA} Gives an Upper Bound of the Optimal Profit}

In this section, we show that the optimal objective value of the \textsf{FA} provides an upper bound of the long-run
average profit for the \textsf{MDP} under any stationary mechanism. 
For revenue management problems in large networks where exact solutions are intractable, one common approach is to use fluid approximation to get deterministic optimization problems. \cite{gallego1994optimal,gallego1997multiproduct} introduced a fluid approximation method for finite-horizon dynamic pricing problems and proposed static pricing policies that are asymptotically optimal.
Similar approaches have been applied in many subsequent works (e.g., \citealt{cooper2002asymptotic}, \citealt{maglaras2006dynamic}, \citealt{liu2008choice}, \citealt{chen2019efficacy}), and the fluid approximation model gives an upper bound of the optimal dynamic pricing mechanism. A new contribution of our analysis is to show that the fluid approximation model is an upper bound for not only pricing but also general freight market mechanisms in an infinite-horizon setting.

\begin{theorem} \label{theorem: FA provides an upper bound to an acution}
	The optimal value of the fluid problem \textsf{FA} is an upper bound for the long-run average profit of the system under any stationary policy ${\pi}\in \Pi$. 
\end{theorem} 

It is worth noting that although the fluid model (\textsf{FA}) is constructed under the posted price mechanism, Theorem~\ref{theorem: FA provides an upper bound to an acution} holds not only for posted mechanisms but also for \emph{any} direct mechanism that is IC and IR. 
%As such, the upper bound established in Theorem~\ref{theorem: FA provides an upper bound to an acution} is useful for analyzing different types of mechanisms. 
In view of Theorem \ref{theorem: FA provides an upper bound to an acution}, the optimal objective value of \textsf{FA} provides an upper bound for the long-run average profit for the \ref{eq: MDP} under any platform mechanism. As a result, this upper bound can serve as a benchmark to evaluate the performance of any platform mechanism. 
To avoid the trivial case,
we assume in the subsequent sections that the optimal objective value of \textsf{FA} is strictly positive, because otherwise, it means the platform cannot make an operating profit and thus cannot survive in the long run. 
The proof of Theorem \ref{theorem: FA provides an upper bound to an acution} proceeds in two steps. We first show that the constraints in \textsf{FA} are necessary for any mechanism under which the system is stable. Then we show that the optimal value of the \textsf{FA} is an upper bound of the long-run average profit of any platform mechanism $(\mathcal{M}_r, \mathcal{M}_p)$. 
The full proof is relegated to the appendix.

\section{Posted Price Mechanisms} \label{section: posted price mechanism}
In this section, we study a posted price mechanism in which the platform sets prices $\mathbf{r_t}=(r_{ijt}: (i,j)\in \mathcal{E})$ to shippers and payments $\mathbf{p_t}=(p_{ijt}: (i,j)\in \mathcal{E})$ to carriers for transporting a load in period $t$. Among all the posted price mechanisms, we consider a \emph{static pricing} mechanism, where the platform offers a fixed price to the carriers and charges a fixed price to the shippers.  
We show that with a proper choice of the fixed prices, the static posted price mechanism is asymptotically optimal under an asymptotic scaling regime. 
%The term “open-loop” implies that the prices do not depend on the state of the system, which is different from a “closed-loop” policy that dynamically changes the prices in response to the system state. 
The static pricing mechanism is easy to implement in practice, and our results provide further theoretical support for its effectiveness. %The static posted price mechanism will also be used as a baseline for the analysis of other mechanisms in the subsequent sections.

We consider the following asymptotic regime. Consider a sequence of problem instances $\{\textsf{MDP}^{\theta}\}$ with scaling parameter $\theta \in \{1, 2, ...\}$. In the instance $\textsf{MDP}^{\theta}$, the arrival rates of shipper demands and carriers are equal to $\theta \mathbf{d}$ and $\theta \boldsymbol{\lambda}$, respectively. In other words, the scaling factor $\theta$ can be viewed as a measure of the system size. \rr{Let ($\mathbf{d}^*, \mathbf{\bar{y}}^*,\mathbf{\bar{v}}^*$) be an optimal solution to the \textsf{FA} and let $\gamma^{\textsf{FA}}$ denote the optimal objective value of the \textsf{FA}. Given an optimal solution ($\mathbf{d}^*, \mathbf{\bar{y}}^*,\mathbf{\bar{v}}^*$), the optimal total in-flow of carriers to each node $j$ is given by $\bar{\lambda}^*_j = \sum_{i \in \delta^-(j)} q_{ij} \bar{y}^*_{ij} + \lambda_j$, and denote $\boldsymbol{\bar{\lambda}}^* = (\bar{\lambda}^*_i: i \in \mathcal{N})$. The probability vectors $\mathbf{y}^*$ and $\mathbf{v}^*$ associated with the optimal fluid solution can be derived accordingly by $y^*_{ij} = \bar{y}^*_{ij}/\bar{\lambda}^*_i$ and $v^*_{i} = \bar{v}^*_{i}/\bar{\lambda}^*_i$.} Recall in Section \ref{subsection: convex reformulation} we showed that  $\mathbf{\bar{x}}^* = \mathbf{\bar{y}}^*$ in the optimal solution to the fluid problem \eqref{new FA}. As it shall become clear later, it is sometimes more convenient to use $\mathbf{\bar{x}}^*$ (and correspondingly, $\mathbf{x}^*)$ in our analysis, and therefore we will differentiate between $\mathbf{\bar{x}}^*$ and $\mathbf{\bar{y}}^*$ (correspondingly, $\mathbf{x}^*$ and $\mathbf{y}^*$) by using their respective notations (even though they have the same value under the optimal fluid solution). Finally, let $(\mathbf{{r}}^*,\mathbf{{p}}^*)$ respectively be the prices charged to the shippers and the payments paid to the carriers by the platform corresponding to the optimal fluid solution \rr{($\mathbf{d}^*, \mathbf{\bar{y}}^*,\mathbf{\bar{v}}^*$), where $\mathbf{{r}}^*$ is determined via the inverse shipper demand function, and $\mathbf{{p}}^*$ is given by $p^*_{ij} = \frac{1}{\beta} \left[\ln \left (\frac{\bar{y}^*_{ij}}{\bar{v}^*_{i}} \right)+\alpha_{ij} \right]$ for each $(i,j) \in \mathcal{E}$.}

Our proposed static posted price mechanism, denoted as \textsf{SP}, applies the prices $(\mathbf{{r}}^*,\mathbf{{p}}^*)$ obtained from the optimal solution to the \textsf{FA} in all system states. Given a system state $(\mathbf{S}_t,\mathbf{D}_t)$ and the posted price vector $\mathbf{p}^*=(p^*_{ij}: (i,j)\in \mathcal{E})$, carriers choose to book a load among the available remaining loads following the MNL choice model. We use the superscript $[s]$ to denote the $s^{\mbox{\scriptsize th}}$ carrier that arrives to the marketplace at origin node $i$. Let $D^{[s]}_{ijt}$ denote the number of remaining loads that need to be shipped from node $i$ to node $j$ when the $s^{\mbox{\scriptsize th}}$ carrier arrives to the platform in period $t$, i.e., $D^{[s]}_{ijt} := D_{ijt}-\sum_{s'=1}^{s-1}A^{[s']}_{ijt}$. We define $\delta^{{[s]}+}_t(i):=\{j \in\mathcal{N}: D^{[s]}_{ijt}>0\}$ as the set of destinations still with remaining loads to be transported from origin $i$ when the $s^{\mbox{\scriptsize th}}$ carrier arrives. Then, the choice probabilities of the $s^{\mbox{\scriptsize th}}$ carrier can be defined as follows:
	\begin{equation*}  
		x^{[s]}_{ijt} = \frac{e^{\beta p^*_{ij} -\alpha_{ij}}}{\sum_{k\in \delta^{{[s]}+}_t(i)} e^{\beta p^*_{ik} -\alpha_{ik}}+1}, \quad \forall j\in \delta^{{[s]}+}_t(i) \,\text{ and } \,   w^{[s]}_{it} = \frac{1}{\sum_{k\in \delta^{{[s]}+}_t(i)} e^{\beta p^*_{ik} -\alpha_{ik}}+1}.
\end{equation*}
Under the static posted price mechanism \textsf{SP}, carriers choose to deliver a load among all the available loads only if their choice maximizes their utilities, and otherwise they would leave the marketplace without booking any load. Therefore, the posted price mechanism \textsf{SP} is IC and IR.

For a problem instance with scaling factor $\theta$, let $\gamma^{\textsf{SP}}(\theta)$ denote the long-run average profit under the proposed \textsf{SP} mechanism. The optimal solution to \textsf{FA}($\theta$) is \rr{($\theta\mathbf{d}^*, \mathbf{\theta\bar{y}}^*,\theta\mathbf{\bar{v}}^*$)}, and the optimal objective value of \textsf{FA}($\theta$), denoted as $\gamma^{\textsf{FA}}(\theta)$, is equal to $\theta\gamma^{\textsf{FA}}$. The following theorem establishes the asymptotic optimality of our proposed static posted price mechanism \textsf{SP}.

\begin{theorem} \label{theorem: SI_bound}
	The static posted price mechanism $(\mathbf{{r}}^*,\mathbf{{p}}^*)$ is asymptotically optimal. More specifically, we have
	$$
	\gamma^{\textsf{FA}}(\theta) - \gamma^{\textsf{SP}}(\theta) \le O(\sqrt{\theta}),
	$$
	and therefore ${\gamma^{\textsf{SP}}(\theta)}/{\gamma^{\textsf{FA}}(\theta)} = 1 - O(1/\sqrt{\theta}) \rightarrow 1$ as the scaling factor $\theta$ approaches infinity.
\end{theorem}
The proof of Theorem~\ref{theorem: SI_bound} can be found in the appendix.

\section{Hybrid Mechanisms}\label{section: hybrid mechanisms}

\rr{In addition to posted price mechanisms, auctions are also adopted in the freight industry to match shippers and carriers \citep{figliozzi2005impacts,caplice2007electronic}. Digital freight platforms have also introduced in-app bidding, especially on the carrier side \citep{uber2020}. 
	Unlike posted price mechanisms where carriers cannot negotiate prices with the platform, auction mechanisms allow the platform to collect information from carriers before it decides how to allocate loads and set payment amounts. Therefore, auction mechanisms have the potential to generate higher profit for the platform than posted price mechanisms. However, using auctions comes at the expense of longer waiting time for carriers to receive load booking confirmations, as every carrier must wait until the end of a period to learn the result of the auction; in contrast, in the posted price mechanism, carriers can confirm a load booking instantly without waiting for other carriers' bids.
}

To counter the downside of auctions, we consider a hybrid mechanism in which carriers can either book a load instantly by accepting a posted price offered by the platform or by bidding in an auction for the load. 
Such hybrid mechanisms have become increasingly popular in digital freight marketplaces (e.g., \citealt{uber2020}, \citealt{convoy2023}), as shown earlier in Figure~\ref{Figure: Uber}, which serve as an attractive strategy for platforms to balance the trade-off between platform profit and carrier waiting time.

\rr{We consider a two-stage decision process for carriers under hybrid mechanisms. More specifically, each carrier chooses a lane according to an MNL model based on the posted prices in the first stage, and chooses to submit a bid or accept the posted price for a load on the chosen lane in the second stage. Therefore, auctions are applied to each lane and each period separately. This is aligned with the auction mechanisms used by freight platforms in practice, which would greatly simplify the implementation of the carrier-side auctions since it is well known in the mechanism design literature that multi-item auctions are notoriously difficult to analyze.}

\rr{Before formally defining the hybrid mechanism, we first introduce the auction mechanism that is used in the second stage. We use a sealed-bid uniform price auction with a reserve price for each lane and each period.	Specifically, for each lane $(i,j)$ in period $t$, loads are allocated to the lowest $D_{ijt}$ bids below the reserve price, where $D_{ijt}$ is the demand on the lane in this period. Each winning carrier is paid the lower of the $(D_{ijt}+1)^{\mbox{\scriptsize th}}$ lowest bid and the reserve price, with ties broken randomly. This auction format is simple to implement and intuitive. In practice, freight platforms typically use discriminatory auctions where carriers' payments are equal to their own bids rather than uniform price auctions. However, by the revenue equivalence principle, the two auction mechanisms are essentially equivalent, because given the same reserve price, they generate the same allocation outcomes and expected carrier payments (Proposition \ref{prop:p = a and psi}).
}

\rr{We now formally define our second-stage uniform price auction mechanism. At the beginning of time period $t$, each carrier \rr{$s \in \mathcal{S}_{it}$ first chooses a lane $(i,j)\in \mathcal{E}$ and then submits a bid ${C}^s_{ijt}$ for loads on lane $(i,j)$.}
	After receiving bids from the carriers, the platform makes allocation decisions using the following optimization problem and a reserve price $\xi_{ij}$:
	\begin{align}
		\mathcal{J}_{ij}(\mathbf{C}_{ijt}) := \min_{\mathbf{A}_{ijt},Y^0_{ijt}} &\sum_{s\in\tilde{\mathcal{S}}_{ijt}}C^s_{ijt} A^s_{ijt}(\mathbf{C}_{ijt})+\xi_{ij} Y^0_{ijt}(\mathbf{C}_{ijt}) \label{obj: decomposed auction allocation problem}\\
		\mbox{s.t. }  
		& 0 \le A^s_{ijt}(\mathbf{C}_{ijt})\le 1, \qquad\forall s\in \mathcal{\tilde S}_{ijt},\nonumber \\
		&\sum_{s\in\mathcal{\tilde S}_{ijt}} A^s_{ijt}(\mathbf{C}_{ijt})+Y^0_{ijt}(\mathbf{C}_{ijt})= D_{ijt},\nonumber \\
		&Y^0_{ijt}(\mathbf{C}_{ijt})\ge 0. \nonumber
	\end{align}
	\noindent A few remarks are in order. 
	Due to the network structure and the multi-period dynamics, in the optimal auction, the allocation of loads on a given lane may depend on bids on other lanes.
	However, in Eq \eqref{obj: decomposed auction allocation problem}, the allocation decisions $\mathbf{A}_{ijt}$ are determined separately on each lane and only depend on the bidding information $\mathbf{C}_{ijt}$ submitted by the carriers in $\mathcal{\tilde S}_{ijt}$, where $\tilde{\mathcal{S}}_{ijt}$ denotes the set of carriers in $\mathcal{S}_{it}$ who have chosen lane $(i,j)$. The variable $Y^0_{ijt}$ in the objective function \eqref{obj: decomposed auction allocation problem} represents the number of ``dummy'' bidders who bid at the reserve price $\xi_{ij}$. This ensures that carriers whose submitted bids are higher than the reserve price will not receive any load allocation. The first constraint requires that each carrier can be allocated at most one load. Notice that Eq \eqref{obj: decomposed auction allocation problem} is an assignment problem, so there always exists an integral optimal solution.
}

\rr{After the allocation decisions are set, the payments to carriers are determined as follows.
	Let $\mathbf{A}^*_{ijt}(\mathbf{C}_{ijt})$ denote the allocation of loads to carriers in $\mathcal{\tilde S}_{ijt}$ in the optimal solution to $\mathcal{J}_{ij}(\mathbf{C}_{ijt})$. Notice that such an optimal solution always exists because the objective function \eqref{obj: decomposed auction allocation problem} is bounded from above and there always exists a feasible solution ($\mathbf{A}_{ijt}(\mathbf{C}_{ijt})=\mathbf{0},{Y}^0_{ijt}(\mathbf{C}_{ijt})={D}_{ijt}$) to the above optimization problem. With the allocation rule $\mathbf{A}^*_{ijt}(\mathbf{C}_{ijt})$, the payment to a carrier $s$ in $\mathcal{\tilde S}_{ijt}$, denoted as $P^{s}_{ijt}(\mathbf{C}_{ijt})$, is given by the following payment rule:
	\begin{equation} \label{M-VCG: payment rule}
		P^{s}_{ijt}(\mathbf{C}_{ijt}) = C^s_{ijt} {A}^{s*}_{ijt}(\mathbf{C}_{ijt}) + \mathcal{J}_{ij}(\mathbf{C}^{-s}_{ijt}) - \mathcal{J}_{ij}(\mathbf{C}_{ijt}), \quad \forall s\in\mathcal{\tilde S}_{ijt},\forall ( i,j)\in\mathcal{E}.
	\end{equation}
	}
\rr{
We now elaborate how the allocation decisions obtained from optimization problem \eqref{obj: decomposed auction allocation problem} and the payment rule in \eqref{M-VCG: payment rule} characterize a uniform price auction, and we provide an example in Appendix \ref{appendix: uniform price auction example}. To simplify notations, we  drop the $ijt$ indices and consider a given lane with $|\tilde{\mathcal{S}}| := S$ carriers competing for $D$ loads, and the reserve price is equal to $\xi$. Let $\tilde{\mathcal{W}} \subseteq \tilde{\mathcal{S}}$ and $\tilde{\mathcal{L}} \subseteq \tilde{\mathcal{S}}$ be the subset of winning and losing bidders in the auction as prescribed by \eqref{obj: decomposed auction allocation problem} and \eqref{M-VCG: payment rule}. Let $\mathbf{C}$ denote the carriers' true opportunity cost vector, and we denote the $i^{\mbox{\scriptsize th}}$ lowest bid as $C^{(i)}$. First, we consider a losing bidder $j \in \tilde{\mathcal{L}}$. It is clear that $\mathcal{J}(\mathbf{C}^{-j}) = \mathcal{J}(\mathbf{C})$, since a losing bidder who does not receive any load allocation will not affect the auction outcome if he is excluded from the auction. Then from \eqref{M-VCG: payment rule}, we have $P^j = 0$ for any $j \in \tilde{\mathcal{L}}$. This is consistent with the outcome of a uniform price auction in which all losing bidders have zero payment. Next, we consider a winning bidder $i \in \tilde{\mathcal{W}}$ with $A^{i*}(\mathbf{C}) = 1$. If $\xi \ge C^{(D+1)}$, then it is easy to see that 
$\mathcal{J}(\mathbf{C}) = \sum_{k=1}^D C^{(k)}$. The load originally allocated to bidder $i$ will be allocated to the bidder with opportunity cost $C^{(D+1)}$ when $i$ is removed from the auction, and $\mathcal{J}(\mathbf{C}^{-i}) =\mathcal{J}(\mathbf{C}) - C^{i} + C^{(D+1)}$. Therefore, $P^i = C^i + (\mathcal{J}(\mathbf{C}) - C^{i} + C^{(D+1)}) - \mathcal{J}(\mathbf{C}) = C^{(D+1)}$. On the other hand, if $\xi < C^{(D+1)}$, then we have $Y^{0*}(\mathbf{C}^{-i}) = Y^{0*}(\mathbf{C}) + 1$. In other words, removing the winning bidder $i$ from the auction would make his load allocated to a ``dummy'' bidder who bids at $\xi$. Then it is easy to check that $\mathcal{J}(\mathbf{C}^{-i}) =\mathcal{J}(\mathbf{C}) - C^{i} + \xi$. Therefore, $P^i = C^i + (\mathcal{J}(\mathbf{C}) - C^{i} + \xi) - \mathcal{J}(\mathbf{C}) = \xi$. Combining the above two cases, we have $P^i = \min (\xi, C^{(D+1)})$, which is exactly the payment outcome under a uniform price auction.
}

To fully characterize the uniform price auction, we still need to specify the value of the reserve prices. 
Introducing reserve prices may improve the platform's expected profit, since a reserve price imposes an upper bound on the payment to the carriers. For our multi-period multi-location marketplace model, however, a carefully chosen reserve price must also take into account how the allocation decisions in one lane will affect the states of other lanes in the future. We propose the following uniform price auction, denoted as \textsf{AUC}, with a reserve price $\xi^*_{ij}$ for each origin-destination pair $(i,j) \in \mathcal{E}$: 
\begin{equation} \label{eq: xi^* def in psi}
	\xi_{ij}^* = \max\{\psi_{ij}^{-1}(b_{ij}), p^*_{ij}\},
\end{equation} 
where $\psi_{ij}(\cdot)$ is the virtual cost function defined in Eq~\eqref{eq:psi-def}.

\rr{Now we can define our hybrid mechanism based on the static posted price mechanism \textsf{SP} and the uniform price auction mechanism \textsf{AUC}. One challenge for analyzing hybrid mechanisms is that, even if both \textsf{SP} and \textsf{AUC} are IC and IR, the hybrid mechanism that is a combination of these two mechanisms may not be IC in general.}   To see this, consider a carrier whose true opportunity cost is lower than the posted price. If the carrier reports his opportunity cost truthfully, it will be assigned a load instantly under the \textsf{SP} mechanism with a payment equal to the posted price. 
However, if it turns out that the payout from the auction (which is determined at the end of this period) is higher than the posted price, the carrier may be better off by reporting untruthfully in order to join the auction.

We propose a hybrid mechanism \textsf{HYB} as follows. Similar to the \textsf{SP} mechanism, the \textsf{HYB} mechanism sets the same shipper-side price $\mathbf{r^*}$ as in Section~\ref{section: posted price mechanism}. On the carrier side, consider a carrier who arrives in the marketplace \rr{at node $i$ and chooses to deliver a load} from node $i$ to node $j$ in period $t$. If the submitted bid is less than or equal to the posted price $p^*_{ij}$, then this carrier is assigned a load immediately and the platform guarantees that the payment that carrier would receive is at least $p^*_{ij}$, with the exact payment amount to be determined at a later time. On the other hand, if the submitted bid is higher than $p^*_{ij}$, then this carrier will wait to join an auction with the result to be determined at a later time. An auction will be conducted among all the available carriers if the total number of carriers whose submitted bid does not exceed $p^*_{ij}$ is no more than the demand $D_{ijt}$, and the format of the auction is the uniform price auction \textsf{AUC} with reserve price $\xi^*_{ij}$ defined in \rr{Eq~\eqref{eq: xi^* def in psi}}. More specifically, if the number of carriers who have confirmed a load allocation under \textsf{SP} is smaller than $D_{ijt}$, then the \textsf{AUC} auction will be conducted at the end of the period. If all the loads have already been booked (under \textsf{SP}) upon a carrier's arrival and the bid of this newly arrived carrier is no more than $p^*_{ij}$, the platform makes payment $p^*_{ij}$ to those carriers who have received a load allocation (under \textsf{SP}) and all future carriers \rr{originating from node $i$ can no longer choose the lane $(i,j)$ and will choose another available lane according to the MNL model}. Finally, if all the loads have already been booked (under \textsf{SP}) and there is no additional carrier with a submitted bid smaller than or equal to $p^*_{ij}$ until the end of the period, then the platform conducts the \textsf{AUC} auction and makes payment to those carriers who have confirmed a load allocation (under \textsf{SP}) according to the \textsf{AUC} payment rule.

We now formally define the carrier-side allocation rule and the payment rule under the hybrid mechanism \textsf{HYB}. Let $X^{\textsf{SP}}_{ijt}(\mathbf{C}_{ijt})$ denote the number of carriers who would choose to deliver a load from node $i$ to node $j$ in period $t$ under \textsf{SP} with posted price $p^*_{ij}$ when the opportunity cost vector submitted by the carriers is $\mathbf{C}_{ijt}$. The allocation rule of the \textsf{HYB} mechanism is defined as 
\begin{equation} \label{def: HYB allocation rule}
	A^{s,\textsf{HYB}}_{ijt}(\mathbf{C}_{ijt}) = \begin{cases}
		{A}^{s,\textsf{SP}}_{ijt}(\mathbf{C}_{ijt}), \text{ if }X^{\textsf{SP}}_{ijt}(\mathbf{C}_{ijt}) > D_{ijt},\\
		{A}^{s,\textsf{AUC}}_{ijt}(\mathbf{C}_{ijt}),\text{ otherwise}\\
	\end{cases}
\end{equation}
where $A^{s,\textsf{SP}}_{ijt}(\mathbf{C}_{ijt})$ and ${A}^{s,\textsf{AUC}}_{ijt}(\mathbf{C}_{ijt})$ represent the allocations under the \textsf{SP} and \textsf{AUC} mechanisms defined in Section \ref{section: posted price mechanism} and \rr{Eq~\eqref{obj: decomposed auction allocation problem}}, respectively. The payment rule of \textsf{HYB} is defined as
\begin{equation} \label{def: HYB payment rule}
	P^{s,\textsf{HYB}}_{ijt}(\mathbf{C}_{ijt}) = \begin{cases}
		p^*_{ij}{A}^{s,\textsf{HYB}}_{ijt}(\mathbf{C}_{ijt}), \text{ if }X^{\textsf{SP}}_{ijt}(\mathbf{C}_{ijt}) > D_{ijt},\\
		{C}^s_{ijt}{A}^{s,\textsf{HYB}}_{ijt}(\mathbf{C}_{ijt}) + \mathcal{J}_{ij}(\mathbf{C}^{-s}_{ijt}) - \mathcal{J}_{ij}(\mathbf{C}_{ijt}),\text{ otherwise}\\
	\end{cases} 
\end{equation}
where $\mathcal{J}_{ij}(\cdot)$ represents the optimal objective value of the allocation problem in \textsf{AUC} as defined in Eq~\eqref{obj: decomposed auction allocation problem}.

A few remarks are in order. First, we notice that when $X^{\textsf{SP}}_{ijt}(\mathbf{C}_{ijt}) > D_{ijt}$, all the loads are allocated under \textsf{SP} and the payment to each carrier is $P^{s,\textsf{HYB}}_{ijt}(\mathbf{C}_{ijt})=p^*_{ij}$. Otherwise, the payment under the \textsf{HYB} mechanism is given by the payment under \textsf{AUC}, $P^{s,\textsf{HYB}}_{ijt}(\mathbf{C}_{ijt})=P^{s,\textsf{AUC}}_{ijt}(\mathbf{C}_{ijt})$. It is easy to see that $P^{s,\textsf{AUC}}_{ijt}(\mathbf{C}_{ijt})=\min({C}^{(D_{ijt}+1)}_{ijt}, \xi_{ij}^*)$, where ${C}^{(D_{ijt}+1)}_{ijt}$ is the $(D_{ijt}+1)^{\scriptsize \mbox{th}}$ lowest opportunity cost in the bid vector $\mathbf{C}_{ijt}$. Since in this case we have $X^{\textsf{SP}}_{ijt}(\mathbf{C}_{ijt}) \le D_{ijt}$, it then implies that ${C}^{(D_{ijt}+1)}_{ijt} \ge p^*_{ij}$, and hence $P^{s,\textsf{AUC}}_{ijt}(\mathbf{C}_{ijt}) \ge p^*_{ij}$ since $\xi^*_{ij} \ge p^*_{ij}$ by \eqref{eq: xi^* def in psi}. Therefore, the payment to a carrier who receives a load allocation under \textsf{HYB} is at least the posted price $p^*_{ij}$, i.e., $P^{s,\textsf{HYB}}_{ijt}(\mathbf{C}_{ijt})\ge p^*_{ij}$ when ${A}^{s,\textsf{HYB}}_{ijt}(\mathbf{C}_{ijt})=1$. Intuitively, this ensures that a carrier whose true opportunity cost is no more than $p^*_{ij}$ does not have the incentive to bid untruthfully, since he will be able to receive a higher payment based on the auction outcome in case the carrier supply is insufficient. As shown in Lemma \ref{lemma: HYB is IC and IR} below, the hybrid mechanism \textsf{HYB} is IC and IR. 

\begin{lemma}\label{lemma: HYB is IC and IR}
	The \textsf{HYB} mechanism is IC and IR.
\end{lemma}

We next present our main result in this section. \rr{We show that \textsf{HYB} outperforms the static posted price mechanism \textsf{SP} in terms of maximizing the long-run average profit}.  Together with Theorem \ref{theorem: SI_bound}, this immediately implies the asymptotic optimality of \textsf{HYB}.

\begin{theorem} \label{theorem: asymptotically optimal HYB auction}
	The objective values of the static posted price mechanism and the hybrid mechanism satisfy
	\begin{equation} \label{ineq Thm 4: HYB asymptotic optimal}
		\gamma^{\textsf{SP}} \le \gamma^{\textsf{HYB}}.  
	\end{equation}
\end{theorem}

\rr{While the detailed proof of Theorem \ref{theorem: asymptotically optimal HYB auction} is relegated to the appendix, we provide a few remarks about the intuition behind this result and a road map for the proof. We compare the long-run average profits between \textsf{SP} and \textsf{HYB} by the coupling method. We show by induction that there exists a coupling of available carriers $(\widetilde{S}^{\textsf{SP}}_{ijt},\widetilde{S}^{\textsf{HYB}}_{ijt})$ such that the number of available carriers under \textsf{SP} is no more than that under \textsf{HYB} for each lane and in each period, i.e., $\widetilde{S}^{\textsf{SP}}_{ijt}\le \widetilde{S}^{\textsf{HYB}}_{ijt}$ for all $(i,j)\in\mathcal{E}$ and for all $t$. The intuition is that more carriers actually transport a load under \textsf{HYB} compared to under \textsf{SP}, because \textsf{HYB} allows carriers to receive a higher payment than the posted price. With this coupling of carriers, we then compare the cost under the two mechanisms. For a specific lane, if there are enough carriers to take all the loads under \textsf{SP}, then there must be enough carriers who accept the posted price under \textsf{HYB} in the aforementioned coupling. Therefore, the costs under \textsf{SP} and \textsf{HYB} are identical in this case. On the other hand, if the number of carriers is fewer than the demand under \textsf{SP}, then one can expect that the penalty is higher under \textsf{SP} and the payment is higher under \textsf{HYB}. In this case, we construct an optimization problem for each lane such that \textsf{HYB} is an optimal solution to that problem, while \textsf{SP} provides a feasible solution to that problem, and therefore \textsf{SP} performs worse than \textsf{HYB}.}

\rr{Next, we bound the gap between posted price and hybrid mechanisms in the asymptotic regime defined in Section~\ref{section: posted price mechanism}, in which
the arrival rates of loads and carriers are scaled by the same factor $\theta$.
Theorem~\ref{theorem: SI_bound} implies that the gap between the posted price mechanism and the optimal mechanism is no greater than $O(\sqrt{\theta})$. We show that this bound is tight.
\begin{theorem}\label{theorem: gap between posted price and auction}
	There exists a problem instance such that
	\[
	\gamma^{\textsf{HYB}}(\theta) - \gamma^{\textsf{SP}}(\theta) \geq \Omega(\sqrt{\theta}),
	\] 
	and therefore ${\gamma^{\textsf{SP}}(\theta)}/{\gamma^{\textsf{HYB}}(\theta)} = 1 - \Omega(1/\sqrt{\theta})$.
\end{theorem}
The above result shows that although both the posted price and hybrid mechanisms are asymptotically optimal as the scaling factor $\theta$ approaches infinity, hybrid mechanisms can be especially beneficial in markets with low demand and few carriers.
}

\rr{In what follows, we present some key ideas of the proof. The lower bound between \textsf{HYB} and \textsf{SP} is derived from a problem instance characterized by perfect symmetry in network structure and problem parameters, where each node is connected to $k$ destinations and all the model parameters are identical across nodes and origin-destination pairs. Using a coupling method, we establish a lower bound on the performance gap between \textsf{HYB} and \textsf{SP} that is proportional to the expectation of the minimum of two random variables: (a) the unmet demand from \textsf{SP} carriers on each lane, and (b) the number of carriers whose opportunity cost exceeds the posted price but is less than the reserve price on each lane. We denote this minimum as $D^{\textsf{HYB-SP}}$, representing the demand that can be fulfilled by the platform in \textsf{HYB} but incurs a penalty cost in \textsf{SP}. To bound $D^{\textsf{HYB-SP}}$, we apply Markov's inequality to show that the expectation of $D^{\textsf{HYB-SP}}$ is greater than $\sqrt{\theta}$  multiplied by the tail probability $\Pr(D^{\textsf{HYB-SP}} \geq \sqrt{\theta})$. We then utilize the Berry-Esséen inequality, a well-known error estimation method in the context of the Central Limit Theorem, to derive the asymptotic behavior of this tail probability, which is shown to be greater than a $\text{constant} + O(1/\sqrt{\theta})$. Therefore, the performance gap between \textsf{HYB} and \textsf{SP} is lower bounded by $ \sqrt{\theta} \cdot \text{constant} + O(1)$, which is clearly $\Omega(\sqrt{\theta})$.
}

\rr{
The construction and the analysis of the aforementioned lower bound instance also provide valuable insights into the types of problem instances where the hybrid mechanism can significantly outperform the posted price mechanism.  In particular, the condition $\Ex \left[D^{\textsf{HYB-SP}} \right] \ge \Omega(\sqrt{\theta})$ holds, indicating that the hybrid mechanism is especially advantageous and should be prioritized, when the following condition holds 
$$
 \beta b-\alpha+\ln k  + \ln([\lambda/d-1]^+) - \frac{1}{[\lambda/d-1]^+}> 1. 
$$ 
In the above inequality, $b$ is the penalty cost, $k$ is the degree of each node, $\alpha$ is carriers' average cost for transporting a load on each lane, $\lambda$ is the external carrier arrival rate, and $d$ is the customer demand date. Therefore, besides the obvious market size factor, the hybrid mechanism should be more preferable if the penalty cost of unsatisfied demand is high, the node degree is high, the carriers' average cost is low, or the ratio of carrier supply to the customer demand is high.
}

\section{Numerical Studies} \label{section: case study}
In this section, we conduct a case study to provide further insights regarding the performance of different mechanisms. In our numerical studies, we use the freight transportation data \citep{census2017} and the national trucking rates from DAT Freight \& Analytics \citep{dat2022} to calibrate our model parameters. The 2017 federal government data include average delivery miles and volumes (tons) for each state-level O-D pair, and the 2022 national flatbed rates include regional levels rates. The details of data sources and parameter estimation are relegated to the appendix.

In our numerical studies, we relax the travel time assumption in the model so that the travel times of different lanes are heterogeneous and not necessarily equal to one period. 
Our data set has information about the distance (average miles) for each lane (O-D pair). We assume each period is equal to one day. To obtain the number of transportation periods needed for each lane, we divide the average miles for each lane by 500 miles, which is about the maximum distance that a truckload driver can make under federal regulation, and round up the value to the nearest integer. Therefore, in our case study, the travel times of different lanes are heterogeneous and not necessarily equal to one period. Next, we calculate the daily demand rate $d_{ij}$ for each O-D pair $(i,j)$ as follows:
\[
\text{$d_{ij}$ = (delivery volumes on lane $(i,j)$ per year/365)/Container Volume} \times \text{Market Share},
\]	 
where the container volume is set to 20 tons and the market share of a platform is set \rr{around} $1.0\%$ (the market share of Uber Freight is approximately $1.0\%$ in the Transportation \& Fleet Management sector\footnotemark\footnotetext{\url{https://enlyft.com/tech/products/uber-freight}, accessed February 10, 2025.}). In our numerical studies, we exclude lanes with extremely small daily demand rates (i.e., smaller than 0.2). Our final data set includes the daily demand rate, the average miles, and the transportation period for \rr{1074} lanes in 48 states within the United States.

We use simulation to generate independent sample paths, each with a total number of $T = 1000$ time periods. For each sample path, the first $T_0 = 200$ time periods are discarded and the performance of the system under a given mechanism (e.g., average cost) is evaluated against the remaining $T - T_0$ time periods. Furthermore, we set the number of carriers at the beginning of period 1 as \rr{$S_{i1}=\ceil{\bar{\lambda}^*_{i}\cdot \text{Market Share}}$ for each node $i$}, where \text{Market Share} is the percentage of loads transacted on a focal platform and $\bar{\lambda}_{i}^*$ is the optimal solution to the fluid problem \textsf{FA}. This helps to reduce the number of iterations to reach the stationary distribution. \rr{In our numerical simulation, model parameters are estimated based on real freight data from Census Bureau and DAT Freight \& Analytics, which provide U.S. mode data that includes information on annual shipment volumes and average mile per shipment per each geographical origin-destination pair, and national flatbed rates. Details about our data source, parameter estimation process, and model calibration can be found in Appendix \ref{appendix: data}.}

\subsection{Cost Gap Ratio and Booking Channel} \label{subsection: cost gap and AWT}

\rr{Our first set of numerical experiments compares the performance of the static posted price mechanism \textsf{SP} and the hybrid mechanism \textsf{HYB}. In our two-stage model, both  \textsf{SP} and \textsf{HYB} share the same shipper-side mechanism, which uses the same static price $\mathbf{r}^*$ obtained from the fluid model. As a result, profit maximization is essentially equivalent to cost minimization when evaluating different mechanisms. Therefore, we mainly focus on the carrier-side mechanisms to compare their performance from a cost minimization perspective. }

The first performance metric that we consider is the \emph{cost gap ratio}. More specifically, the \emph{cost gap ratio} of a given policy $\pi\in \Pi$ is defined as 
$$
\frac{\kappa^{\pi}-\kappa^{\textsf{FA}}}{\kappa^{\textsf{FA}}}.
$$ 
In the above cost gap ratio, $\kappa^{\pi}$ denotes the long-run average cost incurred by the platform under policy $\pi$: 
\begin{equation} \label{def: kappa_ij under pi}
	\kappa^{\pi}:=\sum_{(i,j)\in\mathcal{E}}\Ex[P^{\pi}_{ij}+b_{ij}(D_{ij}-Y^{\pi}_{ij})],
\end{equation} 
which consists of payments made to the carriers on the platform and penalty costs incurred due to unsatisfied demand (if any).

In addition to the cost gap ratio, we are also interested in the \emph{\textsf{SP} ratio}, which captures carriers' booking channel selection behavior. More specifically, the \emph{\textsf{SP} ratio} under a given policy is defined as the percentage of carriers who confirmed a load booking immediately upon their arrival through posted pricing, among all carriers who delivered a load. Under the posted price mechanism, all carriers can confirm a load instantly and the \textsf{SP} ratio is equal to one. %In contrast, under the auction mechanism, all carriers have to wait until the end of a period for load confirmation, and the \textsf{SP} ratio is equal to zero. 
For the hybrid mechanism, the \textsf{SP} ratio is somewhere between zero and one, which reflects the proportion of carriers who do not need to wait and are able to book a load instantly upon arrival.

Table \ref{Table: simulation results with the basic setting} summarizes the cost gap ratio and the \textsf{SP} ratio under \textsf{SP} \rr{and \textsf{HYB}}. We observe that the gap between the simulated long-run average cost under \rr{both} policies and the fluid bound decreases as the platform's market share (and hence the system size) becomes larger, which is consistent with our findings in Theorems \ref{theorem: SI_bound}-\ref{theorem: asymptotically optimal HYB auction}. Moreover, our numerical results suggest that the \textsf{SP} ratio under \textsf{HYB} increases and becomes closer to one as the market size grows. This implies that most carriers can confirm a booking instantly and do not have to wait for load confirmation under the \textsf{HYB} mechanism. In view of this, the \textsf{HYB} mechanism can be an attractive alternative for platforms that care about both their cost performance and the carriers' waiting time experience.

\begin{table}[!htb]
	\centering
    \rr{
	\begin{tabular}{c|cc|ccc}
		\toprule
		&\multicolumn{2}{c|}{Cost Gap Ratio($\%$)} &\textsf{SP} Ratio (\%) \\\hline
		\text{Market Share}&\textsf{SP}	&\textsf{HYB}  & \textsf{HYB}\\ 
		0.1\%	&	37.21 	&	14.05 	&	87.30 	\\
		0.5\%	&	24.39 	&	7.37 	&	92.73 	\\
		1.0\%	&	19.55 	&	5.27 	&	94.46 	\\
		5.0\%	&	10.10 	&	1.55 	&	97.14 	\\
		\bottomrule
	\end{tabular}
    }
	\caption{\rr{Cost Gap Ratio under \textsf{SP} and \textsf{HYB}.} } 
	\label{Table: simulation results with the basic setting}
\end{table}

\vspace{-0.2in}
To gain further insights into the performance of the different policies, we break down the long-run average cost into two components, the payment made to the carriers in the marketplace and the penalty incurred (or, payment made to third-party companies) due to the excess demand. The \emph{cost ratio}, \emph{payment ratio}, and \emph{penalty ratio} of a given policy $\pi\in \Pi$ are respectively defined as 
$$
\frac{\kappa^{\pi}}{\kappa^{\textsf{FA}}}, \text{ } \frac{\sum^T_{t= T_0+1}\sum_{(i,j)
		\in \mathcal{E}}P_{ijt}^{\pi}}{(T-T_0)\kappa^{\textsf{FA}}}, \text{ } \text{and } \frac{\kappa^{\pi}}{\kappa^{\textsf{FA}}} - \frac{\sum^T_{t= T_0+1}\sum_{(i,j)
		\in \mathcal{E}}P_{ijt}^{\pi}}{(T-T_0)\kappa^{\textsf{FA}}}.$$ 

Table \ref{Table: cost comparison for } summarizes the ratios between the total cost and the decomposed cost components relative to the long-run average cost of \textsf{FA} under policies \textsf{SP} and \textsf{HYB}. First, we observe that \textsf{SP} incurs a higher average total cost ratio. Second, we observe that \textsf{HYB} makes a higher average payment to the carriers in general. This is intuitive because \textsf{HYB} offers a higher payment than that under \textsf{SP} by the definition of the payment rule \eqref{def: HYB payment rule}. As for the average penalty, \textsf{SP} incurs a higher penalty cost due to unsatisfied demand. Recall that the reserve price $\boldsymbol{\xi}^*$ under \textsf{HYB} is higher than the posted price $\mathbf{p}^*$. Then intuitively, \textsf{HYB} accepts more carriers to transport loads, i.e., $ \Ex[Y^{\textsf{HYB}}_{ijt}] \ge \Ex[Y^{\textsf{SP}}_{ijt}]$, which leads to a lower average penalty cost than that under \textsf{SP}.

\begin{table}[!htb]
	\centering
    \rr{
	\begin{tabular}{c|cc|cc|cc}\toprule
		&\multicolumn{2}{c|}{Cost Ratio}&\multicolumn{2}{c|}{Payment Ratio}&\multicolumn{2}{c}{Penalty Ratio}\\\hline
		\text{Market Share}&\textsf{SP}	 & \textsf{HYB}& \textsf{SP} & \textsf{HYB}&\textsf{SP}& \textsf{HYB}\\
		0.1\%	&	1.37 	&	1.14 	&	0.63 	&	0.86 	&	0.75 	&	0.28 	\\
		0.5\%	&	1.24 	&	1.07 	&	0.75 	&	0.92 	&	0.49 	&	0.15 	\\
		1.0\%	&	1.20 	&	1.05 	&	0.80 	&	0.95 	&	0.39 	&	0.11 	\\
		5.0\%	&	1.10 	&	1.02 	&	0.90 	&	0.98 	&	0.20 	&	0.03 	\\
		\bottomrule
	\end{tabular}
    }
	\caption{\rr{Cost decomposition under \textsf{SP} and \textsf{HYB}.}}\label{Table: cost comparison for }
\end{table}

\subsection{Robustness Checks} \label{subsection: robustness checks}
In this section, we conduct additional numerical experiments to test the robustness of the insights obtained from our earlier results.  
Table \ref{Table: simulation results with b=10} summarizes how the performance of the different policies change with respect to the penalty cost parameter. In our simulation, we fix the platform's market share as $0.5\%$. All the other parameters remain the same as those in Section \ref{subsection: cost gap and AWT} except the penalty cost parameter. We vary the penalty cost parameter such that $b_{ij}/p_{ij} \in \{1.25,1.50,1.75,2.00\}$ for each O-D pair $(i,j) \in \mathcal{E}$, where $p_{ij}$ is an estimated shipping cost calculated from the data. From Table \ref{Table: simulation results with b=10}, we observe that the cost gap ratios of both policies increase as the penalty cost parameter becomes larger, with \textsf{SP} having a much more significant increase compared with \textsf{HYB}. Intuitively, the posted price mechanism has  a larger amount of excess demand and is therefore more affected by the change in the penalty cost parameter. 

\begin{table}[!htb]
	\centering
    \rr{
	\begin{tabular}{c|ccccccc}\toprule
		&\multicolumn{2}{c}{Cost Gap Ratio($\%$)}\\\hline
		$b_{ij}/p_{ij}$&\textsf{SP}	& \textsf{HYB} \\ 
		1.25  &6.18  &1.92 \\
		1.50  &12.30  &3.74 \\
		1.75  &18.23  &5.55 \\
		2.00  &24.39  &7.37 \\
		\bottomrule
	\end{tabular}
    }
	\caption{\rr{Impact of penalty cost parameter $\boldsymbol{b}$ on cost gap ratio.} }\label{Table: simulation results with b=10}
\end{table}

In addition to the penalty cost parameter, we have also conducted additional numerical experiments to investigate the impact of the probability that a carrier will stay in the marketplace after completing a load. The market share is fixed at $0.5\%$, and all the other parameters remain the same as those in Section \ref{subsection: cost gap and AWT} except the staying probability parameter. Note that \rr{$q_{ij}$} represents the probability that a carrier would stay in the marketplace after finishing a load delivery on lane $(i,j)$. In our numerical studies, we assume that \rr{$q_{ij} = q \in \{0, 0.2,0.4, 0.6\}$ for all $(i,j) \in \mathcal{E}$}. It is worth noticing that changing the value of this probability would affect the exogenous inflow rate of the carriers based on our parameter estimation process. Therefore, the results below measure the overall effect of the remaining probability, with the exogenous carriers' arrival rate taken into account. As shown in Table \ref{Table: simulation results with q}, the cost gap ratios \rr{in general increase} as $q$ increases, \rr{but the differences are small: the cost gap ratio increases by less than 2\% under \textsf{SP} and 0.3\% under \textsf{HYB} when the stay probability increases from 0\% to 60\%. This demonstrates that the platform's performance is robust to changes with respect to $q$ under both mechanisms. Intuitively, the platform's costs should decrease as $q$ increases, since a larger $q$ results in more carriers in the network, which helps to reduce penalties.
However, the network structure provides carriers with multiple options, increasing the likelihood that a carrier eventually transports a load, which may partially mitigate the benefits of a higher $q$. Additionally, a larger $q$ leads to a smaller exogenous inflow in our parameter estimation process, which can potentially reduce the number of carriers and increase penalties.}

\begin{table}[!htb]
	\centering
    \rr{
	\begin{tabular}{c|cccccc}\toprule
		&\multicolumn{2}{c}{Cost Gap Ratio($\%$)}\\\hline
		$q$&\textsf{SP}	&	\textsf{HYB}\\ 
		0.00&  24.17 	&	7.44 	\\
		0.20& 24.39	&	7.37 	\\
		0.40&25.15 	&	7.59 	\\
		0.60&25.84 	&	7.67 	\\
		\bottomrule
	\end{tabular}
    }
	\caption{\rr{Impact of staying probability $\mathbf{q}$ on cost gap ratio. }}\label{Table: simulation results with q}
\end{table}

\section{Concluding Remarks} 
\label{section: conclusion}

In this paper, we study a mechanism design problem for freight marketplaces. We consider a freight platform that serves as an intermediary between shippers and carriers in a truckload transportation network and aims to maximize its long-run average profit. \rr{We develop a two-stage framework to model carriers' load choice behavior. In the first stage, 
carriers select a lane using a multinomial logit (MNL) model based on the prices posted by the platform. In the second stage, carriers decide whether to book a load on the chosen lane by either accepting the posted price immediately or submitting a bid for the load if the platform offers an auction option.} 
We have proposed and analyzed \rr{two types of mechanisms: posted price mechanisms and a hybrid of auction mechanisms and posted price mechanisms.} We show that a static posted price mechanism based on fluid approximation is asymptotically optimal when the shipper demand and the carrier supply are both large. Furthermore, we study a hybrid mechanism, in which carriers can either book a load instantly by accepting the posted price offered by the platform or join \rr{a uniform price auction} to seek higher payments. We show that the hybrid mechanism can achieve a trade-off between platform profit and carrier waiting time, and is asymptotically optimal. We also provide tight bounds between the posted price mechanism and the hybrid mechanism as a function of the scaling factor.

There are several possible directions to extend our research. First, our model assumes that the lead time of each load is one period. That is, any load arriving at the beginning of a period will expire at the end of the period and cannot be carried over to the next period. % Our simulation results show that this approach is also valid. 
It would be interesting to generalize our model and consider loads with heterogeneous, multi-period lead times (see discussion in Section \ref{section: model}). Second, our proposed static posted price mechanism uses a fixed price for each O-D pair. To improve the platform's profit, considering prices that dynamically change over time in response to the system states can be a good extension. Lastly, it would be more desirable that a platform can choose different types of mechanisms on each lane. The effect of this flexibility may offer more managerial insights in practice.

\bibliographystyle{apalike}
\bibliography{ref.bib}

\ECSwitch

\ECHead{Appendix}

\section{Notations} \label{appendix: notations}
\begin{table}[!htb]
	\begin{tabularx}{\textwidth}{l|X}\toprule
		%\textbf{Set}\\
		$\mathcal{N}$ & The set of nodes (locations)\\
		$\mathcal{E}$ & The set of arcs (lanes)\\
		$\mathcal{S}_{it}$ & The set of carriers available to deliver loads at node $i$ in period $t$\\
		$\delta^+(j) $& $\{k \in \mathcal{N}:(j,k)\in \mathcal{E}\}$, the set of outbound nodes from node $j$\\ 
		$\delta^-(j)$ & $\{i \in \mathcal{N}: (i,j)\in \mathcal{E}\}$, the set of inbound nodes to node $j$\\ \hline
		%\textbf{Random Variables}\\
		${S}_{it}$ & Number of available carriers at node $i$ in period $t$\\
		$D_{ijt}$	& Number of loads that need to be shipped on lane $(i,j)$ in period $t$	\\
		$C^s_{ijt} $ & True opportunity cost of carrier $s$ transporting a load on lane $(i,j)$ in period $t$\\
		$X_{ijt}$	& Number of carriers who would choose to book loads on lane $(i,j)$ in period $t$\\
		$Y_{ijt}$	& Number of carriers who are awarded loads on lane $(i,j)$ in period $t$\\
		$Z_{ijt}$	& \rr{Number of carriers who decide to stay in the marketplace after completing a load shipment from node $i$ to node $j$ in period $t$. } \\
        $V_{it}$  & Number of carriers in $\mathcal{S}_{it}$ who leave the marketplace at the end of period $t$ \\
		$P_{ijt}$	& Total payment to carriers who transported loads on lane $(i,j)$ in period $t$\\
		$\Lambda_{it}$ & Number of exogenous arrival of carriers at node $i$ in period $t$\\\hline
		%\textbf{Fluid Approximation}\\
		$\gamma^{\textsf{FA}}$ & Optimal objective value of \textsf{FA}\\
		$b_{ij}$	& Unit penalty cost for unsatisfied demand on lane $(i,j)$ \\
		$q_{ij}$	& 
        \rr{Probability that a carrier who just delivered a load on lane $(i,j)$ will remain in the marketplace} \\
		$\lambda_{i}$	& Exogenous arrival rate of carriers at node $i$ \\
		$d^*_{ij}$	& Optimal demand rate of loads on lane $(i,j)$ \\
		$r^*_{ij}$	& Optimal spot price for loads that need to be shipped on lane $(i,j)$ \\
		$\bar{\lambda}^*_{i}$ & Optimal total inflow of carriers at node $i$ \\
		$x^*_{ij}$	&Optimal probability that a carrier chooses to deliver a load on lane $(i,j)$ \\
		$\bar{y}^*_{ij}$	&Optimal flow of carriers who transported a load on lane $(i,j)$ \\
		$\bar{v}^*_{i}$	&Optimal flow of leaving carriers at node $i$ \\
		\hline
		%\textbf{Auction Mechanism}\\
		$P^{s}_{ijt}(\mathbf{C}_{ijt})$ & Payment to carrier $s$ on lane $(i,j)$ in period $t$ with opportunity cost vector $\mathbf{C}_{ijt}$\\
		${A}_{ijt}^s(\mathbf{C}_{ijt})$ & Load allocation for carrier $s$ on lane $(i,j)$ in period $t$ with opportunity cost vector $\mathbf{C}_{ijt}$\\
		${\xi}^*_{ij}$	&Reserve price of \textsf{AUC} on lane $(i,j)$ \\
		$\psi_{ij}({C}_{ijt}^s)$ &Virtual cost of carrier $s$ on lane $(i,j)$ in period $t$ with opportunity cost ${C}_{ijt}^s$ \\
		$\rho_{ijt}$ & Auxiliary notation used in the proofs, see Eq \eqref{def:rho_it}\\
		$\mathcal{H}_{ijt}$ & Auxiliary notation used in the proofs, see Eq \eqref{def: H_it}\\
		$\mathcal{K}_{ijt}$ & Auxiliary notation used in the proofs, see Eq \eqref{eq: K_it objective}\\
		\bottomrule
	\end{tabularx}
	\caption{List of notations. }
\end{table}

\section{Numerical Experiments: Dataset and Model Calibration} \label{appendix: data}
In this section, we provide detailed information about our dataset and the parameter estimation procedure to calibrate our model. The freight data used in our numerical case studies come from two sources, Census Bureau and DAT Freight \& Analytics. The Census Bureau provides U.S. mode data on the website (\url{https://data.census.gov/cedsci/}). By using the keyword \emph{CFSAREA2017.CF1700A20} in the search box, we obtain a table ``Geographic Area Series: Shipment Characteristics by Origin Geography by Destination Geography: 2017", which includes information on annual shipment volumes and average mile per shipment per each geographical origin-destination pair. Our second data source is DAT Freight \& Analytics, which provides national flatbed rates on their website (\url{https://www.dat.com/industry-trends/trendlines/flatbed/national-rates}). In our numerical studies, we accessed the website and retrieved the average outbound flatbed rates of five regions in the U.S. on February 28, 2022: West (\$3.10), West South (\$2.78), Mid West (\$3.46), South East (\$3.03), and North East (\$2.94) .

In our numerical studies, we assume that the true opportunity cost of a carrier follows \rr{a logistic distribution (which is due to the Gumbel error terms in the utility) with mean $\alpha_{ij}/\beta$. We set the mean opportunity cost $\alpha_{ij}/\beta$ equal to $p_{ij}$}, where $p_{ij}$ represents the average freight shipping cost on lane $(i,j)$. To calculate this average shipping cost, we utilize the data by DAT Freight \& Analytics, which provides the average regional flatbed rates per mile of five regions within the U.S \rr{(i.e., \$3.10, \$2.78, \$3.46, \$3.03, and \$2.94 as mentioned earlier)}. We term this average regional flatbed rate as the \emph{normal freight rate}. We assume that 90\% of daily demand is satisfied by normal freight rates; and the rest of the demand that cannot be served by the carriers in the marketplace is fulfilled by penalty shipping rates (e.g., third-party carriers' shipping costs) which are assumed two times more expensive than the normal freight rates. Then the average freight shipping on lane $(i,j)$ can be computed as follows:
\[
p_{ij} = \text{(Origin $i$ rate + Destination $j$ rate)/2} \times \text{Average miles}/ (2 \times 0.1+0.9),
\]
where the origin and destination rates are based on the DAT regional flatbed rate data. 

We assume that the probability that a carrier will remain in the marketplace after transporting a load is $0.2$. Finally, in our case study, we assume that the fraction of carriers that would choose to ship a load on the platform is 0.5, and the carriers' external arrival rate $\lambda_{i}$ to the marketplace at each node $i$ can be obtained by solving the flow balance Equation \eqref{FA linking constraint}.

\section{An Example of Uniform Price Auctions} \label{appendix: uniform price auction example}
\rr{Here we provide an example to illustrate how the allocation decisions obtained from optimization problem \eqref{obj: decomposed auction allocation problem} and the payment rule in \eqref{M-VCG: payment rule} characterize a uniform price auction. Consider a lane with 5 carriers competing for 3 loads, and the reserve price is equal to $\xi$. Let $C^{(1)}< C^{(2)}<\cdots< C^{(5)}$ denote the carriers' true opportunity costs. When $C^{(3)} \le \xi$, an optimal solution to problem \eqref{obj: decomposed auction allocation problem} is given by $A^{(i)*}=1$ for $i \in \{1,2,3\}$, $A^{(i)*}=0$ for $i \in \{4,5\}$, $Y^{0*}=0$, and $\mathcal{J}(\mathbf{C})=C^{(1)}+C^{(2)}+C^{(3)}$. This is exactly the allocation outcome in a uniform price auction, where the lowest three bids below the reserve price win a load. Then we consider the payment rule given by Eq \eqref{M-VCG: payment rule}. Consider a carrier who wins a load, for example, carrier 2. It is easy to see that $\mathcal{J}(\mathbf{C}^{-(2)})=C^{(1)}+C^{(3)}+ \min (C^{(4)}, \xi)$.
By Eq \eqref{M-VCG: payment rule}, the payment to carrier 2 is $P^{(2)}=C^{(2)}+(C^{(1)}+C^{(3)}+ \min (C^{(4)}, \xi))-(C^{(1)}+C^{(2)}+C^{(3)})=\min (C^{(4)}, \xi)$, which is exactly the uniform price auction payment --- winning carriers are paid the lower of the $(D+1)^{\mbox{\scriptsize th}}$ lowest bid and the reserve price. On the other hand, for carrier $i \in \{4,5\}$ who loses the auction and does not receive a load allocation, it is clear that $\mathcal{J}(\mathbf{C}^{-(i)})=\mathcal{J}(\mathbf{C})=C^{(1)}+C^{(2)}+C^{(3)}$ and $A^{(i)*}=0$, and the payment to carrier $i$ is $P^{(i)}=0$, which is exactly the outcome under a uniform price auction. When $C^{(3)} > \xi \ge C^{(2)}$, it is easy to check that $A^{(1)*}= A^{(2)*} = Y^{0*}= 1$, $A^{(3)*} = A^{(4)*} = A^{(5)*} = 0$, and $\mathcal{J}(\mathbf{C})=C^{(1)}+C^{(2)}+ \xi$. For carrier $i$ who wins a load, say carrier 2,  his payment is $P^{(2)}=C^{(2)}+(C^{(1)} + 2 \xi)-(C^{(1)}+C^{(2)}+\xi)=\xi$. For  carrier $i \in \{3,4,5\}$ who loses the auction, it's clear that $P^{(i)}=0$. The other remaining cases $C^{(2)} > \xi \ge C^{(1)}$ and $C^{(1)} > \xi$ can be analyzed in a similar vein. Therefore, the allocations and payments from \eqref{obj: decomposed auction allocation problem} and \eqref{M-VCG: payment rule} are indeed consistent with the outcome of a uniform price auction.}

\section{Proofs}  \label{appendix: proof}
\subsection*{Proofs in Section~\ref{section: model}}

\begin{proof}{\rr{Proof of Lemma \ref{Lemma: virtual function properties}.}}

    Consider a carrier $s$ who has chosen lane $(i,j)$ in the first stage. Then we have
		\begin{align*}
			& U^s_{ijt} > U^s_{ikt}, \quad \forall k \neq j \\
			\Leftrightarrow \quad  & \beta p_{ijt} - \alpha_{ij} + \epsilon^s_{ijt} > \beta p_{ikt} - \alpha_{ik} + \epsilon^s_{ikt}, \quad \forall k \neq j \\
			\Leftrightarrow \quad  & C^s_{ijt} < (p_{ijt}-p_{ikt}) + C^s_{ikt}, \quad \forall k \neq j\\
			\Leftrightarrow \quad & \epsilon_{ikt}^s<\beta p_{ijt} -\beta p_{ikt} + \alpha_{ik}- \alpha_{ij} + \epsilon^s_{ijt},\quad \forall k\neq j.
		\end{align*}
		Let ${F}_{ij}(\cdot)$ and ${f}_{ij}(\cdot)$ denote the posterior cumulative distribution function and the posterior probability density function of $C^s_{ijt}$ given the condition that carrier $s$ has chosen lane $(i,j)$. For notation simplicity, we will omit the superscript $s$. It then follows that
		\begin{align}
			{F}_{ij}(c) 
			& = \Pr \left(C_{ijt} \le c  \mid C_{ijt} < (p_{ijt}-p_{ikt}) + C_{ikt}, \, \forall k \neq j\right) \nonumber \\
			& = \frac{\Pr \left(C_{ijt} \le c,  C_{ijt} < (p_{ijt}-p_{ikt}) + C_{ikt}, \, \forall k \neq j \right)}{\Pr \left(C_{ijt} < (p_{ijt}-p_{ikt}) + C_{ikt}, \, \forall k \neq j\right)}  \nonumber\\
			& = \frac{\Pr \left(\epsilon_{i0t} \le \beta c-\alpha_{ij}+\epsilon_{ijt},  \epsilon_{ikt}<\beta p_{ijt} -\beta p_{ikt} + \alpha_{ik}- \alpha_{ij} + \epsilon_{ijt}, \, \forall k \neq j \right)}
			{\Pr \left( \epsilon_{ikt}<\beta p_{ijt} -\beta p_{ikt} + \alpha_{ik}- \alpha_{ij} + \epsilon_{ijt}, \, \forall k \neq j \right)}  \nonumber\\
			& = \frac{\int_{-\infty}^{\infty} e^{-e^{-(\beta c-\alpha_{ij}+\epsilon_{ijt})}}\prod_{k \neq j} e^{-e^{-(\beta p_{ijt} -\beta p_{ikt} + \alpha_{ik}- \alpha_{ij} + \epsilon_{ijt})}} \cdot e^{-\epsilon_{ijt}}e^{-e^{-\epsilon_{ijt}}} \,d \epsilon_{ijt}}
			{\int_{-\infty}^{\infty} \prod_{k \neq j} e^{-e^{-(\beta p_{ijt} -\beta p_{ikt} + \alpha_{ik}- \alpha_{ij} + \epsilon_{ijt})}} \cdot e^{-\epsilon_{ijt}}e^{-e^{-\epsilon_{ijt}}} \,d \epsilon_{ijt}}  \label{eq: eps Gumbel}\\
			& = \frac{\int_{-\infty}^{\infty} e^{-e^{-\epsilon_{ijt}} \left[1+e^{-(\beta c-\alpha_{ij})}+\sum_{k \neq j}e^{-(\beta p_{ijt} -\beta p_{ikt} + \alpha_{ik}- \alpha_{ij} )}\right]} \cdot e^{-\epsilon_{ijt}}\,d \epsilon_{ijt}}
			{\int_{-\infty}^{\infty} e^{-e^{-\epsilon_{ijt}} \left[1+\sum_{k \neq j}e^{-(\beta p_{ijt} -\beta p_{ikt} + \alpha_{ik}- \alpha_{ij} )}\right]} \cdot e^{-\epsilon_{ijt}}\,d \epsilon_{ijt}}  \nonumber\\
			& = \frac{\int_{0}^{\infty} e^{-u \left[1+e^{-(\beta c-\alpha_{ij})}+\sum_{k \neq j}e^{-(\beta p_{ijt} -\beta p_{ikt} + \alpha_{ik}- \alpha_{ij} )}\right]} \,d u}
			{\int_{0}^{\infty} e^{-u \left[1+\sum_{k \neq j}e^{-(\beta p_{ijt} -\beta p_{ikt} + \alpha_{ik}- \alpha_{ij} )}\right]} \,d u}\quad (\text{let }u=e^{-\epsilon_{ijt}})  \nonumber\\
			& = \frac{1+\sum_{k \neq j}e^{-(\beta p_{ijt} -\beta p_{ikt} + \alpha_{ik}- \alpha_{ij} )}}{1+e^{-(\beta c-\alpha_{ij})}+\sum_{k \neq j}e^{-(\beta p_{ijt} -\beta p_{ikt} + \alpha_{ik}- \alpha_{ij} )}}  \quad (\text{by }\int_{0}^{\infty} e^{-u A} \,d u = \frac{1}{A}) \nonumber\\
			& = \frac{\sum_{k\in \delta^+(i)} e^{\beta p_{ikt} -\alpha_{ik}}}{e^{\beta(p_{ijt}-c)}+\sum_{k\in \delta^+(i)} e^{\beta p_{ikt} -\alpha_{ik}}} \label{eq: posterior independence}
		\end{align}
		where \eqref{eq: eps Gumbel} follows from the assumption that $\epsilon_{ikt}$ and $\epsilon_{i0t}$ follow i.i.d Gumbel distribution for all $k \in \delta^+(i)$. For the posterior probability density function, we have
		\begin{equation} \label{eq:posterior density}
			{f}_{ij}(c) =F'_{ij}(c)= \frac{\beta e^{\beta(p_{ijt}-c)} \sum_{k\in \delta^+(i)} e^{\beta p_{ikt} -\alpha_{ik}}}{\left(e^{\beta(p_{ijt}-c)}+\sum_{k\in \delta^+(i)} e^{\beta p_{ikt} -\alpha_{ik}}\right)^2}
		\end{equation}
Substituting \eqref{eq: posterior independence} and \eqref{eq:posterior density} into \eqref{eq:psi-def}, we have
		\begin{equation} \label{eq: psi-result}
			{\psi}_{ij} (C_{ijt}) = C_{ijt} + \frac{1}{\beta}+\frac{1}{\beta}e^{\beta (C_{ijt}-p_{ijt})}\sum_{k\in\delta^+(i)}{e^{\beta p_{ikt}-\alpha_{ik}}}.
		\end{equation}
		(\ref{eq: psi-result}) implies that $\psi_{ij}$ is a continuous and strictly increasing function. Therefore the inverse function of $\psi_{ij}$ exists.
	\Halmos \end{proof}

\vspace{0.1in}

\begin{proof}{Proof of Proposition \ref{Prop: stationary}.}
	We prove Proposition \ref{Prop: stationary} by showing the stability of a modified system whose expected total number of carriers is greater than or equal to that of the original system. We use the Foster's theorem \citep{foster1953stochastic} to prove the stability of the new system, and we show that the existence of the stationary distribution of the new system implies the stability of the original system. Notice that it suffices to consider the stability of the carrier side since the new shipper demand $\mathbf{D}_{t}$ follows a Poisson distribution with known rate and all the loads are served either by the carriers in the marketplace or the third-party.

	Consider a modified model in which a carrier stays in the platform with probability $\hat{q}$ after completing a load shipment in each period, where $\hat{q} := \max_{(i,j)\in\mathcal{E}}{q_{ij}}$, regardless of whether the carrier hauls a load. Let $\hat{S_t}$ be the total number of carriers in the modified system at the beginning of period $t$. \rr{Let $\hat{\lambda} = \sum_i \lambda_i$ be the total arrival rate of carriers to the marketplace in each period.} The total arrival of carriers in period $t$, $\hat{\Lambda}_t$, follows a Poisson distribution with parameter $\hat{\lambda}$. The number of remaining carriers at the end of period $t$, $\hat{Z}_t$, follows a Binomial distribution with parameters $\hat{q}$ and $\hat{S}_t$. To explicitly show the dependence of $\hat{Z}_t$ on the parameter $\hat{q}$, we write $\hat{Z}_t$ as $\hat{Z}_t(\hat{q}|\hat{S_t})$ whenever necessary. Then in the modified model, the carriers available in the next period are 
	\[
	\hat{S}_{t+1}=\hat{Z}_t(\hat{q}|\hat{S_t})+\hat{\Lambda}_{t+1}.
	\]
	The one-step transition probability from state $l$ to state $k$ is given by
	\[\Pr(\hat{S}_{t+1}=k|\hat{S}_{t}=l) = \sum_{m=0}^{\min\{l,k\}}\binom{l}{m}\hat{q}^m(1-\hat{q})^{l-m}e^{-\hat{\lambda}} \frac{\hat{\lambda}^{k-m}}{(k-m)!}.\]
	Clearly, $\Ex[\sum_{i\in \mathcal{N}}S_{it}]\le \Ex[\hat{S}_{t}]$ if $\Ex[\hat{S}_{t}]$ exists, where $S_{it}$ is the number of carriers at node $i$ in period $t$ in the original model.
	%under mechanism $\pi \in \Pi$. 
	Therefore, to prove Proposition \ref{Prop: stationary}, it suffices to show the existence of $\Ex[\hat{S}_{t}]$.
	
	We use Foster's Theorem to show that a stationary distribution exists in the modified system. Let $\mathbb{Z}^+$ denote the set of non-negative integers. Define a Lyapunov function $\Phi$ as 
	\[\Phi(\hat{S}_t):= \hat{S}_t.\]
	By the definition of $\hat{S}_t$, we assume $\Phi(\hat{S}_t)\geq 0$ without loss of generality. By Foster's Theorem, the Markov chain of the modified system is positive recurrent if the Lyapunov function $\Phi$ satisfies
	\[\sum_{k=0}^{\infty} \Pr(\hat{S}_{t+1}=k|\hat{S}_{t}=l)\Phi(k)<\infty, \quad \forall l\in F,\]
	\[\sum_{k=0}^{\infty} \Pr(\hat{S}_{t+1}=k|\hat{S}_{t}=l)\Phi(k) \leq \Phi(l) -\epsilon, \quad \forall l\notin F,\]    
	for some finite set $F$ and $\epsilon>0$. 
	
	We first show that $\sum_{k=0}^{\infty} \Pr(\hat{S}_{t+1}=k|\hat{S}_{t}=l)\Phi(k)< \infty$ for any state $l\in \mathbb{Z}^+$. We have
	\begin{align*}
		\sum_{k=0}^{\infty} \Pr(\hat{S}_{t+1}=k|\hat{S}_{t}=l)\Phi(k) &= \sum_{k=0}^{\infty} \Pr(\hat{S}_{t+1}=k|\hat{S}_{t}=l)k\\
		&= \sum_{k=0}^{l} \Pr(\hat{S}_{t+1}=k|\hat{S}_{t}=l)k  + \sum_{k=l+1}^{\infty} \Pr(\hat{S}_{t+1}=k|\hat{S}_{t}=l) k.
	\end{align*}
	It is clear that $\sum_{k=0}^{l} \Pr(\hat{S}_{t+1}=k|\hat{S}_{t}=l)k$ is bounded. Consider the second term in the above equation:
	\begin{align*}
		\sum_{k=l+1}^{\infty} \Pr(\hat{S}_{t+1}=k|\hat{S}_{t}=l) k &= \sum_{k=l+1}^{\infty}\sum_{m=0}^l \binom{l}{m}\hat{q}^m(1-\hat{q})^{l-m}e^{-\hat{\lambda}}\frac{\hat{\lambda}^{k-m}}{(k-m)!}k\\&= \sum_{k=l+1}^{\infty}\sum_{m=0}^l \binom{l}{m}\hat{q}^m(1-\hat{q})^{l-m}e^{-\hat{\lambda}}\frac{\hat{\lambda}^{k-m}}{(k-m)!}(k-m+m)\\
		&        =    \sum_{m=0}^l \binom{l}{m}\hat{q}^m(1-\hat{q})^{l-m}\sum_{k=l+1}^{\infty}e^{-\hat{\lambda}}\frac{\hat{\lambda}^{k-m}}{(k-m)!}(k-m + m).
	\end{align*}
	We can get an upper bound of the above inner summation term as follows:
	\begin{align*}
		\sum_{k=l+1}^{\infty}e^{-\hat{\lambda}}\frac{\hat{\lambda}^{k-m}}{(k-m)!}(k-m + m)&\le \sum_{k=m}^{\infty}e^{-\hat{\lambda}}\frac{\hat{\lambda}^{k-m}}{(k-m)!}(k-m + m)\\
		&=\sum_{k=m}^{\infty}e^{-\hat{\lambda}}\frac{\hat{\lambda}^{k-m}}{(k-m)!}(k-m) +\sum_{k=m}^{\infty}e^{-\hat{\lambda}}\frac{\hat{\lambda}^{k-m}}{(k-m)!} m\\
		&\le\hat{\lambda}+m.
	\end{align*}
	Then, we have
	\begin{align*}
		\sum_{k=l+1}^{\infty} \Pr(\hat{S}_{t+1}=k|\hat{S}_{t}=l) k     
		&    \le    \sum_{m=0}^l \binom{l}{m}\hat{q}^m(1-\hat{q})^{l-m}(\hat{\lambda} + m)\\ &\le {\hat{\lambda}+l\hat{q}(1-\hat{q})}.
	\end{align*}
	It then immediately follows that $\sum_{k=0}^{\infty} \Pr(\hat{S}_{t+1}=k|\hat{S}_{t}=l)\Phi(k)<\infty$ for all $l\in \mathbb{Z}^+$. 
	
	We next show that there exist a finite set $F^*$ and a real positive number $\epsilon^*$ that satisfy the second condition of the Foster's theorem:
	\[
	\sum_{k=0}^{\infty} \Pr(\hat{S}_{t+1}=k|\hat{S}_{t}=l)\Phi(k) \leq \Phi(l) -\epsilon^*, \quad \forall l\notin F^*.
	\]    
	We have
	\begin{align*}
		&\quad \sum_{k=0}^{\infty} \Pr(\hat{S}_{t+1}=k|\hat{S}_{t}=l)\Phi(k) - \Phi(l)\\
		&= \sum_{k=0}^{\infty}  \Pr(\hat{S}_{t+1}=k|\hat{S}_{t}=l)k - l \\&= \sum_{k=0}^{\infty} \Pr(\hat{S}_{t+1}=k|\hat{S}_{t}=l) (k- l)\\
		&= \sum_{k=0}^{l} \Pr(\hat{S}_{t+1}=k|\hat{S}_{t}=l) (k- l)  + \sum_{k=l+1}^{\infty} \Pr(\hat{S}_{t+1}=k|\hat{S}_{t}=l) (k- l).
	\end{align*}
	Then, the condition can be expressed as:
	\[
	\sum_{k=0}^{l} \Pr(\hat{S}_{t+1}=k|\hat{S}_{t}=l) (k- l)  + \sum_{k=l+1}^{\infty} \Pr(\hat{S}_{t+1}=k|\hat{S}_{t}=l) (k- l)\le -\epsilon^*, \quad \forall l\notin F^*.
	\]
	In what follows, we will show the existence of $\epsilon^*$ and $F^*$ that satisfy the following two inequalities: 
	\begin{align}
		& \sum_{k=0}^{l} \Pr(\hat{S}_{t+1}=k|\hat{S}_{t}=l) (k - l) < -\hat{\lambda} - \epsilon^*, \quad \forall l\notin F^*, \label{ineq_1: exist eps F} \\
		& \sum_{k=l+1}^{\infty} \Pr(\hat{S}_{t+1}=k|\hat{S}_{t}=l) (k-l)\le \hat{\lambda},\quad \forall l\notin F^*. \label{ineq_2: exist eps F}
	\end{align}

	\noindent To show \eqref{ineq_1: exist eps F}, we first show 
	\begin{equation}
		\sum_{k=0}^{l} \Pr(\hat{S}_{t+1}=k|\hat{S}_{t}=l) k\le {\hat{\lambda}+l\hat{q}(1-\hat{q})} \tag{*}  \label{ineq_1: exist eps F_condition a}
	\end{equation}
	for any $l\in \mathbb{Z}^+$, and then show the existence of $\epsilon^*>0$ and $F^*$ such that 
	\begin{equation}
		\sum_{k=0}^{l} \Pr(\hat{S}_{t+1}=k|\hat{S}_{t}=l) l > l\hat{q}(1-\hat{q}) + 2\hat{\lambda} +\epsilon^*, \quad \forall l\notin F^*. \tag{**}  \label{ineq_1: exist eps F_condition b}
	\end{equation}
	Combining the above two inequality immediately leads to \eqref{ineq_1: exist eps F}.

	\noindent We first show \eqref{ineq_1: exist eps F_condition a}. We have
	\begin{align*}
		\sum_{k=0}^{l} \Pr(\hat{S}_{t+1}=k|\hat{S}_{t}=l) k &= \sum_{k=0}^{l}\sum_{m=0}^k \binom{k}{m}\hat{q}^m(1-\hat{q})^{k-m}e^{-\hat{\lambda}}\frac{\hat{\lambda}^{k-m}}{(k-m)!}k\\&= \sum_{k=0}^{l}\sum_{m=0}^k \binom{k}{m}\hat{q}^m(1-\hat{q})^{k-m}e^{-\hat{\lambda}}\frac{\hat{\lambda}^{k-m}}{(k-m)!}(k-m+m)\\
		&    =    \sum_{m=0}^k \binom{k}{m}\hat{q}^m(1-\hat{q})^{k-m}\sum_{k=0}^{l}e^{-\hat{\lambda}}\frac{\hat{\lambda}^{k-m}}{(k-m)!}(k-m + m)\\
		&    \le    \sum_{m=0}^k \binom{k}{m}\hat{q}^m(1-\hat{q})^{k-m}(\hat{\lambda} + m)\\ &\le {\hat{\lambda}+l\hat{q}(1-\hat{q})}.
	\end{align*}
	
	\noindent We next show \eqref{ineq_1: exist eps F_condition b}. Consider any $k, l\in \mathbb{Z}^+$ such that $l>k$. By the probability mass function of the Binomial distribution, we have
	\[
	\Pr(Z_t(\hat{q}|l) = l-k) = \binom{l}{l-k}\hat{q}^{l-k}(1-\hat{q})^{k}.\] 
	It then follows that
	\[\lim_{l\rightarrow \infty}\frac{\Pr(Z_t(\hat{q}|l) = l-k)}{\Pr(Z_t(\hat{q}|l+1) = l+1-k)} = \lim_{l\rightarrow \infty}\frac{l-k+1}{(l+1)\hat{q}} =\frac{1}{\hat{q}},
	\] 
	which implies that $\lim_{l \rightarrow \infty}\Pr(Z_t(\hat{q}|l)= l-k) = 0$. Then, for any $\epsilon>0$ and $k$, there exists some $N\in \mathbb{Z}^+$ such that 
	\[
	\Pr(Z_t(\hat{q}|l) = l-k) <\epsilon, \quad \forall l>N.
	\]
	
	Consider some $\epsilon^*$ such that $0<\epsilon^*<1-\hat{q}+\hat{q}^2$.
	As $\hat{\Lambda}_t$ follows a Poisson distribution, there exists a $\lambda' \in \mathbb{Z}^+$ such that $\Pr(\hat{\Lambda}_t > \lambda')<\epsilon^*/2$. Given $\epsilon^*>0$, consider $N_k\in \mathbb{Z}^+$ such that $\Pr(Z_t(\hat{q}|l) = l-k)  <\epsilon^*/(2\lambda')$ for all $l>N_k.$ Let $N' :=\max_{k\le \lambda'}\{N_k\}$. We have
	\[
	\Pr(Z_t(\hat{q}|l) > l-\lambda') = \sum_{k=0}^{\lambda'-1} \Pr(Z_t(\hat{q}|l) = l-k)  <\epsilon^*/2, \quad \forall l>N'.
	\]
	Combining the two inequalities $\Pr(\hat{\Lambda}_{t+1} > \lambda')<\epsilon^*/2$ and $\Pr(Z_t(\hat{q}|l) > l-\lambda') <\epsilon^*/2$, we have
	\[
	\Pr(Z_t(\hat{q}|l) > l-\lambda' \text{ or } \hat{\Lambda}_{t+1} > \lambda') <\epsilon^*, \quad \forall l>N'.
	\]
	Recall that $\hat{S}_{t+1}=\hat{Z}_t(\hat{q}|\hat{S_t})+\hat{\Lambda}_{t+1}$. If $\hat{S}_{t+1}>\hat{S}_{t}$ and $\hat{\Lambda}_{t+1} \le \lambda'$, then we have $\hat{Z}_t(\hat{q}|\hat{S_t})>\hat{S_t}-\hat{\Lambda}_{t+1}\ge \hat{S_t}-\lambda'$, which leads to $
	\Pr(\hat{S}_{t+1}>\hat{S}_{t})\le \Pr(\hat{Z}_t(\hat{q}|\hat{S_t})>\hat{S_t}-\lambda' \text{ or }\hat{\Lambda}_{t+1} > \lambda')$. It then implies that  
	\[
	\sum_{k=l+1}^{\infty} \Pr(\hat{S}_{t+1}=k|\hat{S}_{t}=l)< \epsilon^*, \quad \forall l>N'.
	\] 
	Or equivalently,
	\[
	\sum_{k=0}^{ l} \Pr(\hat{S}_{t+1}=k|\hat{S}_{t}=l)\ge 1 - \epsilon^*, \quad \forall l>N'.
	\]
	Let $N^* := \max\{\ceil{(2\hat{\lambda}+\epsilon^*)/(1-\hat{q}+\hat{q}^2-\epsilon^*)},N'\}$. We have
	\[
	\sum_{k=0}^{l} \Pr(\hat{S}_{t+1}=k|\hat{S}_{t}=l) l \ge (1-\epsilon^*)l > l\hat{q}(1-\hat{q}) + 2\hat{\lambda} +\epsilon^*, \quad \forall l>N^*.
	\] 
	Let $F^* := \{l\in \mathbb{Z}^+: l \le N^*\}$ and recall that $\sum_{k=0}^{l} \Pr(\hat{S}_{t+1}=k|\hat{S}_{t}=l) k \le {\hat{\lambda}+l\hat{q}(1-\hat{q})}$ for any $l$. 
	It then leads to 
	\[
	\sum_{k=0}^{l} \Pr(\hat{S}_{t+1}=k|\hat{S}_{t}=l) (k - l) < -\hat{\lambda} - \epsilon^*, \quad \forall l\notin F^*.
	\]
	
	\noindent Finally, it remains to show \eqref{ineq_2: exist eps F}. We have
	\begin{align*}
		\sum_{k=l+1}^{\infty} \Pr(\hat{S}_{t+1}=k|\hat{S}_{t}=l) (k-l) &= \sum_{k=l+1}^{\infty}\sum_{m=0}^l \binom{l}{m}\hat{q}^m(1-\hat{q})^{l-m}e^{-\hat{\lambda}}\frac{\hat{\lambda}^{k-m}}{(k-m)!}(k-l)\\
		&= \sum_{m=0}^l \binom{l}{m}\hat{q}^m(1-\hat{q})^{l-m}\sum_{k=l+1}^{\infty} e^{-\hat{\lambda}}\frac{\hat{\lambda}^{k-m}}{(k-m)!}(k-m+m-l)\\
		& \le    \sum_{m=0}^l \binom{l}{m}\hat{q}^m(1-\hat{q})^{l-m}\sum_{k=l+1}^{\infty}e^{-\hat{\lambda}}\frac{\hat{\lambda}^{k-m}}{(k-m)!}(k-m)\\
		& \le \sum_{m=0}^l \binom{l}{m}\hat{q}^m(1-\hat{q})^{l-m}{\hat{\lambda}}\\
		&= {\hat{\lambda}}.
	\end{align*}
	Therefore, we have $\sum_{k=0}^{\infty} \Pr(\hat{S}_{t+1}=k|\hat{S}_{t}=l) (k- l)< -\epsilon^*$ for all $l\notin F^*$ where $F^* := \{l\in \mathbb{Z}^+: l\le N^*\}$. 
	
	The existence of $F^*$ and $\epsilon^*$ that satisfy the two conditions for the Foster's theorem implies that the modified system is stable. By definition, the expected total number of carriers in the modified system is an upper bound of that in the original system. Therefore, the expected total number of carriers in our original is bounded, which completes the proof.
	\Halmos \end{proof}

\vspace{0.1in}

\begin{proof}{Proof of Proposition \ref{prop:p = a and psi}.}
	Since our analysis does not depend on the specific lane, and with the assumption of symmetric carriers, we will drop the node and time indices $i$, $j$, $t$, and the carrier index $s$ for notation simplicity whenever the context is clear.
	
	In a state ($\mathbf{S},\mathbf{D})$ of an IC and IR mechanism, consider the net utility ${u}(C)$ of a carrier with true opportunity costs $C\in \mathbb{R}$. By the IC constraint, we have
	$$
	u(C) = p (C) - {a}( C) C \ge p ({c}) - {a}( {c}) C = u( {c}) - {a}( {c}) (C - {c})
	$$
	for any ${c}\in \mathbb{R}$. This implies that $u(C)$ is a convex function with gradient $- {a}( C)$. Recall that ${p}({C}) = {u}({C})+{a}({C}) {C} \text{ and }\boldsymbol{\psi}({C}) = {C}+ \frac{F({C})}{f({C})}$. Note that $\lim_{C\rightarrow \infty}a(C)=0$ because the platform should reject a carrier with opportunity cost  $\infty$. Then, the expected payment to the carrier is
	\begin{align*}
		\hat{\Ex}[{p}({C})] &= \hat{\Ex}[{u}({C})+{a}({C}) {C}]\\
		&=\int_{0}^{\infty}\int_{C}^{\infty}a(\zeta)f(C)d\zeta dC+\int_{0}^{\infty}a(C)Cf(C)dC\\
		&=\int_{0}^{\infty}\int_{0}^{\zeta}a(\zeta)f(C)dC d\zeta +\int_{0}^{\infty}a(C)Cf(C)dC\\
		&=\int_{0}^{\infty}a(\zeta)F(\zeta)d\zeta+\int_{0}^{\infty}a(C)Cf(C)dC\\
		&=\int_{0}^{\infty}a(C)\left(C + \frac{F(C)}{f(C)}\right)f(C)dC \\
		&=\hat{\Ex}\left[{a}({C}) \left( {C} + \frac{F({C})}{f({C})} \right) \right]\\
		&=\hat{\Ex}[{a}({C}) \boldsymbol{\psi}({C})],
	\end{align*}
	which completes the proof. 
	\Halmos 
\end{proof}

\subsection*{Proofs in Section~\ref{section: Fluid Approx}}

We first show an auxiliary result below.
\begin{lemma}\label{lemma: necessary conditions for any stationary mechanism}
	Constraints 
	\eqref{flow_FA}-\eqref{noneneagetive_FA} in \textsf{FA} are necessary conditions for any mechanism under which the system is stable. 
\end{lemma}

\begin{proof}{Proof of Lemma~\ref{lemma: necessary conditions for any stationary mechanism}.}
	We showed that any mechanism has a stationary distribution by Proposition \ref{Prop: stationary}. Consider any platform mechanism $\pi\in \Pi$. As the system is assumed to be in a steady state, we have
	$$
	\Ex [\mathbf{S}_{t+1}] = \Ex [\mathbf{S}_{t}],
	$$
	or alternatively,
	\rr{$$
		\Ex [S_{jt+1}] = \Ex [S_{jt}], \quad \forall j \in \mathcal{N}, 
		$$
		where the expectation is taken with respect to the stationary distribution of the Markov chain induced by a given policy. Taking expectation on both sides of Eq~\eqref{eq: sum Y and V = S} and defining $\bar{S}_{j}:=\Ex [S_{jt}]$, $\bar{Y}_{jk} := \Ex[Y_{jkt}]$, and $\bar{V}_j := \Ex[V_{jt}]$, we have 
		\begin{equation}  %\label{E[S_j t]_Y}
			\bar{S}_{j} = \Ex \left[\sum_{k\in \delta^+(j)} Y_{jkt} + V_{jt} \right] = \sum_{k\in \delta^+(j)} \bar{Y}_{jk} + \bar{V}_j.
	\end{equation}}
	Taking expectation on both sides of the system dynamics equation Eq~\eqref{eq: dynamics of St}, we have
	\rr{\begin{align}  
			\bar{S}_{j} &= \Ex \left[ \sum_{i\in \delta^-(j)} Z_{ijt} + \Lambda_{jt+1} \right]\nonumber\\
			&= \sum_{i\in \delta^-(j)} \Ex \left[ \Ex \left[ Z_{ijt} | Y_{ijt} \right]\right] + \Ex[\Lambda_{jt+1}] \nonumber \\
			&= \sum_{i\in \delta^-(j)}  \Ex \left[ \Ex \left[ Y_{ijt}  q_{ij} | Y_{ijt} \right] \right] + \lambda_j \label{Z conditional on Y} \\
			&= \sum_{i\in \delta^-(j)} q_{ij} \Ex \left[ Y_{ijt}  \right] + \lambda_j \nonumber \\
			&= \sum_{i\in \delta^-(j)} q_{ij} \bar{Y}_{ij} + \lambda_j, \nonumber %\label{outflow Y}
		\end{align}
		where Eq~\eqref{Z conditional on Y} holds because for a given value of $Y_{ijt}$, $Z_{ijt}$ follows a binomial distribution with parameters $(q_{ij},Y_{ijt})$.} Notice that $\bar{Y}_{ij}$ represents the long-run average fluid rate of carriers who actually ship a load from node $i$ to node $j$, which should be no more than the (fluid) demand rate $d_{ij}$ from node $i$ to node $j$, i.e., $\bar{Y}_{ij}\le d_{ij}$. Then, the following constraints are necessary for any mechanism under which the system is stable:
	\rr{\begin{align*}
			& \sum_{i\in\delta^{-}(j)}q_{ij}\bar{Y}_{ij}+\lambda_j = \sum_{k\in\delta^{+}(j)}\bar{Y}_{jk}+\bar{V}_{j} , \quad \forall j\in \mathcal{N}, \\
			& \bar{Y}_{ij} \leq d_{ij}, \quad \quad \forall (i,j)\in\mathcal{E}, \\
			& \bar{Y}_{ij} \ge 0, \quad \quad \forall (i,j)\in\mathcal{E}, \\
			& \bar{V}_{i} \ge 0, \quad \quad \forall i\in \mathcal{N}. 
	\end{align*}}
	Therefore, we can conclude that the constraints of \textsf{FA} are necessary conditions for any mechanism under which the system is stable.
	\Halmos 
\end{proof}

\vspace{0.1in}

We next analyze the expected total payment made to all the carriers at a given node. Consider node $i \in \mathcal{N}$ under a state ($\mathbf{S},\mathbf{D}$) in the stationary distribution, and we index the carriers at node $i$ as $s\in \{1,...,S_{i}\}$. Let $P_{i}$ denote the payment made to all the carriers who have transported a load originating from node $i$, where $P_{i} := \sum_{s = 1}^{S_i} p_{i}^s (\mathbf{C}_i^s)$.  

\begin{lemma}\label{lemma: lower bound of P}
	Under a state in the stationary distribution of the system induced by the platform's mechanism ${\pi}(\mathbf{S},\mathbf{D})=(\mathcal{M}_r,\mathcal{M}_p)$ that is IC and IR, the expected total payment to all the carriers in this period is lower bounded by
	\rr{\begin{equation}  \label{eq: total payment from all carriers}
			\Ex \left[{P}_{i} \right]  
			\ge \sum_{j\in \delta^+(i)} \left[ \ln \left( \frac{{\Ex}[{Y}_{ij}]}{{\Ex}[S_{i}]-\sum_{k\in \delta^+(i)}{\Ex}[{Y}_{ik}]}\right) +\alpha_{ij}  \right]  \frac{1}{\beta}{\Ex}[{Y}_{ij}], \quad \forall i\in \mathcal{N}\nonumber.
	\end{equation}}
\end{lemma}

\begin{proof}{Proof of Lemma~\ref{lemma: lower bound of P}.}
	Consider a state ($\mathbf{S}$,$\mathbf{D}$) in the stationary distribution of the system. For notation simplicity, we will use 
	$\hat{\Ex}[\cdot]$ to denote the conditional expectation $\Ex [\cdot | \mathbf{S}, \mathbf{D}]$ in the proof.
	
	By Proposition \ref{prop:p = a and psi}, we have $ \hat{\Ex}[{p}_{ij}^s({C}^s_{ij})]=\hat{\Ex}[{a}_{ij}^s({C}^s_{ij}) {\psi}_{ij}(C^s_{ij})]$. \rr{In this proof, we use $\hat{x}_{ij}(\mathbf{C}^s_{i})$ to denote the probability that the carrier $s$ at node $i$ chooses lane $(i,j)$ if his opportunity cost vector is $\mathbf{C}^s_{i}$. Then the expected total payment made by the platform to carrier $s$ at node $i$ is given by
		\begin{align*}
			\hat{\Ex}[p^s_i(\mathbf{C}^s_{i})]&=\sum_{j\in\delta^+(i)}\hat{x}_{ij}(\mathbf{C}^s_{i})\hat{\Ex}[{p}_{ij}^s({C}^s_{ij})]\\
			& = \sum_{j\in\delta^+(i)}\hat{x}_{ij}(\mathbf{C}^s_{i})\hat{\Ex}[{a}_{ij}^s({C}^s_{ij}) {\psi}_{ij}(C^s_{ij})]\\
			& = \hat{\Ex}\left[\sum_{j\in\delta^+(i)} \hat{x}_{ij}(\mathbf{C}^s_{i}) {a}_{ij}^s({C}^s_{ij}) {\psi}_{ij}(C^s_{ij})\right]\\
			& = \hat{\Ex} \left[\mathbf{a}_i^s(\mathbf{C}_i^s)\cdot \boldsymbol{\psi}_i(\mathbf{C}_i^s)\right],
		\end{align*}
		where the $j^{\text{th}}$ entry of vector $\mathbf{a}_i^s(\mathbf{C}_i^s)$ is $\hat{x}_{ij}(\mathbf{C}^s_{i}){a}_{ij}^s({C}^s_{ij})$ (which is clearly less than or equal to 1), and the $j^{\text{th}}$ entry of vector $\boldsymbol{\psi}_i(\mathbf{C}_i^s)$ is ${\psi}_{ij}(C^s_{ij})$.
		Then the expected total payment made by the platform to all carriers at node $i$ is given by}
	\rr{\begin{align}
			\hat{\Ex}[{P}_{i}] &= \hat{\Ex} \left[\sum_{s=1}^{S_i} \mathbf{a}_i^s(\mathbf{C}_i^s)\cdot \boldsymbol{\psi}_i(\mathbf{C}_i^s)\right]\nonumber\\
			&\ge \hat{\Ex}\left[ \min \sum_{s=1}^{S_i} \boldsymbol{\xi}_i^s\cdot \boldsymbol{\psi}_i(\mathbf{C}_i^s), \text{ s.t. } \boldsymbol{\xi}_i^s\in \Delta_i, \sum_{s=1}^{S_i} \boldsymbol{\xi}_i^s\ge \sum_{s=1}^{S_i} \mathbf{a}_i^s(\mathbf{C}_i^s) \right]\nonumber\\
			&\ge \hat{\Ex} \left[\min \sum_{s=1}^{S_i} \boldsymbol{\xi}_i^s\cdot \boldsymbol{\psi}_i(\mathbf{C}_i^s)+\boldsymbol{\eta}_{i}\cdot \left(\sum_{s=1}^{S_i} \mathbf{a}_i^s(\mathbf{C}^s_i)-\sum_{s=1}^{S_i} \boldsymbol{\xi}_i^s \right), \text{ s.t. } \boldsymbol{\xi}_i^s\in \Delta_i\right]\nonumber
		\end{align}
		where $\Delta_i := \left\{\boldsymbol{\xi}_i\in \mathbb{R}^n_+: \sum_{j\in \delta^+(i)} \xi_{ij}\le 1 \right\}$. The last inequality follows from relaxing the constraints $\sum_{s=1}^{S_i} \boldsymbol{\xi}_i^s\ge \sum_{s=1}^{S_i} \mathbf{a}_i^s(\mathbf{C}_i^s)$ with Lagrangian multipliers $\boldsymbol{\eta}_i = (\eta_{ij}\ge0,\forall j\in \delta^+(i))$ (the value of $\boldsymbol{\eta}_i$ will be specified later). It then follows that}  
	\rr{\begin{align}
			&\hat{\Ex} \left[\min \sum_{s=1}^{S_i} \boldsymbol{\xi}_i^s\cdot \boldsymbol{\psi}_i(\mathbf{C}_i^s)+\boldsymbol{\eta}_{i}\cdot \left(\sum_{s=1}^{S_i} \mathbf{a}_i^s(\mathbf{C}^s_i)-\sum_{s=1}^{S_i} \boldsymbol{\xi}_i^s\right), \text{ s.t. } \boldsymbol{\xi}_i^s\in \Delta_i\right] \nonumber\\
			= & \, \boldsymbol{\eta}_{i} \cdot \hat{\Ex} \left[\sum_{s=1}^{S_i} \mathbf{a}_i^s(\mathbf{C}_i^s)\right] + \hat{\Ex} \left[\min \sum_{s=1}^{S_i} \boldsymbol{\xi}_i^s\cdot (\boldsymbol{\psi}_i(\mathbf{C}_i^s)-\boldsymbol{\eta}_{i}), \text{ s.t. } \boldsymbol{\xi}_i^s\in \Delta_i\right] \nonumber\\
			= & \, \boldsymbol{\eta}_{i} \cdot \hat{\Ex} \left[\sum_{s=1}^{S_i} \mathbf{a}_i^s(\mathbf{C}_i^s)\right] + S_i\hat{\Ex} \left[\min \, \boldsymbol{\xi}_i\cdot (\boldsymbol{\psi}_i(\mathbf{C}_i)-\boldsymbol{\eta}_{i}), \text{ s.t. } \boldsymbol{\xi}_i\in \Delta_i\right] \label{equation:symmetry} \\
			= & \, \boldsymbol{\eta}_{i}\cdot \hat{\Ex}[\mathbf{Y}_i] + S_i \hat{\Ex}\left[\min \, \boldsymbol{\xi}_i\cdot (\boldsymbol{\psi}_i(\mathbf{C}_i)-\boldsymbol{\eta}_{i}), \text{ s.t. } \boldsymbol{\xi}_i\in \Delta_i\right] \label{equation:allocation define Y} \\
			= & \, \boldsymbol{\eta}_{i}\cdot \hat{\Ex}[\mathbf{Y}_i] + S_i \hat{\Ex}\left[\min  \{\psi_{ij}(\mathbf{C}_i)-\eta_{ij}, \forall j=1,...,n\}\wedge 0\right] \nonumber \\
			= & \, \boldsymbol{\eta}_{i}\cdot \hat{\Ex}[\mathbf{Y}_i] + S_i \hat{\Ex}\left[\mathbf{I}_i(\mathbf{C}_i,\boldsymbol{\eta}_{i})\cdot(\boldsymbol{\psi}_i(\mathbf{C}_i)-\boldsymbol{\eta}_{i})\right]  \label{eq: def of I min among all j} \\
			= & \, \boldsymbol{\eta}_{i}\cdot \left(\hat{\Ex}[\mathbf{Y}_i] -S_i \hat{\Ex}\left[\mathbf{I}_i(\mathbf{C}_i,\boldsymbol{\eta}_{i})\right]\right) + S_i \hat{\Ex}\left[\mathbf{I}_i(\mathbf{C}_i,\boldsymbol{\eta}_{i})\cdot \boldsymbol{\psi}_i(\mathbf{C}_i)\right].\label{eq: existence of eta_Y}
	\end{align}}
	A few remarks are in order. Eq~\eqref{equation:symmetry} follows from the assumption of symmetric carriers where \rr{$\mathbf{C}_i^s$ are independent and identically distributed, and index $s$ is dropped from this equation onwards. Eq~\eqref{equation:allocation define Y} follows from the relationship between the allocation probabilities $a_{ij}^s$, the lane choice probabilities $\hat{x}_{ij}$, and the number of transported loads $Y_{ij}$ for each O-D pair, where  $\hat{\Ex} \left[\sum_{s=1}^{S_i}\hat{x}_{ij}(\mathbf{C}^s_i) a^s_{ij}(\mathbf{C}^s_i) \right] = \hat{\Ex}[Y_{ij}]$. Eq~\eqref{eq: def of I min among all j} holds by the definition of the 0-1 vector $\mathbf{I}_i(\mathbf{C}_i,\boldsymbol{\eta}_i)$, where its $j^{\mathrm{th}}$ entry is equal to one if $\psi_{ij}(\mathbf{C}_i)-\eta_{ij}$ is non-positive and is the minimum among $\{\psi_{ik}(\mathbf{C}_i)-\eta_{ik}: k \in \delta^+(i)\}$, and equal to zero otherwise.
		Note that the function $\hat{\Ex}[\mathbf{I}_i(\mathbf{C}_i,\boldsymbol{\eta}_i)]$ is continuous (since $\mathbf{C}_i$ follows a continuous distribution) and increasing in $\boldsymbol{\eta}_i$. Therefore,
		for any $\mathbf{Y}_i$ such that $\hat{\Ex}[\mathbf{Y}_i]\cdot \mathbf{1}\le S_i$, we can pick some $\boldsymbol{\eta}_{Y_i}$ such that $\hat{\Ex}[\mathbf{Y}_i]=S_i \hat{\Ex}[\mathbf{I}_i(\mathbf{C}_i,\boldsymbol{\eta}_{Y_i})]$. Then it follows from Eq~\eqref{eq: existence of eta_Y} that
		\begin{equation}
			\label{lower bound of the auction price in a single origin}
			\hat{\Ex}\left[P_i\right] \ge S_i \hat{\Ex}\left[\mathbf{I}_i(\mathbf{C}_i,\boldsymbol{\eta}_{Y_i})\cdot \boldsymbol{\psi}_i(\mathbf{C}_i)\right]. 
	\end{equation}}
	\rr{Consider a posted price mechanism with the price vector $\mathbf{p}_i$ given by
		\[
		p_{ij}=\left[\ln \left (\frac{\hat{\Ex}[{Y}_{ij}]}{S_i-\sum_{k\in \delta^+(i)} \hat{\Ex}[{Y}_{ik}]} \right)+\alpha_{ij} \right]\frac{1}{\beta}.
		\]
		By the property of the MNL choice model, the choice probabilities of carriers under this price vector  are $\hat{\Ex}[\mathbf{Y}_i]/S_i$. 
		That is, $\mathbf{p}_i$ satisfies
		\[
		\frac{\hat{\Ex}[{Y}_{ij}]}{S_i}=\frac{\exp(\beta p_{ij} -\alpha_{ij})}{\sum_{k\in \delta^+(i)} \exp(\beta p_{ik} -\alpha_{ik})+1}.
		\]
		Note that for loads on lane $(i,j)$, this posted price mechanism has the same allocation probability as a mechanism with allocation function ${I}_{ij}(\mathbf{C}_i,\boldsymbol{\eta}_{Y_i})$.
		By the revenue equivalence principle (Proposition~\ref{prop:p = a and psi}), any mechanism that results in the same allocation probability must have the same expected payment. 
		Therefore, we have
		\[
		S_{i} \hat{\Ex}\left[I_{ij}(\mathbf{C}_i,\boldsymbol{\eta}_{Y_i}) \psi_{ij}(C_{ij})\right]
		=S_{i}p_{ij}\hat{\Ex}[Y_{ij}]/S_{i}=p_{ij} \hat{\Ex}[Y_{ij}],
		\]
		and taking summation over $j\in\delta^+(i)$ on both side, we have
		\[
		S_i \hat{\Ex}\left[\mathbf{I}_i(\mathbf{C}_i,\boldsymbol{\eta}_{Y_i})\cdot \boldsymbol{\psi}_i(\mathbf{C}_i)\right]
		=\mathbf{{p}}_i\cdot \hat{\Ex}[\mathbf{Y}_i],
		\]
		where the left-hand side is a lower bound of the expected payment of the mechanism being considered (see Eq~\eqref{lower bound of the auction price in a single origin}) and the right-hand side is the expected payment of the posted price mechanism. }
	\rr{Combining the above equation with Eq \eqref{lower bound of the auction price in a single origin}, it then follows that
		\begin{align}
			{\Ex}\left[P_i\right] 
			&= \Ex\left[\hat{\Ex}\left[\sum_{s=1}^{S_i} {p}_i(\mathbf{C}_i)\right]\right]\nonumber\\
			&\ge {\Ex}[\mathbf{{p}}_i\cdot \hat{\Ex}[\mathbf{Y}_i]]\nonumber\\
			&=\sum_{j\in \delta^+(i)} \Ex\left[\left[\ln \left( \frac{\hat{\Ex}[{Y}_{ij}]}{S_i-\sum_{k\in \delta^+(i)}\hat{\Ex}[{Y}_{ik}]}\right) +\alpha_{ij} \right]  \frac{1}{\beta}\hat{\Ex}[{Y}_{ij}]\right]\nonumber\\
			&\ge \sum_{j\in \delta^+(i)} \left[\ln \left( \frac{{\Ex}[{Y}_{ij}]}{{\Ex}[S_i]-\sum_{k\in \delta^+(i)}{\Ex}[{Y}_{ik}]}\right) +\alpha_{ij} \right]  \frac{1}{\beta}{\Ex}[{Y}_{ij}]\nonumber.
		\end{align}
		The last inequality follows from the Jensen's inequality. To see this, recall that we showed in Section \ref{subsection: convex reformulation} that $\frac{1}{\beta}\left[\ln \left (\frac{\bar{y}_{ij}}{\bar{v}_{i}} \right)+\alpha_{ij} \right] \bar{y}_{ij}$ is a convex function, where $\bar{v}_i = \bar{\lambda}_i - \sum_j \bar{y}_{ij}$. As convexity is preserved under linear transformations, $\frac{1}{\beta}\left[\ln \left (\frac{\bar{y}_{ij}}{\bar{\lambda}_{i}-  \sum_j \bar{y}_{ij}} \right)+\alpha_{ij} \right] \bar{y}_{ij}$ is also a convex function.}
	\Halmos 
\end{proof}

\vspace{0.1in}

\begin{proof}{Proof of Theorem \ref{theorem: FA provides an upper bound to an acution}.}
	
	By Proposition \ref{Prop: stationary} and Lemma \ref{lemma: necessary conditions for any stationary mechanism}, any platform policy under which the system is stable should satisfy the constraints of the \textsf{FA}. The stability of the system implies that the long-run average profit is the same as the expected profit under the stationary distribution. In the remainder of the proof, we will show that the optimal solution to \textsf{FA} gives an upper bound on the expected profit of any policy ${\pi}\in \Pi$ under the stationary distribution. By the revelation principle, we can restrict our focus to IC and IR mechanisms. Consider the expected payment made by the platform ${\Ex}\left[P_{ij}\right]$ under an IC and IR mechanism ${\pi}(\mathbf{S},\mathbf{D})=(\mathcal{M}_r,\mathcal{M}_p)$. By Lemma \ref{lemma: lower bound of P}, we have
	\rr{$$
		\sum_{j\in \delta^+(i)}{\Ex}\left[P_{ij}\right] = {\Ex}\left[P_i\right] 
		\ge \sum_{j\in \delta^+(i)} \left[\ln \left( \frac{{\Ex}[{Y}_{ij}]}{{\Ex}[S_i]-\sum_{k\in \delta^+(i)}{\Ex}[{Y}_{ik}]}\right) +\alpha_{ij} \right]  \frac{1}{\beta}{\Ex}[{Y}_{ij}]\nonumber.
		$$}
	Then, it follows that the expected profit of the platform mechanism $\pi$ is upper bounded by
	\rr{\begin{align}
			&\sum_{(i,j)\in \mathcal{E}}{\Ex}[r_{ij}(d_{ij})d_{ij} - b_{ij}(D_{ij}-Y_{ij})
			- P_{ij}] \nonumber\\
			\le & \, \sum_{(i,j)\in \mathcal{E}}r_{ij}(d_{ij})d_{ij} - b_{ij}(d_{ij}-{\Ex}\left[Y_{ij}\right])
			- \left[\ln \left( \frac{{\Ex}[{Y}_{ij}]}{{\Ex}[S_i]-\sum_{k\in \delta^+(i)}{\Ex}[{Y}_{ik}]}\right) +\alpha_{ij} \right]  \frac{1}{\beta}{\Ex}[{Y}_{ij}]\nonumber\\
			\le & \,\sum_{(i,j)\in \mathcal{E}}r^*_{ij}(d^*_{ij})d^*_{ij} -b_{ij}(d^*_{ij} - \bar{y}^*_{ij}) - p_{ij}^*\bar{y}^*_{ij}\nonumber.\end{align}}
	Therefore, the optimal value of the \textsf{FA} provides an upper bound on the long-run average profit of any platform mechanism, which completes the proof. \Halmos 
\end{proof}

\subsection*{Proofs in Section \ref{section: posted price mechanism}}
In order to prove Theorem~\ref{theorem: SI_bound}, we first present some auxiliary results which will prove useful for the analysis of the system performance under \textsf{SP}.
\rr{Notice that although \textsf{SP} uses a static posted price for each O-D pair that is independent of the system state, a carrier's choice probabilities are not constants and may dynamically change within each time period. In particular, the choice probabilities depend on the remaining shipper demands upon the carrier's arrival.
	%, since the carrier will not choose a destination with zero remaining demand. 
	This poses challenges in analyzing the \textsf{SP} mechanism directly. Instead, we will consider an auxiliary static posted price mechanism with \emph{fixed} choice probabilities, denoted as \textsf{SP-2}, to facilitate our analysis. As we shall show later, the long-run average profit under \textsf{SP-2} is lower than that under \textsf{SP}. Moreover, we show that the \textsf{SP-2} mechanism is asymptotically optimal, which then immediately implies the asymptotic optimality of the \textsf{SP} mechanism.}

\rr{The \textsf{SP-2} mechanism also applies the optimal fluid prices $(\mathbf{{r}}^*,\mathbf{{p}}^*)$ to all system states. However, different from \textsf{SP} under which carriers only choose among the available loads, under \textsf{SP-2} each carrier chooses to book a load that maximizes his utility \emph{without} considering the system state (in particular, the remaining demands). In other words, the carrier makes his decisions only based on the offered posted price, and he may end up choosing a destination with zero remaining load, in which case the platform would reject this carrier. It is worth noting that \textsf{SP-2} is not a realistic policy to be implemented in practice and it is only considered for the purpose of theoretical analysis. Under \textsf{SP-2}, the choice probabilities $\mathbf{x}_{it}=(x_{ijt},\forall j\in\delta^+(i))$ and $w_{it}$ are given in Section \ref{subsection: mechanism}, and the posted price satisfies $p^*_{ij}=\frac{1}{\beta}\left[\ln \left (\frac{x_{ijt}}{w_{it}} \right)+\alpha_{ij} \right]$ for each $j\in \delta^+(i)$ (cf. Eq~\eqref{eq: def of p_ijt}). }

\rr{For a problem instance with scaling parameter $\theta$, let $\gamma^{\textsf{SP-2}}(\theta)$ denote the long-run average profit of the platform under the \textsf{SP-2} mechanism. We first show in Lemma \ref{lemma: gamma_SP-2 <= gamma_SP} that $\gamma^{\textsf{SP-2}}(\theta)$ is a lower bound of $\gamma^{\textsf{SP}}(\theta)$.}
\rr{\begin{lemma} \label{lemma: gamma_SP-2 <= gamma_SP}
		The long-run average profit of the platform under the \textsf{SP-2} mechanism is no more than that under the \textsf{SP} mechanism:
		$$
		\gamma^{\textsf{SP-2}}(\theta) \le \gamma^{\textsf{SP}}(\theta).
		$$
\end{lemma}}

\begin{proof}{Proof of Lemma \ref{lemma: gamma_SP-2 <= gamma_SP}.}
	\rr{First, notice that $p^*_{ij}\le b_{ij}$ for all $(i,j)\in \mathcal{E}$. To see this, suppose $p^*_{ij}> b_{ij}$, and then we have
		\[
		\sum_{(i, j)\in \mathcal{E}} r^*_{ij}d_{ij}(r^*_{ij}) -p^*_{ij} \bar{y}^*_{ij} - b_{ij}(d^*_{ij}-\bar{y}^*_{ij}) < \sum_{(i, j)\in \mathcal{E}} r^*_{ij}d_{ij}(r^*_{ij}) - b_{ij}d^*_{ij},
		\]
		which leads to a contradiction that $p^*_{ij}$ is the optimal price. Recall that the platform's long-run average profits under \textsf{SP} and \textsf{SP-2} are given by 
		\begin{align*}
			\gamma^{\textsf{SP}} &= \sum_{(i,j)\in \mathcal{E}} r^*_{ij}d_{ij}(r^*_{ij}) - b_{ij}d^*_{ij}   + \sum_{(i,j)\in \mathcal{E}}(b_{ij}- p^*_{ij})\Ex \left[Y^{\textsf{SP}}_{ij} \right] \\
			\gamma^{\textsf{SP-2}} &= \sum_{(i,j)\in \mathcal{E}} r^*_{ij}d_{ij}(r^*_{ij}) - b_{ij}d^*_{ij}   + \sum_{(i,j)\in \mathcal{E}}(b_{ij}- p^*_{ij})\Ex \left[Y^{\textsf{SP-2}}_{ij} \right],
		\end{align*}
		where $\Ex[Y^{\textsf{SP}}_{ij} ]$ and $\Ex[Y^{\textsf{SP-2}}_{ij} ]$ are the expected number of carriers who delivered a load from node $i$ to node $j$ in a state of the stationary distribution under \textsf{SP} and \textsf{SP-2}, respectively. For notation simplicity, we will omit the dependency on $\theta$ in the proof. We will next show $\Ex[Y^{\textsf{SP-2}}_{ij}]\le \Ex[Y^{\textsf{SP}}_{ij}]$, which then immediately leads to $\gamma^{\textsf{SP-2}}\le \gamma^{\textsf{SP}}$.}
	
	\rr{To show $\Ex[Y^{\textsf{SP-2}}_{ij}]\le \Ex[Y^{\textsf{SP}}_{ij}]$, consider a coupling ($\widetilde{Y}^{\textsf{SP}}_{ijt},\widetilde{Y}^{\textsf{SP-2}}_{ijt}$) for all $(i,j)\in\mathcal{E}$ and for all $t$ in the same probability space, where $\widetilde{Y}^{\textsf{SP}}_{ijt}\sim {Y}^{\textsf{SP}}_{ijt}$ and $\widetilde{Y}^{\textsf{SP-2}}_{ijt}\sim {Y}^{\textsf{SP-2}}_{ijt}$. We next show $\widetilde{Y}^{\textsf{SP-2}}_{ijt}\le \widetilde{Y}^{\textsf{SP}}_{ijt}$ for all $(i,j)\in\mathcal{E}$ and for all $t$ by induction. Without loss of generality, we can assume that ${S}^{\textsf{SP}}_{i1}={S}^{\textsf{SP-2}}_{i1}=0$ for all $i\in \mathcal{N}$ because the initial state does not affect the stationary distribution.}
	
	\rr{\textbf{\underline{Base Case}.} When $t=1$, it is clear that $\widetilde{Y}^{\textsf{SP}}_{ijt}= \widetilde{Y}^{\textsf{SP-2}}_{ijt}=0$ for all $(i,j)\in\mathcal{E}$ because our model assumes zero carrier in the marketplace at the beginning of period 1.}
	
	\rr{\textbf{\underline{Induction Step}.} Suppose $\widetilde{Y}^{\textsf{SP-2}}_{ijt}\le \widetilde{Y}^{\textsf{SP}}_{ijt}$ for all $(i,j)\in\mathcal{E}$. We set ${Z}^{\textsf{SP}}_{ijt}={Z}^{\textsf{SP-2}}_{ijt}+{Z}^{\textsf{DIFF}}_{ijt}$ and $\Lambda^{\textsf{SP}}_{jt+1}=\Lambda^{\textsf{SP-2}}_{jt+1}$, where ${Z}^{\textsf{DIFF}}_{ijt}$ follows the binomial distribution $B(\widetilde{Y}^{\textsf{SP}}_{ijt} - \widetilde{Y}^{\textsf{SP-2}}_{ijt}, q_{ij})$. Then, we have 
		\[
		{S}^{\textsf{SP}}_{jt+1}=\sum_{i\in\delta^-(j)}{Z}^{\textsf{SP-2}}_{ijt}+{\Lambda}^{\textsf{SP-2}}_{jt+1}+\sum_{i\in\delta^-(j)}{Z}^{\textsf{DIFF}}_{ijt}={S}^{\textsf{SP-2}}_{jt+1}+\sum_{i\in\delta^-(j)}{Z}^{\textsf{DIFF}}_{ijt}, \forall j\in \mathcal{N}.
		\]
		As ${Z}^{\textsf{DIFF}}_{ijt}\ge 0$, we have ${S}^{\textsf{SP}}_{jt+1}\ge {S}^{\textsf{SP-2}}_{jt+1}$ for each $j\in \mathcal{N}.$ Set $\mathbf{D}_{t+1}^{\textsf{SP}}=\mathbf{D}_{t+1}^{\textsf{SP-2}}$ and $\mathbf{C}_{it+1}^{\textsf{SP}}=(\mathbf{C}_{it+1}^{\textsf{SP-2}},\mathbf{C}_{it+1}^{\textsf{DIFF}})$ for each $i\in\mathcal{N}$, where $\mathbf{C}_{it+1}^{\textsf{DIFF}}$ is the opportunity cost vector for those $({S}^{\textsf{SP}}_{it+1}-{S}^{\textsf{SP-2}}_{it+1})$ carriers who arrive at node $i$ after $\mathbf{C}_{it+1}^{\textsf{SP-2}}$ is realized. Note that the static posted price mechanisms make load allocation decisions independently at each origin node.  
		As the decisions under the \textsf{SP} mechanism for the first ${S}^{\textsf{SP-2}}_{it+1}$ carriers who arrive at node $i$ do not depend on the carriers who arrive afterwards, we have $A^{[s],\textsf{SP}}_{ijt+1}(\mathbf{C}^{\textsf{SP-2}}_{it+1})=A^{[s],\textsf{SP}}_{ijt+1}(\mathbf{C}^{\textsf{SP}}_{it+1})$ for all $s\le S^{\textsf{SP-2}}_{it+1}$. }
	%we can independently analyze the operations of the \textsf{SP} mechanism with the first ${S}^{\textsf{SP-2}}_{jt+1}$ carriers;

	\rr{Let $\mathcal{X}^{\textsf{SP-2}}_{ijt+1}(\mathbf{C}^{\textsf{SP-2}}_{it+1})$ be the set of carriers who choose to deliver a load from node $i$ to node $j$ in period $t+1$ under the \textsf{SP-2} mechanism when carriers report their opportunity costs $\mathbf{C}^{\textsf{SP-2}}_{it+1}$. Note that we have $\widetilde{Y}^{\textsf{SP-2}}_{ijt+1}=\sum_{[s]\in\mathcal{X}^{\textsf{SP-2}}_{ijt+1}(\mathbf{C}_{it+1}^{\textsf{SP-2}})}A^{[s],\textsf{SP-2}}_{ijt+1}(\mathbf{C}^{\textsf{SP-2}}_{it+1})$, where the superscript $[s]$ is used to denote the $s^{\mbox{\scriptsize th}}$ carrier that arrives to the marketplace. Consider carrier $[s]\in\mathcal{X}^{\textsf{SP-2}}_{ijt+1}(\mathbf{C}_{it+1}^{\textsf{SP-2}})$:}
	
	\rr{\textbf{\underline{Case 1}.} If $A^{[s],\textsf{SP}}_{ijt+1}(\mathbf{C}^{\textsf{SP-2}}_{it+1})=1$, then all the carriers who arrive at node $i$ earlier than carrier $[s]$ and choose to deliver a load from node $i$ to node $j$ in period $t+1$ are accepted by the \textsf{SP} mechanism. More specifically, if $[s']\in \mathcal{X}^{\textsf{SP-2}}_{ijt+1}(\mathbf{C}^{\textsf{SP-2}}_{it+1})$ and $s'\le s$, then 
		$A^{[s'],\textsf{SP}}_{ijt+1}(\mathbf{C}^{\textsf{SP-2}}_{it+1})=A^{[s'],\textsf{SP-2}}_{ijt+1}(\mathbf{C}^{\textsf{SP-2}}_{it+1})=1$. %because there are enough demands to accept carriers who choose to deliver the loads among the first $s$ number of carriers. 
		Otherwise, if $[s']\notin \mathcal{X}^{\textsf{SP-2}}_{ijt+1}(\mathbf{C}^{\textsf{SP-2}}_{it+1})$ and $s'\le s$, then $A^{[s'],\textsf{SP}}_{ijt+1}(\mathbf{C}^{\textsf{SP-2}}_{it+1})=1$ and $A^{[s'],\textsf{SP-2}}_{ijt+1}(\mathbf{C}^{\textsf{SP-2}}_{it+1})=0$ because the \textsf{SP-2} mechanism cannot accept any carrier not in $ \mathcal{X}^{\textsf{SP-2}}_{ijt+1}(\mathbf{C}^{\textsf{SP-2}}_{it+1})$. Then, we have $\sum_{s'=1}^{s}A^{[s'],\textsf{SP-2}}_{ijt+1}(\mathbf{C}^{\textsf{SP-2}}_{it+1})\le \sum_{s'=1}^{s}A^{[s'],\textsf{SP}}_{ijt+1}(\mathbf{C}^{\textsf{SP-2}}_{it+1})$. }
	
	\rr{\textbf{\underline{Case 2}.} If $A^{[s],\textsf{SP}}_{ijt+1}(\mathbf{C}^{\textsf{SP-2}}_{it+1})=0$, then it must be the case that $D^{[s],\textsf{SP}}_{ijt+1}=0$ since both \textsf{SP} and \textsf{SP-2} use the same posted price. In other words, all the $D^{\textsf{SP}}_{ijt+1}$ loads have been booked and the remaining demand when carrier $[s]$ arrives is zero, which then implies $D^{[s],\textsf{SP-2}}_{ijt+1} \ge D^{[s],\textsf{SP}}_{ijt+1}$. Therefore, we have  $\sum_{s'=1}^{s}A^{[s'],\textsf{SP-2}}_{ijt+1}(\mathbf{C}^{\textsf{SP-2}}_{it+1})\le \sum_{s'=1}^{s}A^{[s'],\textsf{SP}}_{ijt+1}(\mathbf{C}^{\textsf{SP-2}}_{it+1}) = D^{\textsf{SP}}_{ijt+1}$.}
	
	\rr{By the above two cases, we have 
		\[
		\widetilde{Y}^{\textsf{SP-2}}_{ijt+1}
		=\sum_{[s]\in\mathcal{X}^{\textsf{SP-2}}_{ijt+1}(\mathbf{C}_{it+1})}A^{[s],\textsf{SP-2}}_{ijt+1}(\mathbf{C}^{\textsf{SP-2}}_{it+1})\le \sum_{s=1}^{S^{\textsf{SP-2}}_{it+1}}A^{[s],\textsf{SP}}_{ijt+1}(\mathbf{C}^{\textsf{SP-2}}_{it+1})
		\le \sum_{s=1}^{S^{\textsf{SP}}_{it+1}}A^{[s],\textsf{SP}}_{ijt+1}(\mathbf{C}^{\textsf{SP}}_{it+1})= \widetilde{Y}_{ijt+1}^{\textsf{SP}}.
		\]
		The last inequality holds because the load allocation decisions for the first $S_{it+1}^{\textsf{SP-2}}$ carriers are not affected by the decisions for later carriers that arrive afterwards, and we have $ \sum_{s=1}^{S^{\textsf{SP}}_{it+1}}A^{[s],\textsf{SP}}_{ijt+1}(\mathbf{C}^{\textsf{SP}}_{it+1}) = \sum_{s=1}^{S^{\textsf{SP-2}}_{it+1}}A^{[s],\textsf{SP}}_{ijt+1}(\mathbf{C}^{\textsf{SP}}_{it+1}) + \sum_{s=S^{\textsf{SP-2}}_{it+1}+1}^{S^{\textsf{SP}}_{it}}A^{[s],\textsf{SP}}_{ijt+1}(\mathbf{C}^{\textsf{SP}}_{it+1})$. This completes the proof of the induction step. Therefore, we have $\widetilde{Y}^{\textsf{SP-2}}_{ijt}\le \widetilde{Y}^{\textsf{SP}}_{ijt}$, which implies
		$\Ex[{Y}^{\textsf{SP-2}}_{ij}]\le \Ex[{Y}^{\textsf{SP}}_{ij}]$ and hence $\gamma^{\textsf{SP-2}}\le\gamma^{\textsf{SP}}$.} \Halmos \end{proof}

\vspace{0.1in}

\rr{\noindent In view of Lemma \ref{lemma: gamma_SP-2 <= gamma_SP}, if the \textsf{SP-2} mechanism is asymptotically optimal, then the asymptotic optimality of the \textsf{SP} mechanism immediately follows. We next focus on establishing the asymptotic optimality of the \textsf{SP-2} mechanism. Let $X^{\textsf{SP-2}}_{ij}(\theta)$ and $Y^{\textsf{SP-2}}_{ij}(\theta)$ respectively denote the number of carriers who \emph{choose} to ship a load from node $i$ to node $j$ and who have actually delivered a load from node $i$ to node $j$ in a state of the stationary distribution under the \textsf{SP-2} mechanism, where $Y^{\textsf{SP-2}}_{ij}(\theta) = \min \{X^{\textsf{SP-2}}_{ij}(\theta), D^{\textsf{SP-2}}_{ij}(\theta) \}$. The following result provides an upper bound on $\Ex[X^{\textsf{SP-2}}_{ij}(\theta)] - \Ex[Y^{\textsf{SP-2}}_{ij}(\theta)]$, which measures the expected number of carriers in excess of the shipper demands in steady state under the \textsf{SP-2} mechanism.}

\begin{lemma} \label{(X-Y)_bound}
	Given a problem instance with scaling factor $\theta$, we have
	\begin{equation} \label{ineq: X_ij and Y_ij bound}
		\Ex[X^{\textsf{SP-2}}_{ij}(\theta)]-\Ex[Y^{\textsf{SP-2}}_{ij}(\theta)]\le O(\sqrt{\theta}), \quad \forall i \in \mathcal{N},
	\end{equation}
	where $\Ex[X^{\textsf{SP-2}}_{ij}(\theta)]$ and $\Ex[Y^{\textsf{SP-2}}_{ij}(\theta)]$ are the expected number of carriers who choose to ship a load from node $i$ to node $j$  and who actually delivered a load from node $i$ to node $j$, respectively, in a state of the stationary distribution under the \textsf{SP-2} mechanism.
\end{lemma}

\begin{proof}{Proof of Lemma~\ref{(X-Y)_bound}.}
	
	Consider a problem instance with scaling parameter $\theta$. Let $\mathbf{S}^{\textsf{SP-2}}(\theta)$ and $\mathbf{D}^{\textsf{SP-2}}(\theta)$ denote the number of carriers and the number of loads in the marketplace in a state of the stationary distribution under the \textsf{SP-2} mechanism. Let $\mathbf{X}^{\textsf{SP-2}}(\theta) = (X^{\textsf{SP-2}}_{ij}(\theta): (i,j) \in \mathcal{E})$ and $\mathbf{Y}^{\textsf{SP-2}}(\theta) = (Y^{\textsf{SP-2}}_{ij}(\theta): (i,j) \in \mathcal{E})$, where $X^{\textsf{SP-2}}_{ij}(\theta)$ and $Y^{\textsf{SP-2}}_{ij}(\theta)$ respectively denote the number of carriers who choose to ship a load and who actually shipped a load from node $i$ to node $j$ in a state under the stationary distribution. Consider an optimal solution ($\theta\mathbf{d}^*, \theta\mathbf{\bar{y}}^*,\mathbf{\theta\bar{v}}^*$) to the \textsf{FA}$(\theta)$ and the corresponding variables $(\mathbf{x}^*,\mathbf{y}^*,\boldsymbol{\bar{\lambda}}^*)$ and $(\mathbf{r}^*,\mathbf{p}^*)$ associated with this solution. We first show 
	\begin{equation} \label{ineq: S_i less than theta lambda bar}
		\Ex[S^{\textsf{SP-2}}_i(\theta)]\leq \theta\bar{\lambda}^*_i, \quad \forall i \in \mathcal{N}.
	\end{equation}
	Given the posted carrier-side prices $\mathbf{p}^*$, we have $x^*_{ij} = \Ex[X^{\textsf{SP-2}}_{ij}(\theta)]/\Ex[S^{\textsf{SP-2}}_i(\theta)]$. Let $y_{ij} = \Ex[Y^{\textsf{SP-2}}_{ij}(\theta)]/\Ex[S^{\textsf{SP-2}}_i(\theta)]$.  As $\Ex[Y^{\textsf{SP-2}}_{ij}(\theta)]\le\Ex[X^{\textsf{SP-2}}_{ij}(\theta)]$, we have $y_{ij}\le x^*_{ij}$. Then in steady state, we have
	\[
	\Ex[S^{\textsf{SP-2}}_{j}(\theta)] =  \sum_{i\in \delta^-(j)} q_{ij}\Ex[Y^{\textsf{SP-2}}_{ij}(\theta)] + \theta \lambda_{j}=  \sum_{i\in \delta^-(j)} q_{ij}y_{ij}\Ex[S^{\textsf{SP-2}}_{i}(\theta)] + \theta \lambda_{j}, \quad \forall j\in \mathcal{N}.
	\]
	As $y_{ij}\le x^*_{ij}$, it follows that $\Ex[S^{\textsf{SP-2}}_{j}(\theta)]\le \theta \bar{\lambda}_j^*$ for each $j \in \mathcal{N}$. Moreover, we have
	\begin{equation}  \label{eq: lambda bar x less than d}
		\bar{\lambda}^*_i{x}^*_{ij}\le d^*_{ij}, \quad \forall (i,j) \in \mathcal{E},
	\end{equation}
	since $\bar{\lambda}^*_i{x}^*_{ij}= \bar{x}^*_{ij}=\bar{y}^*_{ij}\le d^*_{ij}$.

	Now consider the expected number of carriers in excess of shipper demands $\Ex[X^{\textsf{SP-2}}_{ij}(\theta)] - \Ex[Y^{\textsf{SP-2}}_{ij}(\theta)]$:
	\begin{align*}
		\Ex[X^{\textsf{SP-2}}_{ij}(\theta)] - \Ex[Y^{\textsf{SP-2}}_{ij}(\theta)] &= \Ex [(X^{\textsf{SP-2}}_{ij}(\theta)-D^{\textsf{SP-2}}_{ij}(\theta))^+]\\
		&= \Ex [(X^{\textsf{SP-2}}_{ij}(\theta)-\theta d^*_{ij}+\theta d^*_{ij}-D^{\textsf{SP-2}}_{ij}(\theta))^+]\\
		&\leq \Ex [(X^{\textsf{SP-2}}_{ij}(\theta)-\Ex[X^{\textsf{SP-2}}_{ij}(\theta)]+\theta d^*_{ij}-D^{\textsf{SP-2}}_{ij}(\theta))^+]\\
		&\leq \Ex [|X^{\textsf{SP-2}}_{ij}(\theta)- \Ex[X^{\textsf{SP-2}}_{ij}(\theta)]+\theta d^*_{ij}-D^{\textsf{SP-2}}_{ij}(\theta)|]\\
		&\leq \Ex[|X^{\textsf{SP-2}}_{ij}(\theta)-\Ex[X^{\textsf{SP-2}}_{ij}(\theta)]|] +\Ex[|D^{\textsf{SP-2}}_{ij}(\theta)- \theta d^*_{ij}|] \\
		&\leq \sqrt{\Ex[(X^{\textsf{SP-2}}_{ij}(\theta)- \Ex[X^{\textsf{SP-2}}_{ij}(\theta)])^2}] + \sqrt{\Ex[(D^{\textsf{SP-2}}_{ij}(\theta)-\theta d^*_{ij})^2]}\\
		&= \sqrt{Var[X^{\textsf{SP-2}}_{ij}(\theta)]}+ \sqrt{Var[D^{\textsf{SP-2}}_{ij}(\theta)]}\\
		&= \sqrt{\Ex[S^{\textsf{SP-2}}_{i}(\theta)]x^*_{ij}(1-x^*_{ij})} +\sqrt{\theta d^*_{ij}}\\
		&\leq \sqrt{\theta \bar{\lambda}^*_i x^*_{ij}(1-x^*_{ij})} +\sqrt{\theta d^*_{ij}}.
	\end{align*}
	A few remarks are in order. The first inequality follows from $\Ex[X^{\textsf{SP-2}}_{ij}(\theta)]=\Ex[S^{\textsf{SP-2}}_i(\theta)]x^*_{ij}$, and Eq~\eqref{ineq: S_i less than theta lambda bar}  and Eq~\eqref{eq: lambda bar x less than d}: 
	\[
	\Ex[X^{\textsf{SP-2}}_{ij}(\theta)] = \Ex[S^{\textsf{SP-2}}_i(\theta)]x^*_{ij}\le \theta \bar{\lambda}^*_ix^*_{ij}\le \theta d^*_{ij}.
	\]
	The third inequality follows from the triangle inequality, and the fourth inequality holds by the Cauchy-Schwartz inequality. The last equality holds by the variance of random variable that follow the multinomial distribution and the Poisson distribution. Therefore, $\Ex[X^{\textsf{SP-2}}_{ij}(\theta)] - \Ex[Y^{\textsf{SP-2}}_{ij}(\theta)]\le O(\sqrt{\theta})$. \Halmos \end{proof}

\vspace{0.1in}

Next, we show an upper bound on the difference between the expected number of carriers in steady state under the \textsf{SP-2} mechanism and that in the fluid system.

\begin{lemma} \label{lemma: (lambda-S)_bound}
	Given a problem instance with scaling factor $\theta$, we have
	\[
	\theta \bar{\lambda}^*_i - \Ex[S^{\textsf{SP-2}}_i(\theta)]\le O(\sqrt{\theta}), \quad \forall i \in \mathcal{N},
	\]
	where $\Ex[S^{\textsf{SP-2}}_i(\theta)]$ is the expected number of carriers in node $i$ in a state of the stationary distribution under the \textsf{SP-2} mechanism.
\end{lemma}

\begin{proof}{Proof of Lemma \ref{lemma: (lambda-S)_bound}.}
	%Consider $\Ex[Y^{\textsf{SP-2}}_{ij}(\theta)]$ and $\Ex[S^{\textsf{SP-2}}_i(\theta)]$ of the \textsf{SP-2} mechanism which does not consider a sequence of arrivals under the stationary distribution. 
	By Lemma \ref{lemma: necessary conditions for any stationary mechanism}, the \textsf{SP-2} mechanism satisfies
	\[
	\sum_{i\in\delta^{-}(j)}q_{ij}\Ex[Y^{\textsf{SP-2}}_{ij}(\theta)]+\theta\lambda_j=\sum_{k\in \delta^+(j)}\Ex[Y^{\textsf{SP-2}}_{jk}(\theta)] + \Ex[V^{\textsf{SP-2}}_{j}(\theta)]=\Ex[S^{\textsf{SP-2}}_j(\theta)], \quad \forall j \in \mathcal{N}.
	\] 
	Let $y^{\textsf{SP-2}}_{ij}(\theta) = \Ex[{Y}^{\textsf{SP-2}}_{ij}(\theta)]/\Ex[S^{\textsf{SP-2}}_i(\theta)]$ and $\bar{\lambda}^{\textsf{SP-2}}_i(\theta) = \Ex[S^{\textsf{SP-2}}_i(\theta)]/\theta$. Then, we have
	\[
	\sum_{i\in \delta^{-}(j)} q_{ij}y_{ij}^{\textsf{SP-2}}(\theta)\bar{\lambda}^{\textsf{SP-2}}_i(\theta)+\lambda_j = \bar{\lambda}^{\textsf{SP-2}}_j(\theta),  \quad \forall j \in \mathcal{N}.
	\]
	Consider a deterministic system with a $\theta$-scaled problem instance. Let $S_{it}(\theta)$ denote the number of carriers in this deterministic system at node $i$ in period $t$. Given any initial state $S_{i0}(\theta)$, the system dynamics of this deterministic system are given by
	\begin{equation} \label{eq: deterministic system dynamic}
		S_{jt+1}(\theta) = \sum_{i\in \delta^{-}(j)} q_{ij}x^*_{ij}S_{it}(\theta)+\theta\lambda_j, \quad \forall j \in \mathcal{N}.
	\end{equation}
	%We next analyze how the system dynamics change over time in the above deterministic system. 
	\noindent Recall that the optimal solution of the \textsf{FA} satisfies the following equations:
	\[
	\bar{\lambda}^*_j = \sum_{i\in \delta^{-}(j)} q_{ij}x^*_{ij}\bar{\lambda}^*_i+\lambda_j, \quad \forall j \in \mathcal{N}.
	\]
	Comparing the above two equations, we observe that with the initial state given as $S_{i0}(\theta)=\theta\bar{\lambda}^*_i$, $\lim_{t \rightarrow \infty} S_{jt}(\theta) = \theta \bar{\lambda}^*_j$. Since the initial state does not affect the steady state of the system, we have $\lim_{t \rightarrow \infty} S_{jt}(\theta) = \theta \bar{\lambda}^*_j$ for each $j \in \mathcal{N}$. 
	
	%diminishes goes to 0 when $t$ goes to infinity. The initial state does not affect to $\lim_{t \rightarrow \infty} S_{jt}(\theta)$. 
	
	%These two observations imply that $\lim_{t \rightarrow \infty} S_{jt}(\theta) = \theta \bar{\lambda}^*_j$ for any $S_{i0}\in \mathcal{R}^+$. 
	
	Consider the deterministic system with initial state $S_{i0}(\theta)=\theta\bar{\lambda}^{\textsf{SP-2}}_i(\theta)$. Let $\Delta_{jt}(\theta) := S_{jt}(\theta)-\theta\bar{\lambda}^{\textsf{SP-2}}_j(\theta)$ and $\eta_j(\theta) := \sum_{i\in\delta^-(j)}q_{ij}(x_{ij}^*-y^{\textsf{SP-2}}_{ij}(\theta))\theta\bar{\lambda}^{\textsf{SP-2}}_i(\theta)$. Then, by Eq~\eqref{eq: deterministic system dynamic}, $S_{j1}(\theta)$ is given by
	\begin{align*}
		S_{j1}(\theta) &= \sum_{i\in\delta^-(j)}q_{ij}x^*_{ij}\theta\bar{\lambda}^{\textsf{SP-2}}_i(\theta) + \theta\lambda_j\\
		&= \sum_{i\in\delta^-(j)}q_{ij}(x_{ij}^* + y^{\textsf{SP-2}}_{ij}(\theta)-y^{\textsf{SP-2}}_{ij}(\theta))\theta\bar{\lambda}^{\textsf{SP-2}}_i(\theta) + \theta\lambda_j\\
		&= \sum_{i\in\delta^-(j)}q_{ij}y^{\textsf{SP-2}}_{ij}(\theta)\theta\bar{\lambda}^{\textsf{SP-2}}_i(\theta) + \theta \lambda_j + \sum_{i\in\delta^-(j)}q_{ij}(x_{ij}^*-y^{\textsf{SP-2}}_{ij}(\theta))\theta\bar{\lambda}^{\textsf{SP-2}}_i(\theta) \\
		&= \theta\bar{\lambda}^{\textsf{SP-2}}_j(\theta) + \eta_j(\theta),
	\end{align*}
	where the last equation holds by $\sum_{i\in \delta^{-}(j)} q_{ij}y_{ij}^{\textsf{SP-2}}(\theta)\bar{\lambda}^{\textsf{SP-2}}_i(\theta)+\lambda_j = \bar{\lambda}^{\textsf{SP-2}}_j(\theta).$
	By the system dynamics Eq~\eqref{eq: deterministic system dynamic}, $S_{j2}(\theta)$ can be expressed as
	\begin{align*}
		S_{j2}(\theta) &= \sum_{i\in\delta^-(j)}q_{ij}(x_{ij}^*+y^{\textsf{SP-2}}_{ij}(\theta)-y^{\textsf{SP-2}}_{ij}(\theta))S_{i1}(\theta) + \theta \lambda_j\\
		&= \sum_{i\in\delta^-(j)}q_{ij}y^{\textsf{SP-2}}_{ij}(\theta)S_{i1}(\theta) + \theta \lambda_j + \sum_{i\in\delta^-(j)}q_{ij}(x_{ij}^*-y^{\textsf{SP-2}}_{ij}(\theta))S_{i1}(\theta) \\
		&= \sum_{i\in\delta^-(j)}q_{ij}y^{\textsf{SP-2}}_{ij}(\theta)(\theta\bar{\lambda}^{\textsf{SP-2}}_i(\theta) + \Delta_{i1}(\theta)) + \theta \lambda_j + \sum_{i\in\delta^-(j)}q_{ij}(x_{ij}^*-y^{\textsf{SP-2}}_{ij}(\theta))(\theta\bar{\lambda}^{\textsf{SP-2}}_i(\theta) + \Delta_{i1}(\theta)) \\
		&= \theta\bar{\lambda}^{\textsf{SP-2}}_j(\theta) + \sum_{i\in\delta^-(j)}q_{ij}y^{\textsf{SP-2}}_{ij}(\theta)\Delta_{i1}(\theta) + \sum_{i\in\delta^-(j)}q_{ij}(x_{ij}^*-y^{\textsf{SP-2}}_{ij}(\theta))(\theta\bar{\lambda}^{\textsf{SP-2}}_i(\theta) + \Delta_{i1}(\theta)) \\
		&= \theta\bar{\lambda}^{\textsf{SP-2}}_j(\theta) + \sum_{i\in\delta^-(j)}q_{ij}(y^{\textsf{SP-2}}_{ij}(\theta)+(x_{ij}^*-y^{\textsf{SP-2}}_{ij}(\theta)))\Delta_{i1}(\theta)+ \sum_{i\in\delta^-(j)}q_{ij}(x_{ij}^*-y^{\textsf{SP-2}}_{ij}(\theta))\theta\bar{\lambda}^{\textsf{SP-2}}_i(\theta)\\
		&= \theta\bar{\lambda}^{\textsf{SP-2}}_j(\theta) + \sum_{i\in\delta^-(j)}q_{ij}x^*_{ij}\Delta_{i1}(\theta)+ \eta_j(\theta),
	\end{align*}
	where the fourth equation holds by $\sum_{i\in\delta^-(j)}q_{ij}y^{\textsf{SP-2}}_{ij}(\theta)\bar{\lambda}^{\textsf{SP-2}}_i(\theta) +  \lambda_j =\bar{\lambda}^{\textsf{SP-2}}_j(\theta)$. From the last equation above, we have
	\[
	S_{j2}(\theta) - \theta\bar{\lambda}^{\textsf{SP-2}}_j(\theta) = \Delta_{j2}(\theta) = \sum_{i\in\delta^-(j)}q_{ij}x^*_{ij}\Delta_{i1}(\theta)+ \eta_j(\theta).
	\] 
	It is easy to see that the above relationship between $\Delta_{i2}(\theta)$ and $\Delta_{i1}(\theta)$ can be generalized to
	\[
	\Delta_{jt}(\theta) = \sum_{i\in\delta^-(j)}q_{ij}x^*_{ij}\Delta_{it-1}(\theta)+ \eta_j(\theta),\quad \forall t>1.
	\] 
	%\[
	%\Delta_{j1}(\theta)=\eta_j(\theta)=\sum_{i\in\delta^-(j)}q_{ij}(x_{ij}^*-y^{\textsf{SP-2}}_{ij}(\theta))\theta\bar{\lambda}^{\textsf{SP-2}}_i(\theta).
	%\]
	Recall that we have shown $\Ex[S^{\textsf{SP-2}}_j(\theta)]\le \theta \bar{\lambda}^*_j$ in the proof of Lemma \ref{(X-Y)_bound}. Then, we have 
	$\lim_{t\rightarrow \infty}\Delta_{jt}(\theta) = \theta \bar{\lambda}^*_j - \theta\bar{\lambda}^{\textsf{SP-2}}_j(\theta) = \theta \bar{\lambda}^*_j - \Ex[S^{\textsf{SP-2}}_j(\theta)] <\infty$. In addition, $\eta_j(\theta)$ is bounded by $O(\sqrt{\theta})$ since 
	\[
	\theta\bar{\lambda}^{\textsf{SP-2}}_i(\theta)(x^*_{ij} - y^{\textsf{SP-2}}_{ij}(\theta)) = \Ex[X^{\textsf{SP-2}}_{ij}(\theta)] - \Ex[Y^{\textsf{SP-2}}_{ij}(\theta)]\le O(\sqrt{\theta})
	\]
	by Lemma \ref{(X-Y)_bound}. Intuitively, $\Delta_{jt}(\theta)$ can be considered as a linear function of $\boldsymbol{\eta(\theta)}$ with non-negative coefficients, and these coefficients converge to certain real numbers because $\lim_{t\rightarrow \infty}\Delta_{jt}(\theta)$ is bounded.
	We note that these coefficients depend on $\mathbf{x}^*$ and $\mathbf{q}$, but independent of the \textsf{SP-2} mechanism and $\theta$. It then follows that
	\[
	\lim_{t\rightarrow \infty}\Delta_{jt}(\theta) = \theta \bar{\lambda}^*_j - \theta\bar{\lambda}^{\textsf{SP-2}}_j(\theta) = \theta \bar{\lambda}^*_j - \Ex[S^{\textsf{SP-2}}_j(\theta)]\le O(\sqrt{\theta}),
	\]
	which completes the proof.
	\Halmos \end{proof}

With Lemmas \ref{lemma: gamma_SP-2 <= gamma_SP}, \ref{(X-Y)_bound}, and \ref{lemma: (lambda-S)_bound}, we are now ready to complete the proof of Theorem \ref{theorem: SI_bound}.

\begin{proof}{Proof of Theorem \ref{theorem: SI_bound}.}
	Given a problem instance with scaling parameter $\theta$, the long-run average profit of the platform under the \textsf{SP-2} mechanism, $\gamma^{\textsf{SP-2}}(\theta)$, is given by
	\begin{align*}
		\gamma^{\textsf{SP-2}}(\theta) &=  \Ex \left[ \sum_{(i,j)\in \mathcal{E}} \theta r^*_{ij}d_{ij}(r^*_{ij}) - b_{ij}(D^{\textsf{SP-2}}_{ij}(\theta)-Y^{\textsf{SP-2}}_{ij}(\theta)) - p^*_{ij} Y^{\textsf{SP-2}}_{ij}(\theta)\right]\\
		&= \Ex \left[ \sum_{(i,j)\in \mathcal{E}} \theta r^*_{ij}d_{ij}(r^*_{ij}) - b_{ij}(D^{\textsf{SP-2}}_{ij}(\theta)-X^{\textsf{SP-2}}_{ij}(\theta)) - b_{ij}(X^{\textsf{SP-2}}_{ij}(\theta)-D^{\textsf{SP-2}}_{ij}(\theta))^+ - p^*_{ij} Y^{\textsf{SP-2}}_{ij}(\theta)\right]\\
		&\geq \Ex \left[ \sum_{(i,j)\in \mathcal{E}}\theta r^*_{ij}d_{ij}(r^*_{ij}) - b_{ij}(D^{\textsf{SP-2}}_{ij}(\theta)-X^{\textsf{SP-2}}_{ij}(\theta)) - b_{ij}(X^{\textsf{SP-2}}_{ij}(\theta)-D^{\textsf{SP-2}}_{ij}(\theta))^+ - p^*_{ij} X^{\textsf{SP-2}}_{ij}(\theta) \right]\nonumber\\
		&= \Ex \left [\sum_{(i,j)\in \mathcal{E}} \theta r^*_{ij}d_{ij}(r^*_{ij}) - b_{ij}D^{\textsf{SP-2}}_{ij}(\theta) + (b_{ij}-p^*_{ij})X^{\textsf{SP-2}}_{ij}(\theta)\right] -\Ex \left[ \sum_{(i,j)\in \mathcal{E}}  b_{ij}(X^{\textsf{SP-2}}_{ij}(\theta)-D^{\textsf{SP-2}}_{ij}(\theta))^+ \right],
	\end{align*}
	where the first inequality follows from the definition of $Y^{\textsf{SP-2}}_{ij}(\theta)$. The second term in the last equation is bounded by $O(\sqrt{\theta})$ as shown in Lemma \ref{(X-Y)_bound}. Consider the first term in the last equation:
	\begin{align*}
		&\Ex \left [\sum_{(i,j)\in \mathcal{E}} \theta r^*_{ij}d_{ij}(r^*_{ij}) - b_{ij}D^{\textsf{SP-2}}_{ij}(\theta) + (b_{ij}-p^*_{ij})X^{\textsf{SP-2}}_{ij}(\theta)\right]\\
		&=\sum_{(i,j)\in \mathcal{E}} \theta r^*_{ij}d_{ij}(r^*_{ij}) - \theta b_{ij} d^*_{ij} + (b_{ij}- p^*_{ij}) \Ex[S^{\textsf{SP-2}}_{i}(\theta)]x^*_{ij}\\
		&=\sum_{(i,j)\in \mathcal{E}} \theta r^*_{ij}d_{ij}(r^*_{ij}) - \theta b_{ij} d^*_{ij} + (b_{ij}- p^*_{ij}) (\theta \bar{\lambda}^*_i-\theta \bar{\lambda}^*_i+\Ex[S^{\textsf{SP-2}}_{i}(\theta)])x^*_{ij}\\
		&=\sum_{(i,j)\in \mathcal{E}} \theta r^*_{ij}d_{ij}(r^*_{ij}) 
		- p^*_{ij} \cdot \theta \bar{\lambda}^*_i x^*_{ij} - b_{ij} \cdot \theta (d^*_{ij} - \bar{\lambda}^*_i x^*_{ij}) + (b_{ij}- p^*_{ij})x^*_{ij}( -\theta \bar{\lambda}^*_i+\Ex[S^{\textsf{SP-2}}_{i}(\theta)]) \\
		&= \gamma^{\textsf{FA}}(\theta) -\sum_{(i,j)\in \mathcal{E}}(b_{ij}- p^*_{ij})x^*_{ij}(\theta \bar{\lambda}^*_i-\Ex[S^{\textsf{SP-2}}_{i}(\theta)]).
	\end{align*}
	By Lemma \ref{lemma: (lambda-S)_bound}, $\theta \bar{\lambda}^*_i-\Ex[S^{\textsf{SP-2}}_{i}(\theta)]$ is bounded from above by $O(\sqrt{\theta})$. Therefore, $\gamma^{\textsf{FA}}(\theta) - \gamma^{\textsf{SP-2}}(\theta)$ is bounded from above by $O(\sqrt{\theta})$. It then follows that
	\[
	\gamma^{\textsf{FA}}(\theta) - \gamma^{\textsf{SP}}(\theta) \le \gamma^{\textsf{FA}}(\theta) - \gamma^{\textsf{SP-2}}(\theta) \le O(\sqrt{\theta}), 
	\]
	where the first inequality follows from Lemma \ref{lemma: gamma_SP-2 <= gamma_SP}, and this completes our proof.
	\Halmos \end{proof}

\subsection*{Proofs in Section \ref{section: hybrid mechanisms}}
\rr{The \textsf{HYB} mechanism is a hybrid of \textsf{SP} and \textsf{AUC}. Before proving Lemma \ref{lemma: HYB is IC and IR}, we first show that \textsf{AUC} is IC and IR on each lane.}

\begin{lemma}\label{lemma: VCG is IC and IR}
	\rr{The uniform price auction \textsf{AUC} is IC and IR.}
\end{lemma}
\begin{proof}{Proof of Lemma \ref{lemma: VCG is IC and IR}.}
	Consider the lane $(i,j)$ in time period $t$. Since the uniform price auction \textsf{AUC} is applied to each lane separately, we will remove the indices $i$, $j$, and $t$ to simplify the notation. Let $(\mathbf{C}^{-s},{c}^s)$ be the opportunity cost vector submitted by the carriers, where we assume that all carriers other than carrier $s$ submit their bids truthfully. Now consider the payoff of the carrier $s$ under the uniform price auction \textsf{AUC}:
	\begin{align*}
		&P^s(\mathbf{C}^{-s},{c}^s) - {C}^s  {A}^{s*}(\mathbf{C}^{-s}, {c}^s) \\
		= &\, {c}^s  {A}^{s*}(\mathbf{C}^{-s}, {c}^s) + \mathcal{J}(\mathbf{C}^{-s}) - \mathcal{J}(\mathbf{C}^{-s}, {c}^s)- {C}^s  {A}^{s*}(\mathbf{C}^{-s}, {c}^s)\\
		= &\, \mathcal{J}(\mathbf{C}^{-s}) - \mathcal{J}(\mathbf{C}^{-s}, {c}^s)+( {c}^s - {C}^s)  {A}^{s*}(\mathbf{C}^{-s}, {c}^s) \\
		= &\, \mathcal{J}(\mathbf{C}^{-s}) -\sum_{s'\in \mathcal{S}\setminus\{s\}} {C}^{s'}  {A}^{s'*}(\mathbf{C}^{-s}, {c}^s) - {\xi}^*  {Y}^{0*}(\mathbf{C}^{-s}, {c}^s)\\
		&\qquad - {c}^{s}  {A}^{s*}(\mathbf{C}^{-s}, {c}^s) +( {c}^s - {C}^s)  {A}^{s*}(\mathbf{C}^{-s}, {c}^s)\\
		= &\, \mathcal{J}(\mathbf{C}^{-s}) -\sum_{s'\in \mathcal{S}} {C}^{s'}  {A}^{s'*}(\mathbf{C}^{-s}, {c}^s) - {\xi}^*  {Y}^{0*}(\mathbf{C}^{-s}, {c}^s)\\
		\le &\, \mathcal{J}(\mathbf{C}^{-s}) -\mathcal{J}( \mathbf{C})\\
		= &\, P^s( \mathbf{C}) - {C}^s  {A}^{s*}( \mathbf{C}).
	\end{align*}
	Therefore, it is optimal for the carrier $s$ to bid truthfully, so the uniform price auction is incentive compatible. In addition, we have $\mathcal{J}(\mathbf{C}^{-s}) \ge \mathcal{J}( \mathbf{C})$. This is because a feasible solution to problem $\mathcal{J}( \mathbf{C})$ can be constructed from the optimal solution to $\mathcal{J}(\mathbf{C}^{-s})$ by adding additional relevant variables associated with the carrier $s$ and restricting the value of these variables equal to zero, and therefore the optimal objective value of $\mathcal{J}( \mathbf{C})$ is no more than that of $\mathcal{J}(\mathbf{C}^{-s})$. As a result, the uniform price auction is also individual rational.
	\Halmos \end{proof}

\vspace{0.1in}

\begin{proof}{Proof of Lemma \ref{lemma: HYB is IC and IR}.}
	Suppose carriers report their opportunity costs $\mathbf{C}_{ijt}$ to the platform. For simplicity, indices $i$, $j$, and $t$ are dropped. We first show the \textsf{HYB} mechanism is IR. If $X^{\textsf{SP}}(\mathbf{C})>D$, then the payments and the load allocations follow  \textsf{SP} which is IR. Otherwise, the payments and the load allocations follow \textsf{AUC}, which is IR. Therefore, the \textsf{HYB} mechanism is also IR.
	
	We next show the \textsf{HYB} mechanism is IC. Suppose the bid vector submitted by the carriers on a lane in period $t$ is $(\mathbf{C}^{-s},{c}^s)$. We will show the following inequality holds by considering three cases:
	\begin{align*}
		P^{s,\textsf{HYB}}(\mathbf{C}^{-s},{c}^s) - {C}^s A^{s,\textsf{HYB}}(\mathbf{C}^{-s},{c}^s)\le P^{s,\textsf{HYB}}(\mathbf{C}) -{C}^s A^{s,\textsf{HYB}}(\mathbf{C}).
	\end{align*}
	
	\textbf{\underline{Case 1}.} If $A^{s,\textsf{HYB}}(\mathbf{C}^{-s},{c}^s)=0$, then we have $P^{s,\textsf{HYB}}(\mathbf{C}^{-s},{c}^s)=0$ because both \textsf{SP} and \textsf{AUC} do not make any payment to carriers who do not receive a load allocation. Then, we have
	\[ 
	P^{s,\textsf{HYB}}(\mathbf{C}^{-s},{c}^s) - {C}^s A^{s,\textsf{HYB}}(\mathbf{C}^{-s},{c}^s) = 0 \le P^{s,\textsf{HYB}}(\mathbf{C}) -{C}^s A^{s,\textsf{HYB}}(\mathbf{C}).
	\]
	The above inequality holds because the \textsf{HYB} mechanism is IR.
	
	\textbf{\underline{Case 2}.} If $A^{s,\textsf{HYB}}(\mathbf{C}^{-s},{c}^s)=1$ and $X^{\textsf{SP}}(\mathbf{C}^{-s},{c}^s)>D$, then in this case, we have $P^{s,\textsf{HYB}}(\mathbf{C}^{-s},{c}^s)=p^*$ and $A^{s,\textsf{SP}}(\mathbf{C}^{-s},{c}^s)=A^{s,\textsf{HYB}}(\mathbf{C}^{-s},{c}^s)=1$. If $C^s<p^*$, then $A^{s,\textsf{SP}}(\mathbf{C})=1$ because the outcome of \textsf{SP} is not affected by changing opportunity costs less than $p^*$. In addition, $C^s<p^*$ together with $A^{s,\textsf{SP}}(\mathbf{C})=1$ imply that $A^{s,\textsf{HYB}}(\mathbf{C})=1$ because $X^{\textsf{SP}}(\mathbf{C})= X^{\textsf{SP}}(\mathbf{C}^{-s},{c}^s)>D$. It then results in $P^{s,\textsf{HYB}}(\mathbf{C})\ge p^*$. Otherwise if $C^s \ge p^*$, then we have 
	\[
	P^{s,\textsf{HYB}}(\mathbf{C}^{-s},{c}^s) - {C}^s A^{s,\textsf{HYB}}(\mathbf{C}^{-s},{c}^s) = p^*-C^s\le 0.
	\]
	It then follows that
	\[ 
	P^{s,\textsf{HYB}}(\mathbf{C}^{-s},{c}^s) - {C}^s A^{s,\textsf{HYB}}(\mathbf{C}^{-s},{c}^s)\le P^{s,\textsf{HYB}}(\mathbf{C}) - {C}^s A^{s,\textsf{HYB}}(\mathbf{C}).
	\]

	\textbf{\underline{Case 3}.} If $A^{s,\textsf{HYB}}(\mathbf{C}^{-s},{c}^s)=1$ and $X^{\textsf{SP}}(\mathbf{C}^{-s},{c}^s)\le D$, then in this case, we have $P^{s,\textsf{HYB}}(\mathbf{C}^{-s},{c}^s)=P^{s,\textsf{AUC}}(\mathbf{C}^{-s},{c}^s)$ and $A^{s,\textsf{AUC}}(\mathbf{C}^{-s},{c}^s)=A^{s,\textsf{HYB}}(\mathbf{C}^{-s},{c}^s)=1$. If $C^s< P^{s,\textsf{AUC}}(\mathbf{C}^{-s},{c}^s)$, then $P^{s,\textsf{AUC}}(\mathbf{C}^{-s},{c}^s)=P^{s,\textsf{AUC}}(\mathbf{C})$ because the payment of \textsf{AUC} is not affected by changing an opportunity cost less than the payment. Otherwise if $C^s \ge P^{s,\textsf{AUC}}(\mathbf{C}^{-s},{c}^s)$, then we have 
	\[P^{s,\textsf{HYB}}(\mathbf{C}^{-s},{c}^s) - {C}^s A^{s,\textsf{HYB}}(\mathbf{C}^{-s},{c}^s)\le 0.\]
	It then follows that 
	\[ 
	P^{s,\textsf{HYB}}(\mathbf{C}^{-s},{c}^s) - {C}^s A^{s,\textsf{HYB}}(\mathbf{C}^{-s},{c}^s)\le P^{s,\textsf{HYB}}(\mathbf{C}) - {C}^s A^{s,\textsf{HYB}}(\mathbf{C}).
	\]

	By combining the above three cases, we conclude that the \textsf{HYB} mechanism is IC and IR.
	\Halmos
\end{proof}

To prove Theorem \ref{theorem: asymptotically optimal HYB auction}, we establish another lemma below.
Consider the platform's one-period objective function Eq~\eqref{Ex profit under policy pi} for a given state:
\begin{equation*}
	\Ex \left[\sum_{(i, j)\in \mathcal{E}} r^*_{ij}d_{ij}(r^*_{ij}) - b_{ij}(D_{ijt}-Y_{ijt})- P_{ijt} \;\middle|\; \mathbf{S}_t,\mathbf{D}_t \right].
\end{equation*}
For notation convenience, we define $\rho_{ijt}$ and $\mathcal{H}_{ijt}$ for each lane $(i,j)\in\mathcal{E}$ as follows:
\begin{align}
	\rho_{ijt} &:= r^*_{ij}d_{ij}(r^*_{ij}) - b_{ij}D_{ijt}, \label{def:rho_it}\\
	\mathcal{H}_{ijt} &:= b_{ij}Y_{ijt} - P_{ijt}. \label{def: H_it}
\end{align}
We define a single-period single-lane mechanism design problem under a given system steady state $(\mathbf{S}_t,\mathbf{D}_t)$ as follows:
\begin{subequations}\label{single-period problem}
	\begin{align}
		\max_{Y, P}\;& \sum_{(i,j)\in\mathcal{E}}{\Ex}\left[\rho_{ijt}+ \mathcal{H}_{ijt}\middle | \mathbf{S}_t,\mathbf{D}_t\right]\label{single-period problem: obj function}
		\\
		\mbox{s.t. } \;
		&\; Y_{ijt} \leq \tilde{S}_{ijt}, \label{single-period problem: state constraint}\\
		&\;Y_{ijt} \leq D_{ijt}, \label{single-period problem: demand constraint} \\
		& \;Y_{ijt} \ge 0. \nonumber 
	\end{align} \label{formulation: a single-period problem}
\end{subequations}
In the above formulation, constraint \eqref{single-period problem: state constraint} follows from Eq~\eqref{eq: sum Y and V = S} \rr{and $\sum_j \tilde{S}_{ijt}=S_{it}$}, and constraint \eqref{single-period problem: demand constraint} follows from the definition of $Y_{ijt}$. 
We show that \textsf{AUC} achieves the optimal solution to this problem.

\begin{lemma} \label{lemma: auc and decomposed problem}
	The long-run expected profit of any IC and IR mechanism is given by
	\[
	\gamma = \sum_{(i,j)\in\mathcal{E}}{\Ex}\left[{\Ex}\left[\rho_{ijt}+ \mathcal{H}_{ijt}\middle | \mathbf{S}_t,\mathbf{D}_t\right]\right],
	\]
	where the outer expectation is taken over the steady state distribution of $(\mathbf{S}_t,\mathbf{D}_t)$ under the given mechanism. 
	If $\psi_{ij}(p^*_{ij})\leq b_{ij}$, then \textsf{AUC} is an optimal solution to single-period problem \eqref{single-period problem} conditional on any given state $(\mathbf{S}_t,\mathbf{D}_t)$ \rr{and given $\tilde{S}_{ijt}$}.
\end{lemma}

\begin{proof}{Proof of Lemma \ref{lemma: auc and decomposed problem}.}
	By the law of total expectation, the long-run average profit of any given mechanism is equal to
	\begin{align*}
		\gamma =& \sum_{(i, j)\in \mathcal{E}} \Ex\left[r^*_{ij}d_{ij}(r^*_{ij}) - b_{ij}(D_{ijt}-Y_{ijt})- P_{ijt} \right] \\
		=& \sum_{(i, j)\in \mathcal{E}} \Ex\left[\Ex \left[ r^*_{ij}d_{ij}(r^*_{ij}) - b_{ij}(D_{ijt}-Y_{ijt})- P_{ijt} \;\middle|\; \mathbf{S}_t,\mathbf{D}_t \right]\right],
	\end{align*}
	where the outer expectation is taken over the steady state distribution of $(\mathbf{S}_t,\mathbf{D}_t)$ under the given mechanism. 
	%Moreover, by Lemma~\ref{lemma: necessary conditions for any stationary mechanism}, we have
	% \[
	%   \Ex\left[ Z_{ijt}+\Lambda_{ijt+1} - S_{ijt+1} \right] = 0, \quad \forall (i,j)\in \mathcal{E}.
	% \]
	Following the definition of $\rho_{ijt}$ and $\mathcal{H}_{ijt}$, the first part of the lemma is proved.
	
	Next, we prove the second part of the lemma regarding \textsf{AUC}.
	Notice that given a system state $(\mathbf{S}_t,\mathbf{D}_t)$, the term $\Ex\left[ \rho_{ijt}\middle | \mathbf{S}_t,\mathbf{D}_t\right]$ in the objective function \eqref{single-period problem: obj function} is a constant. Therefore, it suffices to focus on the term $\Ex\left[\mathcal{H}_{ijt}\middle | \mathbf{S}_t,\mathbf{D}_t\right]$. For notation simplicity, let $\hat{\Ex}\left[\cdot\right]:=\Ex\left[\cdot\middle | \mathbf{S}_t,\mathbf{D}_t\right]$.
	By the definition of $A_{ijt}$ and the assumption of homogeneous carriers, we have 
	\[
	\sum_{s\in\mathcal{\tilde S}_{ijt}}\hat{\Ex} \left[ A^s_{ijt}(\mathbf{C}_{ijt} )\right] = \sum_{s\in\mathcal{\tilde S}_{ijt}}\hat{\Ex} \left[ a_{ijt}(C^s_{ijt} )\right]= \hat{\Ex}[Y_{ijt}].
	\]
	Because \textsf{AUC} is IC (Lemma~\ref{lemma: VCG is IC and IR}),
	by the definition of $P_{ijt}$ and Proposition \ref{prop:p = a and psi}, we have
	\[
	\hat{\Ex} [P_{ijt}] = \sum_{s\in\mathcal{\tilde S}_{ijt}}\hat{\Ex}\left[ p_{ijt}(C^s_{ijt}) \right] = \sum_{s\in\mathcal{\tilde S}_{ijt}}\hat{\Ex}\left[ a_{ijt}(C^s_{ijt}) \psi_{ij}(C^s_{ijt})\right].
	\]
	Then, $\hat\Ex [\mathcal{H}_{ijt}]$ can be written as 
	\begin{align}
		\hat{\Ex}\left[\mathcal{H}_{ijt}\right] 
		&= \hat\Ex \left[b_{ij}Y_{ijt}- P_{ijt}\right]\nonumber\\ 
		&= \sum_{s\in\mathcal{\tilde S}_{ijt}}b_{ij}\hat{\Ex} \left[ a_{ijt}(C^s_{ijt} )\right] - \sum_{s\in\mathcal{\tilde S}_{ijt}}\hat{\Ex}\left[ a_{ijt}(C^s_{ijt}) \psi_{ij}(C^s_{ijt})\right]\nonumber\\
		&=  \sum_{s\in\mathcal{\tilde S}_{ijt}}\hat{\Ex} \left[ (b_{ij}-\psi_{ij}(C^s_{ijt}))a_{ijt}(C^s_{ijt} )\right]\nonumber\\
		&= \sum_{s\in\mathcal{\tilde S}_{ijt}}\hat{\Ex} \left[ (b_{ij}-\psi_{ij}(C^s_{ijt}))A^s_{ijt}(\mathbf{C}_{ijt})\right]. \label{eq: H_it rewritten}
	\end{align}
	
	\noindent By Eq~\eqref{eq: H_it rewritten} and $Y_{ijt} = \sum_{s\in \mathcal{\tilde S}_{ijt}} A^s_{ijt}(\mathbf{C}_{ijt})$, it is easy to see that the single-period problem \eqref{single-period problem} can be reformulated as follows:
	\begin{subequations}
		\begin{align*}
			\max &\quad \hat{\Ex} \left[\rho_{ijt}+ \sum_{s\in\mathcal{\tilde S}_{ijt}}(b_{ij}-\psi_{ij}(C^s_{ijt}))A^s_{ijt}(\mathbf{C}_{ijt})\right]\\
			\mbox{s.t. } &\quad 0\le {A}_{ijt}^s(\mathbf{C}_{ijt}) \le 1,\qquad \forall s\in \mathcal{\tilde S}_{ijt},\\
			&\quad \sum_{s\in\mathcal{\tilde S}_{ijt}} A^s_{ijt}(\mathbf{C}_{ijt})\le D_{ijt}.
		\end{align*}
	\end{subequations}
	Denote $\mathcal{K}_{ijt}(\mathbf{C}_{ijt})$ as the optimal value of the following problem:
	\begin{align}
		\mathcal{K}_{ijt}(\mathbf{C}_{ijt}):= \max & \sum_{s\in \mathcal{\tilde S}_{ijt}} (b_{ij}-\psi_{ij}(\mathbf{C}_{ijt}^s)) A^s_{ijt}(\mathbf{C}_{ijt}) \label{eq: K_it objective} \\
		\mbox{s.t. } & 0\le {A}_{ijt}^s(\mathbf{C}_{ijt}) \le 1,\qquad \forall s\in \mathcal{\tilde S}_{ijt}, \nonumber \\
		&\sum_{s\in \mathcal{\tilde S}_{ijt}} A^s_{ijt}(\mathbf{C}_{ijt}) \le D_{ijt}. \nonumber 
	\end{align}
	Let $A^{s*}_{ijt}(\mathbf{C}_{ijt})$ be the optimal solution to the above allocation problem $\mathcal{K}_{ijt}(\mathbf{C}_{ijt})$. Recall that $\xi^*_{ij} = \psi^{-1}_{ij}(b_{ij})$ by Eq \eqref{eq: xi^* def in psi} when $\psi_{ij}(p^*_{ij})\leq b_{ij}$. By the monotonicity of $\psi_{ij}$, the optimal solution $A^{s*}_{ijt}(\mathbf{C}_{ijt})$ is to allocate loads to carriers whose opportunity costs are less than $\xi^*_{ij}$ in the increasing order of their opportunity costs. This is exactly the same allocation by \textsf{AUC}. Therefore, the allocation rule of \textsf{AUC} gives the optimal solution to problem $\mathcal{K}_{ijt}(\mathbf{C}_{ijt})$, 
	% Then, we have
	% \begin{align*}
		% \hat{\Ex}[\mathcal{H}^{\textsf{AUC}}_{ijt}]&=\hat{\Ex} \left[ \sum_{s\in\mathcal{\tilde S}_{ijt}}(b_{ij}+\sum_{k\in \delta^+(j)}\mu^*_{jk}q_{jk}-\psi_{ij}(C^s_{ijt}))A^{s*}_{ijt}(\mathbf{C}_{ijt})\right]\\
		% &\ge \hat{\Ex} \left[ \sum_{s\in\mathcal{\tilde S}_{ijt}}(b_{ij}+\sum_{k\in \delta^+(j)}\mu^*_{jk}q_{jk}-\psi_{ij}(C^s_{ijt}))A^s_{ijt}(\mathbf{C}_{ijt})\right]\\
		% &=\hat{\Ex}\left[\mathcal{H}_{ijt}\right].
		% \end{align*}
	% This implies that the \textsf{AUC} mechanism maximizes $\hat{\Ex}\left[\mathcal{H}_{ijt}\right]$. As $\hat{\Ex}[\rho_{ijt}]$ is a constant term, we have
	% \[
	% \sum_{( i,j)\in\mathcal{E}} \hat{\Ex}[\rho_{ijt}+\mathcal{H}^{\textsf{AUC}}_{ijt}]
	% = \mathcal{L}(\mathbf{S}_t,\mathbf{D}_t),
	% \]
	which completes the proof. \Halmos \end{proof}

\vspace{0.1in}

\begin{proof}{Proof of Theorem \ref{theorem: asymptotically optimal HYB auction}.}
	\rr{We first show} 
	\[
	{\Ex} \left[\rho_{ijt}^{\textsf{SP}}+\mathcal{H}_{ijt}^{\textsf{SP}}\right] \le {\Ex} \left[\rho_{ijt}^{\textsf{HYB}}+\mathcal{H}^{\textsf{HYB}}_{ijt}\right],\quad \forall (i,j)\in\mathcal{E},
	\] 
	where the terms $\rho_{ijt}$ and $\mathcal{H}_{ijt}$ under a given mechanism are defined in Eq~\eqref{def:rho_it} and Eq~\eqref{def: H_it}, respectively. By Lemma \ref{lemma: auc and decomposed problem}, we have
	\begin{align*}
		&\sum_{( i,j)\in\mathcal{E}}\Ex\left[ \rho^{\textsf{SP}}_{ijt}+\mathcal{H}^{\textsf{SP}}_{ijt} \right] = \sum_{(i,j)\in\mathcal{E}}\Ex\left[ r^*_{ij}d_{ij}(r^*_{ij})- b_{ij}(D_{ijt}-Y^{\textsf{SP}}_{ijt}) -P^{\textsf{SP}}_{ijt} \right] = \gamma^{\textsf{SP}}, \\
		&\sum_{( i,j)\in\mathcal{E}}\Ex \left[\rho^{\textsf{HYB}}_{ijt}+\mathcal{H}^{\textsf{HYB}}_{ijt} \right]= \sum_{(i,j)\in\mathcal{E}}\Ex \left[r^*_{ij}d_{ij}(r^*_{ij}) - b_{ij}(D_{ijt}-Y^{\textsf{HYB}}_{ijt}) -P^{\textsf{HYB}}_{ijt} \right] = \gamma^{\textsf{HYB}},
	\end{align*}
	so if we can prove the above inequality, it immediately implies Eq~\eqref{ineq Thm 4: HYB asymptotic optimal}.

	\rr{We make the comparison by coupling methods. We first show that there exists a coupling $(\widetilde{S}^{\textsf{SP}}_{ijt},\widetilde{S}^{\textsf{HYB}}_{ijt})$ that satisfies $\widetilde{S}^{\textsf{SP}}_{ijt}\le \widetilde{S}^{\textsf{HYB}}_{ijt}$. Consider a coupling ($\widetilde{S}^{\textsf{SP}}_{ijt},\widetilde{S}^{\textsf{HYB}}_{ijt}$) for all $(i,j)\in\mathcal{E}$ and for all $t$ on the same probability space. We will show that $\widetilde{S}^{\textsf{SP}}_{ijt}\le \widetilde{S}^{\textsf{HYB}}_{ijt}$ for all $( i,j)\in\mathcal{E}$ and for all $t$ by induction.} 
	
	\rr{\textbf{\underline{Base Case}.} When $t=1$, it is clear that $\widetilde{S}^{\textsf{SP}}_{ijt}= \widetilde{S}^{\textsf{HYB}}_{ijt}=0$ for all $( i,j)\in\mathcal{E}$ as we assume there are no carriers in the marketplace at the beginning of period 1.}
	
	\rr{\textbf{\underline{Induction Step}.} Suppose $\widetilde{S}^{\textsf{SP}}_{ijt}\le \widetilde{S}^{\textsf{AUC}}_{ijt}$ for all $( i,j)\in\mathcal{E}$. Set $\mathbf{D}^{\textsf{SP}}_{t}=\mathbf{D}^{\textsf{HYB}}_{t}$ and $\mathbf{C}_{ijt}^{\textsf{HYB}}=(\mathbf{C}_{ijt}^{\textsf{SP}},\mathbf{C}_{ijt}^{\textsf{DIFF}})$, where $\mathbf{C}_{ijt}^{\textsf{DIFF}}$ is the opportunity cost vector for those $(\widetilde{S}^{\textsf{HYB}}_{ijt}-\widetilde{S}^{\textsf{SP}}_{ijt})$ carriers in the marketplace under \textsf{HYB} but not in the system under \textsf{SP}. Recall that ${X}^{\textsf{HYB}}_{ijt}$ is the number of carriers whose opportunity costs are less than $\xi^*_{ij}$, where $\xi^*_{ij}\ge p^*_{ij}$. Then, we have ${X}^{\textsf{SP}}_{ijt}\le {X}^{\textsf{HYB}}_{ijt}$. It then follows that }
	\[
	{Y}^{\textsf{SP}}_{ijt} = \min\{{X}^{\textsf{SP}}_{ijt},{D}^{\textsf{SP}}_{ijt}\} \le \min\{{X}^{\textsf{HYB}}_{ijt},{D}^{\textsf{HYB}}_{ijt}\} = {Y}^{\textsf{HYB}}_{ijt}.
	\] 
	\rr{We set ${Z}^{\textsf{HYB}}_{ijt}={Z}^{\textsf{SP}}_{ijt}+{Z}^{\textsf{DIFF}}_{ijt}$ and $\Lambda ^{\textsf{SP}}_{it+1}=\Lambda^{\textsf{HYB}}_{it+1}$, where ${Z}^{\textsf{DIFF}}_{ijt}$ follows the Binomial distribution with parameters $({Y}^{\textsf{HYB}}_{ijt} - {Y}^{\textsf{SP}}_{ijt}, {q}_{ij})$.} Then, let    \rr{$${S}^{\textsf{HYB}}_{it+1}=\sum_{k\in\delta^-(i)}{Z}^{\textsf{SP}}_{kit}+{\Lambda} ^{\textsf{SP}}_{it+1}+\sum_{k\in\delta^-(i)}{Z}^{\textsf{DIFF}}_{kit}={S} ^{\textsf{SP}}_{it+1}+\sum_{k\in\delta^-(i)}{Z}^{\textsf{DIFF}}_{kit}={S} ^{\textsf{SP}}_{it+1}+S^{\textsf{DIFF}}_{it}, \forall i\in \mathcal{N},$$}
    \rr{where $S^{\textsf{DIFF}}_{it}=\sum_{k\in\delta^-(i)}{Z}^{\textsf{DIFF}}_{kit}\geq 0$. For any $i\in \mathcal{N}$, set $\mathbf{C}_{ijt+1}^{\textsf{HYB}}=(\mathbf{C}_{ijt+1}^{\textsf{SP}},\mathbf{C}_{ijt+1}^{\textsf{DIFF}})$ for all $j\in\delta^+(i)$, where $\mathbf{C}_{ijt+1}^{\textsf{DIFF}}$ is the opportunity cost vector for $S^{\textsf{DIFF}}_{it}$. Since carriers  choose a lane according to the same MNL choice model based on the same set of post prices under \textsf{HYB} and \textsf{SP}, a carrier who chooses lane $(i,j)$ under \textsf{SP} will still choose the same lane under \textsf{HYB} as long as there are remaining loads on this lane. Along with $S_{it}^{\textsf{HYB}}\geq S_{it}^{\textsf{SP}}$, we have $\widetilde{S}^{\textsf{HYB}}_{ijt+1}\ge \widetilde{S} ^{\textsf{SP}}_{ijt+1}$ for all $(i,j)\in \mathcal{E}$.} This completes the proof for the induction step.
    
    \rr{Now we} compare $\mathcal{H}^{\textsf{SP}}_{ijt}$ and $\mathcal{H}^{\textsf{HYB}}_{ijt}$. Consider a coupling $(\widetilde{\mathcal{H}}^{\textsf{SP}}_{ijt},\widetilde{\mathcal{H}}^{\textsf{HYB}}_{ijt})$ in the same probability space, where $\widetilde{\mathcal{H}}^{\textsf{SP}}_{ijt}\sim {\mathcal{H}}^{\textsf{SP}}_{ijt}$ and $\widetilde{\mathcal{H}}^{\textsf{HYB}}_{ijt}\sim {\mathcal{H}}^{\textsf{HYB}}_{ijt}$. \rr{We have shown that there exists a coupling $(\widetilde{S}^{\textsf{SP}}_{ijt},\widetilde{S}^{\textsf{HYB}}_{ijt})$ that satisfies $\widetilde{S}^{\textsf{SP}}_{ijt}\le \widetilde{S}^{\textsf{HYB}}_{ijt}$.} In this coupling, we set $\mathbf{D}^{\textsf{SP}}_{t}=\mathbf{D}^{{\textsf{HYB}}}_{t}$ and $\mathbf{C}_{ijt}^{\textsf{HYB}}=(\mathbf{C}_{ijt}^{\textsf{SP}},\mathbf{C}_{ijt}^{\textsf{DIFF}})$, where $\mathbf{C}_{ijt}^{\textsf{DIFF}}$ is the opportunity cost vector for those $(\widetilde{S}^{\textsf{HYB}}_{ijt}-\widetilde{S}^{\textsf{SP}}_{ijt})$ carriers in the marketplace under mechanism \textsf{HYB} but not in the marketplace under mechanism \textsf{SP}. If $X^{\textsf{SP}}_{ijt}(\mathbf{C}^{\textsf{HYB}}_{ijt}) > D^{\textsf{HYB}}_{ijt}$, then the \textsf{HYB} operates the same as the \textsf{SP} mechanism. In this case, we have $\widetilde{\mathcal{H}}^{\textsf{SP}}_{ijt}\le \widetilde{\mathcal{H}}^{\textsf{HYB}}_{ijt}$ because $\sum_{s\in\mathcal{S}^{\textsf{SP}}_{ijt}}A^{s,\textsf{SP}}_{ijt}(\mathbf{C}^{\textsf{SP}}_{ijt}) \le \sum_{s\in\mathcal{S}^{\textsf{HYB}}_{ijt}}A^{s,\textsf{HYB}}_{ijt}(\mathbf{C}^{\textsf{HYB}}_{ijt})=D^{\textsf{SP}}_{ijt}=D^{\textsf{HYB}}_{ijt}$ and $p^*_{ij}\le b_{ij}$. Otherwise, the \textsf{HYB} operates the same as the \textsf{AUC} mechanism. By the definition of problem \eqref{eq: K_it objective}, we can construct a feasible solution to $\mathcal{K}_{ijt}(\mathbf{C}^{\textsf{HYB}}_{ijt})$ from the optimal solution to $\mathcal{K}_{ijt}(\mathbf{C}^{\textsf{SP}}_{ijt})$ by adding additional $(\widetilde{S}^{\textsf{HYB}}_{ijt}-\widetilde{S}^{\textsf{SP}}_{ijt})$ variables and restricting the value of these added variables equal to zero. In addition, the objective value under this feasible solution to $\mathcal{K}_{ijt}(\mathbf{C}^{\textsf{HYB}}_{ijt})$ is the same as the optimal objective value of $\mathcal{K}_{ijt}(\mathbf{C}^{\textsf{SP}}_{ijt})$. As a result, the optimal objective value of $\mathcal{K}_{ijt}(\mathbf{C}^{\textsf{SP}}_{ijt})$ is always less than or equal to that of $\mathcal{K}_{ijt}(\mathbf{C}^{\textsf{HYB}}_{ijt})$ under this coupling. It then follows that $\hat{\Ex}[\widetilde{\mathcal{H}}^{\textsf{SP}}_{ijt}]\le \hat{\Ex}[\widetilde{\mathcal{H}}^{\textsf{HYB}}_{ijt}]$.
	Since $\widetilde{\mathcal{H}}^{\textsf{SP}}_{ijt}\sim {\mathcal{H}}^{\textsf{SP}}_{ijt}$ and $\widetilde{\mathcal{H}}^{\textsf{HYB}}_{ijt}\sim {\mathcal{H}}^{\textsf{HYB}}_{ijt}$, it then follows that $\Ex[\mathcal{H}^{\textsf{SP}}_{ijt}]\le \Ex[\mathcal{H}^{\textsf{HYB}}_{ijt}]$ in both cases.
	
	Because the expectation of $\rho_{ijt}$ is not affected by carrier side mechanisms, we have $\Ex[\rho^{\textsf{SP}}_{ijt}] = \Ex[\rho^{\textsf{HYB}}_{ijt}]$.
	By Lemma~\ref{lemma: auc and decomposed problem}, it then follows that 
	\[
	\gamma^{\textsf{SP}} = \sum_{( i,j)\in\mathcal{E}}{\Ex} \left[\rho^{\textsf{SP}}_{ijt}+\mathcal{H}_{ijt}^{\textsf{SP}}\right] \leq \sum_{( i,j)\in\mathcal{E}}{\Ex} \left[\rho^{\textsf{HYB}}_{ijt}+\mathcal{H}^{\textsf{HYB}}_{ijt}\right] = \gamma^{\textsf{HYB}},
	\]
	and this completes the proof.
	\Halmos \end{proof}

	\rr{Next we prove Theorem \ref{theorem: gap between posted price and auction}. To prove the bounds between \textsf{HYB} and \textsf{SP}, we need to analyze the asymptotic property of the number of carriers who haul a load under \textsf{HYB} but would not haul a load under \textsf{SP}. To that end, we first show the following two auxiliary results.}
	\rr{
	\begin{lemma}\label{lem: tail distribution approximation}
		Suppose $U_1,U_2$, and $U_3$ are three independent random variables following Poisson distributions with rate $\theta,\theta$, and $\theta/2$, respectively. Then for any $x>0$, we have
		\begin{align*}
			\left|\Pr\{U_1-U_2\geq x \} -\bar{\Phi}\left(\frac{x}{\sqrt{2\theta}}\right)\right|& \leq  \frac{c_0}{\sqrt{2\theta}} ,\\
			\left|\Pr\{U_3\geq x +\theta/2\}-\bar{\Phi}\left(\frac{\sqrt{2}x}{\sqrt{\theta}}\right) \right|& \leq  \frac{c_0\sqrt{2}}{\sqrt{\theta}},
		\end{align*}
		where $\bar{\Phi}(z)=1-\Phi(z)$ and $\Phi(z)$ is the cumulative distribution function of a standard normal distribution $N(0,1)$, and $c_0$ is a constant.
	\end{lemma}
    }
\begin{proof}{Proof of Lemma \ref{lem: tail distribution approximation}.}
		\rr{
			Let $U_1^{(i)},i\in\{1,2,\cdots,N\}$ be $N$ i.i.d random variables following Poisson distributions with rate $\theta/N$. The additive property of Poisson distribution implies that $U_1=\sum_{i=1}^{N}U_1^{(i)}$. Similarly, we can decompose $U_2$ and $U_3$ by $U_2=\sum_{i=1}^{N}U_2^{(i)}$, $U_3=\sum_{i=1}^{N}U_3^{(i)}$, where $U_2^{(i)}$ and $U_3^{(i)}$ are independent and $U_2^{(i)}\sim \text{Poisson}(\theta/N)$ and $U_3^{(i)}\sim \text{Poisson}(\theta/(2N))$. Therefore, we have
		\begin{align}
			\label{eq:poissonDif}\Pr\{U_1-U_2\geq x \}=&\Pr\left\{\sum_{i=1}^{N}(U_1^{(i)}-U_2^{(i)})\geq x \right\},\\
			\label{eq:poissonest}\Pr\{U_3\geq x +\theta/2\}=&\Pr\left\{\sum_{i=1}^{N}\left[U_3^{(i)}-\frac{\theta}{2N}\right]\geq x  \right\}.
		\end{align}
		We use the Berry-Ess\'een inequality (see \cite{Berry_1941}, \cite{ChenEtal2011}) to facilitate our analysis. For $N$ i.i.d. random variables $W_1,W_2,\cdots, W_N$, if $\Ex[W_i]=0$ and $\Ex[|W_i|^3]<\infty$, then by the Berry-Ess\'een inequality, we have 
		\begin{equation}\label{eq: berry-esseen}
			\left|\Pr\left\{\sum_{i=1}^N W_i\geq z\cdot\sqrt{N\Ex[W_i^2]}\right\}-\bar{\Phi}(z)\right|\leq c_0 \frac{\Ex[|W_i|^3]}{\sqrt{N\Ex[W_i^2]^{3}}}, \quad \forall z \mbox{ and } N
		\end{equation}
		where $c_0$ is a constant.	
	}
	
	\rr{By the property of Poisson distribution's second and third moments, it is easy to verify that
		\begin{align*}
			\Ex \left[U_1^{(i)}-U_2^{(i)} \right]&=0,\quad \Ex\left[ \left(U_1^{(i)}-U_2^{(i)} \right)^2\right]=\frac{2\theta}{N},\\
			\Ex\left[|U_1^{(i)}-U_2^{(i)}|^3 \right]&\leq\Ex \left[ \left(U_1^{(i)}+U_2^{(i)} \right)^3\right]=\left(\frac{2\theta}{N}\right)^3+3\left(\frac{2\theta}{N}\right)^2 +\frac{2\theta}{N}.
		\end{align*}
		Then by letting $W_i=U_1^{(i)}-U_2^{(i)}$ and $z=x/\sqrt{2\theta}$ in Eq~\eqref{eq: berry-esseen}, we have
		\begin{align*}
			&\left|\Pr \left\{\sum_{i=1}^N (U_1^{(i)}-U_2^{(i)})\geq x\right\}-\bar{\Phi}\left(\frac{x}{\sqrt{2\theta}}\right)\right|\leq c_0 \frac{\Ex[|U_1^{(i)}-U_2^{(i)}|^3]}{\sqrt{N(2\theta/N)^{3}}}\\
			&\quad\leq c_0 \frac{\left(\frac{2\theta}{N}\right)^3+3\left(\frac{2\theta}{N}\right)^2 +\frac{2\theta}{N}}{\sqrt{N(2\theta/N)^{3}}}=c_0 \frac{\frac{(2\theta)^2}{N^2}+\frac{6\theta}{N} +1}{\sqrt{2\theta}}
		\end{align*}
		for all $x>0$ and $N\in\mathbb{N^+}$. Taking the limit at $N\rightarrow\infty$ and combined with Eq~\eqref{eq:poissonDif}, we have
		$$\left|\Pr\{U_1-U_2\geq x \} -\bar{\Phi}\left(\frac{x}{\sqrt{2\theta}}\right)\right| \leq  \frac{c_0}{\sqrt{2\theta}},$$
		which proves the first inequality in Lemma \ref{lem: tail distribution approximation}.
	}
	
	\rr{
		Similarly, it's easy to verify that
		\begin{align*}
			\Ex[U_3^{(i)}-\theta/(2N)]&=0,\quad \Ex[(U_3^{(i)}-\theta/(2N))^2]=\frac{\theta}{2N},\\
			\Ex[|U_3^{(i)}-\theta/(2N)|^3]&\leq\Ex[(U_3^{(i)}+\theta/(2N))^3]=8\left(\frac{\theta}{2N}\right)^3+6\left(\frac{\theta}{2N}\right)^2 +\frac{\theta}{2N}.
		\end{align*}
		Then let $W_i=U_3^{(i)}-\theta/(2N)$ and $z=\sqrt{2}x/\sqrt{\theta}$ in Eq~\eqref{eq: berry-esseen}, we have
		\begin{align*}
			&\left|\Pr \left\{\sum_{i=1}^N \left[U_3^{(i)}-\frac{\theta}{2N}\right]\geq x\right\}-\bar{\Phi}\left(\frac{\sqrt{2}x}{\sqrt{\theta}}\right)\right|\leq c_0 \frac{\Ex[|U_3^{(i)}-\theta/(2N)|^3]}{\sqrt{N(\theta/(2N))^{3}}}\\
			&\quad\leq c_0 \frac{8\left(\frac{\theta}{2N}\right)^3+6\left(\frac{\theta}{2N}\right)^2 +\frac{\theta}{2N}}{\sqrt{N(\theta/(2N))^{3}}}=c_0 \frac{\frac{2\theta^2}{N^2}+\frac{3\theta}{N} +1}{\sqrt{\theta/2}}
		\end{align*}
		holds for all $x>0$ and $N\in\mathbb{N^+}$. Taking the limit at $N\rightarrow\infty$ and combined with Eq~\eqref{eq:poissonest}, we have
		$$\left|\Pr\{U_3\geq x +\theta/2\}-\bar{\Phi}\left(\frac{\sqrt{2}x}{\sqrt{\theta}}\right) \right| \leq  \frac{c_0\sqrt{2}}{\sqrt{\theta}},$$
		which proves the second inequality in Lemma \ref{lem: tail distribution approximation}.
	}
	\Halmos \end{proof}
\rr{
	\begin{lemma}\label{lem: expected num of AUC carrier}
		Suppose $U_1,U_2$, and $U_3$ are three independent random variables following Poisson distributions with rate $\theta,\theta$, and $\theta/2$, respectively. We have
		\begin{align*}
			\Ex[\min\{[U_1-U_2]^+,U_3\}] \geq \Omega(\sqrt{\theta}),
		\end{align*}
		where $[x]^+$ equals $x$ if $x>0$ and equals $0$ otherwise.
	\end{lemma}
	\begin{proof}{Proof of Lemma \ref{lem: expected num of AUC carrier}.}
		Let $\lceil x \rceil$ denote the smallest integer such that $\lceil x \rceil \geq x$. By the definition of expectation,
		\begin{align*}
			\Ex&[\min\{[U_1-U_2]^+,U_3\}]  =\sum_{i=1}^{\infty} i \Pr\{\min\{[U_1-U_2]^+,U_3\}=i\}\\
			 =& \sum_{i=1}^{\left\lceil\sqrt{\theta}\right\rceil-1} i \Pr\{\min\{[U_1-U_2]^+,U_3\}=i\}+ \sum_{i=\left\lceil\sqrt{\theta}\right\rceil}^{\infty} i \Pr\{\min\{[U_1-U_2]^+,U_3\}=i\}\\
			\geq & \sum_{i=\left\lceil\sqrt{\theta}\right\rceil}^{\infty} i \Pr\{\min\{[U_1-U_2]^+,U_3\}=i\}
			\geq  \sum_{i=\left\lceil\sqrt{\theta}\right\rceil}^{\infty} \left\lceil\sqrt{\theta}\right\rceil \Pr\{\min\{[U_1-U_2]^+,U_3\}=i\}\\
			= & \left\lceil\sqrt{\theta}\right\rceil \Pr\{\min\{[U_1-U_2]^+,U_3\}\geq \left\lceil\sqrt{\theta}\right\rceil \}
			 = \left\lceil\sqrt{\theta}\right\rceil \Pr\{\min\{[U_1-U_2]^+,U_3\}\geq \sqrt{\theta} \}\\
			 \geq & \sqrt{\theta}\Pr\{\min\{[U_1-U_2]^+,U_3\}\geq \sqrt{\theta} \}
			 = \sqrt{\theta}\Pr\{U_1-U_2 \geq \sqrt{\theta}\} \Pr\{U_3\geq \sqrt{\theta}\}\\
			 \geq& \sqrt{\theta}\Pr\{U_1-U_2 \geq \sqrt{\theta}\} \Pr\{U_3\geq \theta+\sqrt{\theta}\}.
		\end{align*}
	By letting $x=\sqrt{\theta}$ in Lemma \ref{lem: tail distribution approximation}, we have
	\begin{align*}
		\Pr\{U_1-U_2\geq \sqrt{\theta} \} &=\bar{\Phi}\left({1}/{\sqrt{2}}\right)+O(1/\sqrt{\theta}),\\
		\Pr\{U_3\geq \theta+\sqrt{\theta} \} &=\bar{\Phi}\left({\sqrt{2}}\right)+O(1/\sqrt{\theta}).
	\end{align*}
	It then follows that
	\begin{align*}
		\Ex&[\min\{[U_1-U_2]^+,U_3\}] \geq \sqrt{\theta} \bar{\Phi}\left({1}/{\sqrt{2}}\right) \bar{\Phi}\left({\sqrt{2}}\right) + O(1)=\Omega(\sqrt{\theta})
	\end{align*}
	and this completes the proof.
	\Halmos \end{proof}
}

	\begin{proof}{Proof of Theorem \ref{theorem: gap between posted price and auction}.}
		\rr{
			We prove the theorem using the following problem instance. Suppose every node is connected to $k$ destinations, i.e., $|\delta^+(i)|=k$ for all $i\in\mathcal{N}$. Shipper demand $D_{ijt}$ follows a Poisson distribution with mean $\theta$ under the optimal shipper rates. Carriers arrivals $\Lambda_{it}$ follows i.i.d. Poisson distributions with mean $2k\theta$. The price sensitivity parameter for carriers $\beta=1$, and the average costs for carriers $\alpha_{ij}=\alpha$. The probability that carriers stay in the next period $q_{ij}=0$. Suppose the optimal shipper rates are $r^*_{ij}=1$ and the penalty costs are $b_{ij}=b=\alpha-\ln k +4+\ln3$. Firstly, we solve the optimal posted price in the \textsf{FA} model. Under the aforementioned problem instance, the objective function and constraints of  \textsf{FA} in Eq~(\ref{FA linking constraint},\ref{FA demand constraint}) can be rewritten as follows:
			\begin{align}
				(\textsf{FA}): \quad \max_{\bar{\mathbf y},\bar{\mathbf v}}  & \quad 
				\sum_{(i, j)\in \mathcal{E}}\theta- \left[\ln \left (\frac{\bar{y}_{ij}}{\bar{v}_{i}} \right)+\alpha \right] \bar{y}_{ij} - b(\theta-\bar{y}_{ij})\nonumber
				\\
				\mbox{s.t. } \; & 2k\theta = \sum_{j\in \delta^+(i)}\bar{y}_{ij}+\bar{v}_{i},\quad \forall i\in \mathcal{N},\label{eq: balance_append_cons}\\
				&\;\bar{y}_{ij} \leq \theta, \quad  \forall (i,j)\in \mathcal{E},\nonumber \\
				& \;\bar{y}_{ij} \ge 0, \quad  \forall (i,j)\in \mathcal{E} \mbox{ and } \bar{v}_{i} \ge 0, \quad  \forall i\in \mathcal{N}. \nonumber 
			\end{align}
			For any feasible $\bar{v}_i$, Eq~\eqref{eq: balance_append_cons} implies that $\sum_{j\in \delta^+(i)}\bar{y}_{ij}=2k\theta-\bar{v}_i$. Applying the Jensen's inequality to the objective function, the optimal $\bar{y}_{ij}$ at a given $\bar{v}_i$ is given by
			\begin{align}
				\bar{y}^*_{ij}(\bar{v}_i)=\frac{2k\theta-\bar{v}_i}{k}.\label{eq:optimal y on v}
			\end{align}
			Substituting Eq~\eqref{eq:optimal y on v} into the \textsf{FA} problem, we have
				\begin{align}
				(\textsf{FA}): \quad \max_{\bar{\mathbf v}}  & \quad 
				\sum_{i\in \mathcal{N}}k\theta-kb\theta+ \left[b-\ln \left (\frac{2k\theta-\bar{v}_i}{k\bar{v}_{i}} \right)-\alpha \right] ({2k\theta-\bar{v}_i}) \nonumber
				\\
				\mbox{s.t. } \; & k\theta \leq \bar{v}_{i}\leq 2k\theta,\quad \forall i\in \mathcal{N}.\label{eq:optimalVCons}
			\end{align}
			Notice that Eq~\eqref{eq:optimalVCons} implies $1/\bar{v}_i\leq 1/(k\theta)$. Taking the first order derivative of the objective function with respect to $\bar{v}_i$, we have
			\begin{align*}
				{\alpha}-b+\ln\left(\frac{2\theta}{\bar{v}_i}-\frac{1}{k}\right)+\frac{2k\theta}{\bar{v}_i}&\leq {\alpha}-b+\ln\left(\frac{2\theta}{k\theta}-\frac{1}{k}\right)+\frac{2k\theta}{k\theta}\\
				&={\alpha}-b-\ln k+2\\
				&={\alpha}-(\alpha-\ln k +4+\ln3)-\ln k+2\\
				&=-2-\ln 3 <0,
			\end{align*}
			which implies that the objective function is strictly decreasing in $\bar{v}_i$. Therefore, $\bar{v}^*_i=k\theta$,  $\bar{y}^*_{ij}=\theta$ and
			\begin{equation}
				p^*_{ij}=\ln\left(\frac{\bar{y}^*_{ij}}{\bar{v}^*_i}\right)+\alpha=\alpha-\ln k.\label{eq:optimal post price instance}
			\end{equation}
			We assume $k<e^\alpha$ in this instance so that $p^*_{ij}>0$.
		}
		
		\rr{
			Next we calculate the reserve price $\xi^*_{ij}=\max\{p^*_{ij},\psi^{(-1)}_{ij}(b)\}$. By Eq~\eqref{eq: psi-result} and \eqref{eq:optimal post price instance}, we have
			\begin{equation}
				{\psi}_{ij} (C) = C + 1 + ke^{C-\alpha}.\label{eq:psi function instance}
			\end{equation}
			Notice that 
			\begin{align*}
				b &= \alpha-\ln k+4+\ln 3=\alpha-\ln k+\ln 3+1+3\\
				&= \alpha-\ln k+\ln 3+1+k\cdot\frac{3}{k} \\
                    & =\alpha-\ln k+\ln 3+1+k\cdot e^{\ln{3}-\ln{k}}\\
				&= (\alpha-\ln k+\ln 3)+1+k\cdot e^{(\alpha-\ln{k}+\ln{3})-\alpha}.
			\end{align*}
			Therefore $b=\psi_{ij}(C)$ is equivalent to
			$$(\alpha-\ln k+\ln 3)+1+k\cdot e^{(\alpha-\ln{k}+\ln{3})-\alpha}=C + 1 + ke^{C-\alpha},$$
			and it is easy to see that $C=\psi_{ij}^{-1}(b)=\alpha-\ln k+\ln 3>p^*_{ij}$. Therefore, the reserve price is given by
			\begin{equation}
				\xi^*_{ij} = \alpha-\ln k + \ln 3.\label{eq:reserve price instance}
			\end{equation}
		}
		\rr{
			Now we are ready to estimate the bound between \textsf{HYB} and \textsf{SP}. By the proof of Theorem \ref{theorem: asymptotically optimal HYB auction}, we have
			\begin{align}				\gamma^{\textsf{HYB}}-\gamma^{\textsf{SP}} =  \sum_{( i,j)\in\mathcal{E}}{\Ex} \left[\mathcal{H}^{\textsf{HYB}}_{ijt}\right] -{\Ex} \left[\mathcal{H}_{ijt}^{\textsf{SP}}\right]=\sum_{( i,j)\in\mathcal{E}}{\Ex}\left[\mathcal{K}_{ijt}(\mathbf{C}^{\textsf{SP}}_{ijt},\mathbf{C}^{\textsf{DIFF}}_{ijt})-\mathcal{K}_{ijt}(\mathbf{C}^{\textsf{SP}}_{ijt})\right].\label{eq:hyb-sp tight bound}
			\end{align}
			Notice that the probability that carriers stay in the next period $q_{ij}=0$ in this problem instance, which implies that all carriers are new arrivals. Therefore, the coupling $(\mathbf{C}^{\textsf{SP}}_{ijt},\mathbf{C}^{\textsf{DIFF}}_{ijt})\sim \mathbf{C}^{\textsf{SP}}_{ijt}$ implies that all elements in $\mathbf{C}^{\textsf{DIFF}}_{ijt}$ must exceed the posted price (because all the newly arrived carriers with opportunity costs below the posted price must be in $\mathbf{C}^{\textsf{SP}}_{ijt}$). Let $\mathbf{C}^{\textsf{DIFF}}_{ijt}=\left(\mathbf{C}^{\textsf{BELOW}}_{ijt},\mathbf{C}^{\textsf{ABOVE}}_{ijt}\right)$ where each element in $\mathbf{C}^{\textsf{BELOW}}_{ijt}$ is in $(p_{ij}^*,\xi_{ij}^*]$ and each element in $\mathbf{C}^{\textsf{ABOVE}}_{ijt}$ is larger than $\xi_{ij}^*$. In other words, $\mathbf{C}^{\textsf{BELOW}}_{ijt}$ includes the carriers whose opportunity costs are larger than the posted price but smaller than the reserve price (and hence they will choose to participate in the auction). Let $U^{\textsf{SP}}_{ijt}$, $U^{\textsf{ABOVE}}_{ijt}$ and $U^{\textsf{BELOW}}_{ijt}$ denote the number of elements in the vectors $\mathbf{C}^{\textsf{SP}}_{ijt}$, $\mathbf{C}^{\textsf{ABOVE}}_{ijt}$ and $\mathbf{C}^{\textsf{BELOW}}_{ijt}$, respectively. By the definition of $\mathcal{K}$ in problem \eqref{eq: K_it objective}, we know that
			$$
            \mathcal{K}_{ijt}(\mathbf{C}^{\textsf{SP}}_{ijt},\mathbf{C}^{\textsf{BELOW}}_{ijt},\mathbf{C}^{\textsf{ABOVE}}_{ijt})-\mathcal{K}_{ijt}(\mathbf{C}^{\textsf{SP}}_{ijt})>0
            $$
			if $U^{\textsf{SP}}_{ijt} < D_{ijt}$ and $U^{\textsf{BELOW}}_{ijt} > 0$ (that is, when the number of  carriers whose opportunity costs are less than the posted price is not enough to cover the demand while there exists at least one  carrier who would join \textsf{AUC}), and 
			$$
            \mathcal{K}_{ijt}(\mathbf{C}^{\textsf{SP}}_{ijt},\mathbf{C}^{\textsf{BELOW}}_{ijt},\mathbf{C}^{\textsf{ABOVE}}_{ijt})-\mathcal{K}_{ijt}(\mathbf{C}^{\textsf{SP}}_{ijt})=0
            $$
			otherwise. Moreover, if $U^{\textsf{SP}}_{ijt} < D_{ijt}$ and $U^{\textsf{BELOW}}_{ijt} > 0$, we can construct a feasible solution to $\mathcal{K}_{ijt}(\mathbf{C}^{\textsf{SP}}_{ijt},\mathbf{C}^{\textsf{BELOW}}_{ijt},\mathbf{C}^{\textsf{ABOVE}}_{ijt})$ by adding additional $(U^{\textsf{ABOVE}}_{ijt}+U^{\textsf{BELOW}}_{ijt})$ variables and restricting  $\min\{D_{ijt}-U^{\textsf{SP}}_{ijt},U^{\textsf{BELOW}}_{ijt}\}$ variables in $\mathbf{C}^{\textsf{BELOW}}_{ijt}$ to be equal to 1 and restricting the remaining variables to be equal to 0. Therefore, we have
			$$\mathcal{K}_{ijt}(\mathbf{C}^{\textsf{SP}}_{ijt},\mathbf{C}^{\textsf{BELOW}}_{ijt},\mathbf{C}^{\textsf{ABOVE}}_{ijt})-\mathcal{K}_{ijt}(\mathbf{C}^{\textsf{SP}}_{ijt})\geq\sum_{s\in\mathcal{S}^{\textsf{BELOW}}}(b-\psi_{ij}(C^s_{ijt}))$$
				if $\min\{D_{ijt}-U^{\textsf{SP}}_{ijt},U^{\textsf{BELOW}}_{ijt}\} > 0$, where $\mathcal{S}^{\textsf{BELOW}}$ is a set of $\min\{D_{ijt}-U^{\textsf{SP}}_{ijt},U^{\textsf{BELOW}}_{ijt}\}$ carriers that are randomly chosen from $\mathbf{C}^{\textsf{BELOW}}_{ijt}$. Together with Eq~\eqref{eq:hyb-sp tight bound}, we have
				\begin{align}
					\gamma^{\textsf{HYB}}-\gamma^{\textsf{SP}}&\geq
					\begin{cases}
						\Ex[|\mathcal{S}^{\textsf{ABOVE}}|]\left(b-\Ex[\psi_{ij}(C^s_{ijt})|p^*_{ij}<C^s_{ijt}\leq\xi^*_{ij}]\right),&\text{if }\min\{D_{ijt}-U^{\textsf{SP}}_{ijt},U^{\textsf{BELOW}}_{ijt}\} > 0\\				0,&\text{otherwise}
					\end{cases} \nonumber\\
					&=\Ex\left[\min\{[D_{ijt}-U^{\textsf{SP}}_{ijt}]^+,U^{\textsf{BELOW}}_{ijt}\}\right] \left(b-\Ex[\psi_{ij}(C^s_{ijt})|p^*_{ij}<C^s_{ijt}\leq\xi^*_{ij}]\right)\label{eq:HYB-SP gap decompo}.
				\end{align}
			}
			\rr{It is easy to verify that 
				$$F_{ij}(c)=\frac{ke^{p^*_{ij}-\alpha}}{e^{p^*_{ij}-c}+ke^{p^*_{ij}-\alpha}}.$$ 
				By Eq~\eqref{eq:optimal post price instance} and \eqref{eq:reserve price instance}, we have
				\begin{align*}
					\Pr\{p^*_{ij}<C^s_{ijt}\leq\xi^*_{ij}\}&=F_{ij}(\xi^*_{ij})-F_{ij}(p^*_{ij})\\
					&=\frac{ke^{p^*_{ij}-\alpha}}{e^{p^*_{ij}-\xi^*_{ij}}+ke^{p^*_{ij}-\alpha}}-\frac{ke^{p^*_{ij}-\alpha}}{1+ke^{p^*_{ij}-\alpha}}\\
					&=\frac{1}{1/3+1}-\frac{1}{2}=\frac{1}{4}.
				\end{align*}
				Then by Eq~\eqref{eq:optimal post price instance}, \eqref{eq:psi function instance} and \eqref{eq:reserve price instance}, we can calculate that
				\begin{align}
					\Ex[\psi_{ij}(C^s_{ijt})|p^*_{ij}<C^s_{ijt}\leq\xi^*_{ij}]&=4\int_{P^*_{ij}}^{\xi^*_{ij}}\psi_{ij}(C)f_{ij}(C)dC\nonumber\\
					&=\alpha-\ln k +\ln 27.\label{eq:gap-expectation}
				\end{align}
			}			
			\rr{
				Next we give an estimation of $\min\{[D_{ijt}-U^{\textsf{SP}}_{ijt}]^+,U^{\textsf{BELOW}}_{ijt}\}$. Notice that $\Lambda_{it}$ follows a Poisson distribution with rate $2k\theta$, and $D_{ijt}$ follows a Poisson distribution with rate $\theta$. It is well known that Poisson random variables can be decomposed into summations of independent Poisson random variables. For example, if $U\sim\text{Poisson}(\lambda)$, $U_1\sim\text{Poisson}(p\lambda)$, and $U_2\sim\text{Poisson}((1-p)\lambda)$, and $U_1$ and $U_2$ are independent, we have $U=U_1+U_2$. It is easy to verify that $\Pr\{C_{ijt}^s\leq p^*_{ij}\}=F_{ij}(p^*_{ij})=1/2$, and therefore all \textsf{SP} arrivals $\sum_{j}U^{\textsf{SP}}_{ijt}$ (i.e., the newly arrived carriers whose opportunity costs are below the posted price) follows a Poisson distribution with rate $2k\theta/2=k\theta$. In this instance, each carrier chooses one of the $k$ lanes with identical probability $1/k$, and therefore we have $U_{ijt}^{\textsf{SP}}\sim\text{Poisson}(\theta)$.
			}
			
			\rr{
				Similarly, by $\Pr\{p^*_{ij}<C^s_{ijt}\leq\xi^*_{ij}\}=1/4$ we have $U_{ijt}^{\textsf{BELOW}}\sim \text{Poisson}(\theta/2)$. Along with $D_{ijt}\sim\text{Poisson}(\theta)$, we can apply Lemma \ref{lem: expected num of AUC carrier} by letting $U_1=D_{ijt}$, $U_2=U_{ijt}^{\textsf{SP}}$ and $U_3=U_{ijt}^{\textsf{BELOW}}$. It then follows that
				$$
                \Ex\left[\min\{[D_{ijt}-U^{\textsf{SP}}_{ijt}]^+,U^{\textsf{BELOW}}_{ijt}\}\right]\geq\Omega(\sqrt{\theta}).
                $$
				Together with Eq~\eqref{eq:HYB-SP gap decompo} and \eqref{eq:gap-expectation}, we have
				$$\gamma^{\textsf{HYB}}-\gamma^{\textsf{SP}}\geq\Omega(\sqrt{\theta})$$
				which completes the proof.
			}
		\Halmos \end{proof}

\end{document}